\documentclass[11pt]{amsart}

\usepackage{amsmath}
\usepackage{amssymb}
\usepackage{bbm}
\usepackage{esint} 

\usepackage[colorlinks,citecolor=red,pagebackref,hypertexnames=false]{hyperref}

\newcommand{\R}{{\mathbb R}}       

\newcommand{\Z}{{\mathbb Z}}       
\newcommand{\DD}{{\mathcal D}}

\newcommand{\HH}{{\mathcal H}}

\newcommand{\WW}{{\mathcal W}}

\newcommand{\MM}{{\mathcal M}}

\newcommand{\RR}{{\mathcal R}}

\newcommand{\diam}{{\rm diam}}
\newcommand{\dist}{{\rm dist}}

\newcommand{\fiproof}{{\hspace*{\fill} $\square$ \vspace{2pt}}}

\newcommand{\rf}[1]{{(\ref{#1})}}

\newcommand{\supp}{\operatorname{supp}}

\newcommand{\vphi}{{\varphi}}
\newcommand{\ve}{{\varepsilon}}
\newcommand{\vv}{{\vspace{2mm}}}

\newcommand{\wt}[1]{{\widetilde{#1}}}
\newcommand{\wh}[1]{{\widehat{#1}}}

\newcommand{\vmo}{{\operatorname{VMO}}}
\newcommand{\bmo}{{\operatorname{BMO}}}

\newcommand{\sss}{{\rm Stop}}

\newcommand{\pv}{\operatorname{pv}}

\newcommand{\HD}{{\mathsf{HD}}}
\newcommand{\LD}{{\mathsf{LD}}}
\newcommand{\LLD}{{\mathcal{LD}}}
\newcommand{\HHD}{{\mathcal{HD}}}
\newcommand{\capp}{\operatorname{Cap}}

\def\Xint#1{\mathchoice
{\XXint\displaystyle\textstyle{#1}}%
{\XXint\textstyle\scriptstyle{#1}}%
{\XXint\scriptstyle\scriptscriptstyle{#1}}%
{\XXint\scriptscriptstyle\scriptscriptstyle{#1}}%
\!\int}
\def\XXint#1#2#3{{\setbox0=\hbox{$#1{#2#3}{\int}$ }
\vcenter{\hbox{$#2#3$ }}\kern-.58\wd0}}

\def\avint{\;\Xint-}

\usepackage{pgf,tikz} 

\definecolor{ffffff}{rgb}{1.0,1.0,1.0}
\definecolor{qqqqff}{rgb}{0.0,0.0,1.0}
\definecolor{ffqqqq}{rgb}{1.0,0.0,0.0}
\definecolor{zzzzqq}{rgb}{0.6,0.6,0.0}
\definecolor{marronet}{rgb}{0.6,0.2,0}
\definecolor{negre}{rgb}{0,0,0}
\definecolor{vermell}{rgb}{0.8,0.05,0.05}
\definecolor{blau}{rgb}{0.3,0.2,1.}
\definecolor{blauclar}{rgb}{0.,0.,1.}
\definecolor{grisfosc}{rgb}{0.25098039215686274,0.25098039215686274,0.25098039215686274}
\definecolor{verd}{rgb}{0.1,0.6,0.1}
\definecolor{taronja}{rgb}{0.9,0.6,0.05}
\definecolor{vermellclar}{rgb}{1.,0.,0.}
\definecolor{verdet}{rgb}{0,0.8,0.1}
\definecolor{blauverd}{rgb}{0,0.4,0.2}
\definecolor{grisclar}{rgb}{0.6274509803921569,0.6274509803921569,0.6274509803921569}

\newtheorem{theorem}{Theorem}[section]
\newtheorem{lemma}[theorem]{Lemma}

\newtheorem{coro}[theorem]{Corollary}

\newtheorem{claim}{Claim}

\newtheorem*{theorem*}{Theorem}
\newtheorem*{theorema}{Theorem A}

\theoremstyle{definition}

\theoremstyle{remark}
\newtheorem{rem}[theorem]{\bf Remark}

\numberwithin{equation}{section}

\newcommand{\brem}{\begin{rem}}
\newcommand{\erem}{\end{rem}}

\usepackage{xcolor}

\begin{document}

\title[The two-phase problem for harmonic measure]{The two-phase problem for harmonic measure in VMO}

\author{Mart\'{\i} Prats}

\author{Xavier Tolsa}

\address{Mart\'{\i} Prats
\\
 Departament de Matem\`atiques,
\\
Universitat Aut\`onoma de Barcelona
\\
08193 Bellaterra (Barcelona), Catalonia.
}

\email{mprats@mat.uab.cat} 

\address{Xavier Tolsa
\\
ICREA, Passeig Llu\'{\i}s Companys 23 08010 Barcelona, Catalonia\\
 Departament de Matem\`atiques, and BGSMath
\\
Universitat Aut\`onoma de Barcelona
\\
08193 Bellaterra (Barcelona), Catalonia.
}
\email{xtolsa@mat.uab.cat}

\thanks{Both authors were supported by 2017-SGR-0395 (Catalonia) and MTM-2016-77635-P (MINECO, Spain).
X.T.~was also supported by  MDM-2014-044 (MINECO, Spain).
}

\begin{abstract}
Let $\Omega^+\subset\R^{n+1}$ be an NTA domain and let $\Omega^-= \R^{n+1}\setminus \overline{\Omega^+}$ be an NTA domain as well. Denote by $\omega^+$ and $\omega^-$ their respective harmonic measures. Assume that $\Omega^+$ is a $\delta$-Reifenberg flat domain for some $\delta>0$ small enough.
In this paper we show that
 $\log\frac{d\omega^-}{d\omega^+}\in \vmo(\omega^+)$ if and only if 
 $\Omega^+$ is vanishing Reifenberg flat, {
 $\Omega^+$ and $\Omega^-$ have joint big pieces of chord-arc subdomains,}
 and the  inner unit normal of $\Omega^+$
 has vanishing oscillation with respect to the approximate normal. This result can be considered
 as a two-phase counterpart of a more well known related one-phase problem for harmonic measure solved by Kenig and Toro.
\end{abstract}

\maketitle

\section{Introduction}

In this paper we study a two-phase problem for harmonic measure in $\R^{n+1}$. The study of this type of problems
has been a subject of thorough investigation in the last years. Roughly speaking, given two domains
$\Omega^+,\Omega^- \subset \R^{n+1}$ whose boundaries have non-empty intersection, one wants to relate
 the analytic 
properties of the respective harmonic measures $\omega^+,\omega^-$ in $\partial\Omega^+\cap\partial\Omega^-$ with some geometric properties of $\partial\Omega^+\cap\partial\Omega^-$.
For example, the recent
 works \cite{AMT-cpam}, \cite{AMTV}, which solve a long standing conjecture of Bishop 
 \cite{Bishop-questions}, show that if $\omega^+$ and $\omega^-$ are mutually absolutely continuous in some
subset $E\subset \partial\Omega^+\cap\partial\Omega^-$, then $\omega^+|_E$ and $\omega^-|_E$ are $n$-rectifiable measures, that is, they are concentrated in an $n$-rectifiable subset of $E$ and they are absolutely continuous with respect to the Hausdorff $n$-dimensional measure $\HH^n$.
Recall that a set $F\subset\R^{d}$ is called $n$-rectifiable if there are Lipschitz maps
$f_i:\R^n\to\R^d$, $i=1,2,\ldots$, such that 
\begin{equation}\label{eq001}
\HH^n\Big(F\setminus\textstyle\bigcup_i f_i(\R^n)\Big) = 0.
\end{equation}
 Let us remark that inn a previous work, Kenig, Preiss, and Toro \cite{KPT09} had already shown that, under the 
additional assumption that $\Omega^+$ and $\Omega^-$ are non-tangentially accessible (NTA) domains with $\Omega^-=\R^{n+1}\setminus\overline{\Omega^+}$, the mutual absolute continuity of $\omega^+$ and
$\omega^-$ implies that both harmonic measures are concentrated in a set of Hausdorff dimension $n$, which is $n$-rectifiable
in the particular case that $\partial\Omega^+$ has locally finite Hausdorff $n$-dimensional measure.
See also \cite{AM-blow} for another more recent work which extends this and other results in different directions, in particular it applies to more general domains and also to the case of elliptic measure associated
to elliptic PDE's in divergence form associated with matrices with $\vmo$ type coefficients.
 
In other related problems of more quantitive nature one assumes stronger quantitative analytic conditions, and consequently one tries to obtain some rather precise quantitative geometric information. For instance, Kenig and Toro
in \cite{Kenig-Toro-crelle} showed that, given two NTA domains $\Omega^+$, $\Omega^-$, so that $\Omega^-=\R^{n+1}\setminus\overline{\Omega^+}$ and 
 $\Omega^+$ is $\delta$-Reifenberg flat for some $\delta>0$ small enough, if $\omega^+$
and $\omega^-$ are mutually absolutely continuous and $\log\frac{d\omega^-}{d\omega^+}\in \vmo(\omega^+)$, then $\partial\Omega^+$ is vanishing Reifenberg flat. See Section \ref{secprelim-nta} for the notions of NTA domain, Reifenberg flatness, vanishing Reifenberg flatness, and $\vmo$.

In another series of papers \cite{Kenig-Toro-duke}, \cite{Kenig-Toro-annals}, \cite{Kenig-Toro-AENS}, Kenig
and Toro also studied the so called one-phase problem for harmonic measure in chord-arc domains, which has a strong connection  with the main results we will prove in this paper. By
a chord-arc domain we mean an NTA domain $\Omega\subset\R^{n+1}$ whose surface measure $\HH^n|_{\partial\Omega}$ is $n$-AD-regular (see \rf{eqAD1}).
Some of the main results in that series of papers of Kenig and Toro can be summarized as follows.

\begin{theorema}[Kenig, Toro]
Let $\Omega\subset\R^{n+1}$ be a bounded chord-arc domain which is $\delta$-Reifenberg flat, with $\delta>0$ small enough.
Denote by $\omega$ the harmonic measure in $\Omega$ with pole $p\in\Omega$ and write $\sigma=\HH^n|_{\partial\Omega}$. Then the following conditions are equivalent:
\begin{itemize}
\item[(a)] $\log\dfrac{d\omega}{d\sigma} \in \vmo(\sigma).$

\item[(b)] The inner normal $N$ to $\partial\Omega$ exists $\sigma$-a.e. and it belongs to $\vmo(\sigma)$.

\item[(c)] $\Omega$ is vanishing Reifenberg flat and the inner normal $N$ to $\partial\Omega$ exists $\sigma$-a.e. and it belongs to $\vmo(\sigma)$.
\end{itemize}
\end{theorema}

The preceding theorem is not valid without the $\delta$-Reifenberg flatness assumption on the domain. 
Indeed, as shown in \cite[Proposition 3.1]{Kenig-Toro-annals}, the harmonic measure with pole at infinity in the cone $\Omega=
\{(x_1,x_2,x_3,x_4)\in\R^4:x_1^2+x_2^2+x_3^2<x_4^2\}$ coincides with its surface measure, while it is
clear that the inner normal does not belong to $\vmo(\sigma)$, by the singularity at the origin. By modifying suitably this domain, one can obtain a bounded chord-arc domain with similar properties.
See also \cite{AMT-apde}, where it is shown that, in $\R^3$, if both $\Omega$ and its exterior domain are NTA,
then the Reifenberg flatness assumption in Theorem A is not necessary.

Concerning the two-phase problem, by analogy with Theorem A, one should expect that the assumption that $\Omega^+$ and $\Omega^-$
 are NTA and $\delta$-Reifenberg flat and the fact that $\log\frac{d\omega^-}{d\omega^+}\in \vmo(\omega^+)$
 imply that the inner unit normal $N$ of $\Omega^+$ belongs to $\vmo(\omega^+)$.
 
Recently Engelstein \cite{Engelstein} proved (among other results) that if one strengthens the $\vmo$ condition on $\omega^+$ by 
asking $\log\frac{d\omega^-}{d\omega^+}\in C^\alpha$ for some $\alpha>0$ (still under 
the $\delta$-Reifenberg flat assumption) then the inner unit normal $N$ of $\Omega^+$ belongs to $C^\alpha$ and $\Omega^+$ is a $C^{1+\alpha}$ domain. His work uses
Weiss type monotonicity formulas, among other tools, and such methods cannot be applied to the $\vmo$ case, as far as
we know. Remark that if one does not impose the Reifenberg flat condition on $\Omega^+$, then there may be singular points in the boundary
$\partial\Omega^+$, i.e.\ points whose blowups are not flat (still under the assumption $\log\frac{d\omega^-}{d\omega^+}\in C^\alpha$).
In the recent papers \cite{BET1} and \cite{BET2} Badger, Engelstein and Toro have obtained very remarkable results about the structure of the singular set. First, in \cite{BET1} they have proven the existence of some 
stratification theorem for the singular set (just assuming that $\log\frac{d\omega^-}{d\omega^+}\in \vmo(\omega^+)$), and 
then in \cite{BET2} under the stronger condition $\log\frac{d\omega^-}{d\omega^+}\in C^\alpha$ they have proven the uniqueness of the blowup at singular points. Again their methods rely strongly on the use of Weiss type monotonicity formulas which, apparently, are not useful in the $\vmo$ case. It is also
worth mentioning that the work \cite{AMT-quantcpam} contains a  quantitative version of the
 solution of the two-phase problem in \cite{AMT-cpam} and \cite{AMTV} described above.

In the current paper we obtain a {characterization} of the condition
$\log\frac{d\omega^-}{d\omega^+}\in\vmo(\omega^+)$ in terms of the oscillation of the unit normal of the boundary and other geometric conditions. Our precise result is the following.

\begin{theorem}\label{teo1} 
Let $\Omega^+\subset\R^{n+1}$ be a bounded NTA domain and let $\Omega^-= \R^{n+1}\setminus \overline{\Omega^+}$ be an NTA domain as well.
Denote by $\omega^+$ and $\omega^-$ the respective harmonic measures with poles $p^+\in\Omega^+$ and $p^-\in\Omega^-$. Suppose that $\Omega^+$ is a $\delta$-Reifenberg flat domain, with $\delta>0$ small enough.
Then the following conditions are equivalent:
\begin{itemize}
\item[(a)] $\omega^+$ and $\omega^-$ are mutually absolutely continous and $\log\dfrac{d\omega^-}{d\omega^+} \in \vmo(\omega^+).$

\item[(b)] $\Omega^+$ is vanishing Reifenberg flat, the inner normal $N$ to $\partial\Omega^+$ exists $\omega^+$-a.e. and it belongs to $\vmo(\omega^+)$, and
both
$\dfrac{d\omega^-}{d\omega^+}\in B_{3/2}(\omega^+)$ and $\dfrac{d\omega^+}{d\omega^-}\in B_{3/2}(\omega^-)$.

\item[(c)] $\Omega^+$ is vanishing Reifenberg flat, the inner normal $N$ to $\partial\Omega^+$ exists $\omega^+$-a.e. and it belongs to $\vmo(\omega^+)$, and
either 
$\dfrac{d\omega^-}{d\omega^+}\in B_{3/2}(\omega^+)$ or $\dfrac{d\omega^+}{d\omega^-}\in B_{3/2}(\omega^-)$.

\item[(d)] $\Omega^+$ is vanishing Reifenberg flat,  { $\Omega^+$ and $\Omega^-$ have joint big pieces of chord-arc subdomains}, and 
\begin{equation}\label{eqnb*}
\lim_{\rho\to 0} \sup_{r(B)\leq \rho} \avint_B |N - N_{B}|\,d\omega^+ = 0, 
\end{equation}
where $N_B$ is the normal to the $n$-plane $L$ pointing to $\Omega^+$ and minimizing 
$$
\max\Big\{\sup_{y\in \partial\Omega\cap B}\dist(y,L), \sup_{y\in L\cap B}\dist(y,\partial\Omega)\Big\}.$$
\end{itemize}
\end{theorem}

\vv
Some remarks are in order. First, 
recall that $B_{3/2}(\omega^+)$ is the class of functions satisfying a reverse $3/2$-H\"older inequality
with respect to $\omega^+$.
So $\frac{d\omega^-}{d\omega^+}\in B_{3/2}(\omega^+)$ means that for any ball $B$ centered { in} $\partial
\Omega^+$,
$$\avint_B \bigg(\dfrac{d\omega^-}{d\omega^+}\bigg)^{3/2}\,d\omega^+\leq C\,\bigg(\frac{\omega^-(B)}{\omega^+(B)}\bigg)^{3/2},$$
for some fixed constant $C>0$.
Second, 
the assumption $\log\frac{d\omega^-}{d\omega^+} \in \vmo(\omega^+)$ implies that $\omega^-\in A_\infty(\omega^+)$ (see Lemma \ref{lem3.1}), which in turn
is equivalent to the fact that { $\Omega^+$ and $\Omega^-$ have joint big pieces of chord-arc subdomains}, and also to the fact that 
both $\omega^+$ and $\omega^-$ have big pieces of uniformly rectifiable measures, as shown recently in \cite{AMT-quantcpam} (for this result to hold, the Reifenberg flatness condition on $\Omega^+$ is not required; one only needs $\Omega^+$ and $\Omega^-$ to be NTA domains).
{ See Section \ref{secprelim-nta} for the notion of joint big pieces of chord-arc subdomains,} and
\rf{equr00'} for the definition of big pieces of uniformly rectifiable measures.
In fact, the condition $\log\frac{d\omega^-}{d\omega^+} \in \vmo(\omega^+)$ together with the John-Nirenberg inequality implies that $\frac{d\omega^-}{d\omega^+}\in B_{p}(\omega^+)$ for any $p\in (1,\infty)$, and in particular for $p=3/2$. So our
contribution in the implication (a) $\Rightarrow$ (b) is that the unit normal $N$ belongs to
$\vmo(\omega^+)$ under the condition (a) (recall that the vanishing Reifenberg flatness follows from \cite{Kenig-Toro-crelle}). 

Notice that the condition (d) is essentially of geometric nature, providing a geometric characterization of when $\log\dfrac{d\omega^-}{d\omega^+} \in \vmo(\omega^+)$, which is a property of more analytic nature.
The vector $N_B$ should be considered as an approximate normal to $\partial\Omega$ in the ball $B$.
We do not know if the condition \rf{eqnb*} in (d) can be replaced by the (a priori) weaker condition $N\in
\vmo(\omega^+)$, or equivalently, if $N_B$ can be replaced by the mean value of $N$ with respect to
$\omega^+$ in $B$.

Finally we remark that again the Reifenberg flatness condition on the domain is necessary in the theorem. This can be easily seen by taking a suitable smooth truncation of the cone $\Omega^+ = \{x_1,x_2,x_3,x_4)\in\R^4:x_1^2+x_2^2 < x_3^2+x_4^2\}$, for which the harmonic measures $\omega^+$ and $\omega^-$ with pole at $\infty$ coincide.

It is worth comparing Theorem \ref{teo1} to Theorem A. 
Notice the analogies between the respective
equivalences (a) $\Leftrightarrow$ (c) in both theorems. The main difference is the presence of 
the $B_{3/2}(\omega^+)$ condition in (c) in our result. 
The condition (d) in Theorem \ref{teo1} also looks similar to (c) in Theorem A.  Indeed, in  Theorem A, $\Omega$ is assumed to be a chord-arc domain, and
the condition on the joint big pieces of chord-arc subdomains
 that appears in (d) in Theorem \ref{teo1} holds trivially in the particular case when $\Omega^\pm$ 
 are chord-arc domains. Further, remark that under the assumptions of Theorem A, when $\Omega$ is Reifenberg flat, the condition $N\in\vmo(\sigma)$ is equivalent to 
$$\lim_{\rho\to 0} \sup_{r(B)\leq \rho} \avint_B |N - N_{B}|\,d\sigma = 0, 
$$
which is the analogue of \rf{eqnb*}.

When $\Omega^+$ is a chord-arc domain we derive the following corollary.

\begin{coro}\label{coro2}
Let $\Omega^+$ and $\Omega^-$ and their harmonic measures satisfy the assumptions of Theorem \ref{teo1} and assume that $\log\dfrac{d\omega^-}{d\omega^+} \in \vmo(\omega^+)$.
 Suppose in addition that $\partial\Omega^+$ is $n$-AD regular. Then,
$$N\in\vmo(\HH^n|_{\partial\Omega^+}).$$
\end{coro}


We think that the analogous theorem and corollary for unbounded NTA domains with poles at $\infty$ also hold and can be obtained by similar techniques. However, for the sake of brevity we have only written the detailed arguments in our paper assuming
$\Omega^+$ to be bounded. We also point out that almost the same proof of Theorem \ref{teo1} shows that, if instead of assuming $\log\frac{d\omega^-}{d\omega^+} \in \vmo(\omega^+)$, one assumes some local $\bmo(\omega^+)$ norm  
of $\log\frac{d\omega^-}{d\omega^+}$  to be small enough, then one gets smallness for some local $\bmo(\omega^+)$ norm of the inner unit normal of
$\partial\Omega^+$. See Theorem \ref{teofifi} for the precise statement. The same happens regarding the converse statements and 
Corollary \ref{coro2}.

The proof of Theorem \ref{teo1} relies on four basic ingredients: first, it is essential to use the information provided by the solution of the two-phase problem for general domains satisfying the so-called capacity density condition (CDC) in
the work \cite{AMT-cpam} (see Theorem \ref{t:AMT} below). This result implies that, under the assumptions of Theorem \ref{teo1}, $\omega^+$ and $\omega^-$
are rectifiable measures concentrated in the set of tangent points for $\partial\Omega^+$. 

The second main ingredient for the proof of Theorem \ref{teo1} is the availability of suitable jump identities for the Riesz transform. Let us remark that Bortz and Hofmann already realized in \cite{BH} that such identities are useful in connection with two-phase problems for harmonic measure. Their work deals with a somewhat different problem: assuming that
$\partial\Omega^+$ is uniformly $n$-rectifiable and that the measure theoretic boundary of $\Omega^+$ has full surface measure
 and that both
$\log\frac{d\omega^+}{d\HH^n|_{\partial\Omega^+}}\in\vmo(\HH^n|_{\partial\Omega^+})$ and $\log\frac{d\omega^-}{d\HH^n|_{\partial\Omega^+}}\in\vmo(\HH^n|_{\partial\Omega^+})$, in \cite{BH} it is shown that $N\in\vmo(\HH^n|_{\partial\Omega^+})$. The same result had been obtained earlier by Kenig and Toro in \cite{Kenig-Toro-crelle} under the somewhat stronger assumption that $\Omega^+$ and $\Omega^-$ are chord-arc domains by using blowup methods. 
See also \cite{MMV} for a related result of
Mitrea, Mitrea, and Verdera 
involving jump identities for the Riesz transforms and the $\vmo$ character of the boundary of a domain
with $n$-AD-regular boundary.

A priori, the assumptions in Theorem \ref{teo1} do not ensure that the surface measure $\HH^n|_{\partial\Omega^+}$
is locally finite (to this end, see a related example in \cite[Section 7]{AMT-quantcpam}, which shows that there exists a planar NTA domain
whose boundary has non-$\sigma$-finite length, and
whose harmonic measure is concentrated in a dense rectifiable subset of the boundary). This does not allow the application 
of the rather classical jump formulas for Riesz transforms for chord arc domains, such as the ones from \cite{HMT}. Instead, in our work we use the more general formulas obtained recently by the second author of this paper in \cite{Tolsa-jump}, which are valid for arbitrary $n$-rectifiable
sets. 

The third main tool to prove Theorem \ref{teo1} is the rectifiability criterion for general Radon measures in terms of Riesz transforms obtained by Girela-Sarri\'on and Tolsa in \cite{GT} by using techniques inspired by the solution of
the David-Semmes problem in \cite{NToV} and the previous related work \cite{ENV}. The application of such criterion is
essential to prevent the degeneracy of the $n$-dimensional density of harmonic measure in suitable big pieces of sets in some key estimates. Notice that the use of this criterion is also one of the essential tools in
the works \cite{AMT-cpam} and \cite{AMTV}. An interesting novelty in the present paper is that to obtain the required big piece where the density of harmonic measure does not degenerate we apply that criterion in an iterated way. See also \cite{AHM3TV}, \cite{MT}, and \cite{GMT} for other applications of the solution of the David-Semmes problem to questions in connection with harmonic measure.

The fourth main ingredient of the proof is the Kenig-Toro solution of the one-phase problem for harmonic measure in Theorem A. We will apply their result in the implications (c) $\Rightarrow$ (a) and  (d) $\Rightarrow$ (a) in Theorem
\ref{teo1}. To this end, we will construct some approximating chord-arc domain following an idea from \cite{AMT-singular}. The most delicate point consists in showing that the unit normal of the new approximating domain belongs to $\bmo(\sigma)$ with a very small constant, where $\sigma$ is the surface measure of this new domain. To prove that this holds we need to use the fact that the oscillation of the unit normal
is small with respect to the approximate normal $N_B$.
It is in this key point that the $B_{3/2}(\omega^\pm)$ condition in (c) is required.
Once this is done, we will transfer the estimates for the harmonic measure in the approximating domain obtained by Theorem A to our  
original domain $\Omega^+$ by means of the maximum principle.

The plan of the paper is the following. Section \ref{secprelim} contains some preliminary results on geometric measure theory, harmonic analysis and potential theory that are used in the subsequent
arguments of the paper. Section \ref{sec3} is devoted to the implication (a) $\Rightarrow$ (b) of Theorem \ref{teo1}. The implication (b) $\Rightarrow$ (c) is trivial, while
(c) $\Rightarrow$ (a) is proven in Section \ref{secctoa}. One of its main steps is the proof of Lemma \ref{keylemma}, which deals with the oscillation of the unit normal $N$ with respect to
the approximate normal $N_B$ in a ball $B$. In the final Section we show (b) $\Rightarrow$ (d) $\Rightarrow$ (a). The arguments in this final section are very similar to the ones for the implication (c) $\Rightarrow$ (a) (and also (b) $\Rightarrow$ (a)). 

\vv
Acknowledgement. We would like to thank the referee for the careful revision of the paper.

\vv


\section{Preliminaries}\label{secprelim}

We denote by $C$ or $c$ some constants that may depend on the dimension and perhaps other fixed parameters. Their value may change at different occurrences. On the contrary, constants with subscripts, like $C_0$, retain their values.
For $a,b\geq 0$, we write $a\lesssim b$ if there is $C>0$ such that $a\leq Cb$. We write $a\approx b$ to mean $a\lesssim b\lesssim a$.

\subsection{Measures, rectifiability, and tangents}
\label{secprelim1}
All measures in this paper are assumed to be Borel measures.
A measure $\mu$ in $\R^d$ is called {\em doubling} if there is some constant $C>0$ such that
$$\mu(B(x,2r))\leq C\,\mu(B(x,r))\quad \mbox{ for all $x\in\supp\mu$.}$$
The measure $\mu$ is called {\em $n$-AD regular} (or $n$-Ahlfors-David regular) if
\begin{equation}\label{eqAD1}
C^{-1}r^n\leq \mu(B(x,r))\leq C r^n \quad \mbox{ for all $x\in\supp\mu$ and $0<r\leq \diam(\supp\mu)$.}
\end{equation}
Obviously, $n$-AD regular measures are doubling.
A set $E\subset \R^d$ is called $n$-AD regular if $\HH^n|_E$ is $n$-AD regular. In case that $\mu$ satisfies the second inequality in \rf{eqAD1}, but not necessarily the first one, we say that $\mu$ has $n$-polynomial growth.

Recall the definition of $n$-rectifiable sets in \rf{eq001}. Analogously, one says that a measure $\mu$ is 
{\em $n$-rectifiable}
if there are Lipschitz maps
$f_i:\R^n\to\R^d$, $i=1,2,\ldots$, such that 
\begin{equation}\label{eq001'}
\mu\Bigl(\R^d\setminus\textstyle\bigcup_i f_i(\R^n)\Big) = 0,
\end{equation} 
and moreover $\mu$ is absolutely continuous with respect to $\HH^n$. An equivalent definition for rectifiability of sets and measures is obtained if we replace Lipschitz images of $\R^n$ by possibly rotated $n$-dimensional graphs of $C^1$ functions.

A measure $\mu$ in $\R^d$ is called {\em uniformly $n$-rectifiable} (UR) if it is $n$-AD-regular and
there exist constants $\theta, M >0$ such that for all $x \in \supp\mu$ and all $0<r\leq \diam(\supp\mu)$ 
there is a Lipschitz mapping $g$ from the ball $B_n(0,r)$ in $\R^{n}$ to $\R^d$ with $\text{Lip}(g) \leq M$ such that
$$
\mu(B(x,r)\cap g(B_{n}(0,r)))\geq \theta r^{n}.$$
A set $E$ is called uniformly  $n$-rectifiable if the measure $\HH^n|_E$ is uniformly $n$-rectifiable.
The notion of uniform $n$-rectifiability is a quantitative version of $n$-rectifiability introduced 
by David and Semmes (see \cite{DS1}). It is very easy to check that uniform $n$-rectifiability implies $n$-rectifiability.

We say that a measure $\nu$ in $\R^d$ has 
 big pieces of uniformly $n$-rectifiable measures if there exists some $\ve\in (0,1)$ such that,
for every ball $B$ centered  {$\supp(\nu)$} 
with radius at most $\diam(\supp(\nu))$, there exists a uniformly $n$-rectifiable set $E$, with UR constants possibly depending on $\ve$, and a subset $F\subset E$ such that
\begin{equation}\label{equr00'}
\nu(B\setminus F)\leq \ve\,\nu(B)
\end{equation}
and
\begin{equation}\label{equr01'}
\nu(D)\approx_\ve \HH^n(D)\,\frac{\nu(B)}{r(B)^n}\quad \mbox{ for all $D\subset F$.}
\end{equation}

For a point $x\in\R^{n+1}$, a unit vector $u$, 
and an aperture parameter $a\in(0,1)$ we consider the one sided cone with axis in the direction of $u$ defined by
$$X_a(x,u)=\bigl\{y\in\R^{n+1}:(y-x)\cdot u> a|y-x|\bigr\}.$$
We say that $E\subset\R^{n+1}$ has a {\em tangent} $n$-plane at $x\in E$ and that $x$ is a tangent point for $E$ if there exists a unit vector $u$ such that, for all
$a\in(0,1)$, there exists some $r>0$ such that
$$E\cap \bigl(X_a(x,u) \cup X_a(x,-u)\bigr) \cap B(x,r) =\varnothing.$$
The $n$-plane $L$ orthogonal to $u$ through $x$ is called a tangent $n$-plane at $x$.

We say that $E$ has an {\em approximate tangent} $n$-plane at $x\in E$ if there exists a unit vector $u$ such that, for all
$a\in(0,1)$,
$$\lim_{r\to0} \frac{\HH^n\bigl(E\cap \bigl(X_a(x,u) \cup X_a(x,-u)\bigr)\cap B(x,r)\bigr)}{r^n}=0.$$
The $n$-plane $L$ orthogonal to $u$ through $x$ is called approximate tangent $n$-plane.
Recall that if $\HH^n(E)<\infty$ (or $\HH^n|_E$ is locally finite) and $E$ is $n$-rectifiable, then there is a unique
approximate tangent $n$-plane at $\HH^n$-a.e.\ $x\in E$. 
\vv

\subsection{The jump identities for singular integrals}

Let $K:\R^{n+1}\setminus\{0\}\to \R$ (or $K:\R^{n+1}\setminus\{0\}\to \R^{n+1}$) be an odd  Calder\'on-Zygmund kernel 
satisfying
\begin{equation}\label{eq4}
K(x) = \frac{H(x)}{|x|^n},\qquad
|\nabla^j K(x)|\leq \frac C{|x|^{n+j}}\quad \mbox{ for $j=0,1,2$ and $x\neq 0$},
\end{equation}
where $H$ is homogeneous of degree $0$ and $C^2$ in the unit sphere.
Given a signed Radon measure $\nu$ in $\R^{n+1}$ we denote
$$T\nu(x) = \int K(x-y)\,d\nu(y),$$
whenever the integral makes sense,
and for $\ve>0$
$$T_\ve\nu(x) = \int_{|x-y|>\ve} K(x-y)\,d\nu(y)$$
and also
$$\pv T\nu(x) = \lim_{\ve\to 0} T_\ve \nu(x),$$
whenever the last limit exists. In fact, as shown by Mas \cite{Mas}, it turns out that,  for any $n$-rectifiable set $E\subset\R^{n+1}$ and any  finite signed measure $\nu$ in $\R^{n+1}$ (not necessarily supported on $E$),  
 $\pv T\nu(x)$ exists for $\HH^n$-a.e.\ $x\in E$, for a kernel $K$ satisfying the estimate on the right hand side of \rf{eq4}, not necessarily homogeneous.
 
As mentioned above, 
if $\HH^n(E)<\infty$ (or $\HH^n|_E$ is locally finite) and $E$ is $n$-rectifiable, there is a unique
approximate tangent $n$-plane at $\HH^n$-a.e.\ $x\in E$. We denote by $L_x$ the approximate tangent $n$-plane at $x$
and by $N_x$  a unit vector orthogonal to $L_x$. We also write
$$X_a^+(x) = X_a(x,N_x),\qquad X_a^-(x) = X_a(x,-N_x),\qquad X_a(x) = X_a^+(x) \cup X_a^-(x).$$
For $\HH^n$-a.e.\ $x\in E$ there are two possible choices for $N_x$, depending on the sense of the normal to $L_x$. In the next theorem the choice $x\mapsto N_x$ does not matter as soon as it is $\HH^n$-measurable (or Borel, say if $E$ is Borel). 

Fix $b\in(0,a)$, so that $\bar B(y,b|y-x|)\cap L_x=\varnothing$ for all $y\in X_a^+(x)\cup X_a^-(x)$.
We define the non-tangential limits
\begin{equation}\label{eqdtt}
T^+\nu(x) = \lim_{X_a^+(x)\ni y\to x} T_{b|x-y|}\nu(y),\qquad T^-\nu(x) = \lim_{X_a^-(x)\ni y\to x} T_{b|x-y|}\nu(y),
\end{equation}
whenever they exist. Note that we use the truncated operators $T_{b|x-y|}$ in these definitions, which may appear rather unusual. On the other hand, if $x$ is a tangent point for $E$
and $\supp\nu\subset E$, then we can replace $T_{b|x-y|}\nu(y)$ by $T\nu(y)$ in the above definitions.
The following result is proved in \cite{Tolsa-jump}.

\begin{theorem}\label{teo-jump}
Let $T$ be the operator associated with an odd Calder\'on-Zygmund kernel of homogeneity $-n$, $C^2$ away from the origin, and satisfying
\rf{eq4}. Let $E\subset \R^{n+1}$ be an $n$-rectifiable set and let $\nu$ be a finite signed Radon measure in $\R^{n+1}$. For fixed $a\in
(0,1)$ and $b\in(0,a)$ as above, the non-tangential limits 
$T^+\nu(x)$, $T^-\nu(x)$, and the principal value $\pv T\nu(x)$ exist for $\HH^n$-a.e.\ $x\in E$ and moreover the following identities hold
for $\HH^n$-a.e.\ $x\in E$ too:
\begin{equation}\label{eqga}
 \frac12 \bigl(T^+\nu(x) + T^-\nu(x)\bigr)= \pv T\nu(x),
 \end{equation}
 and
\begin{equation}\label{eqgb}
\frac12 \bigl(T^+\nu(x) - T^-\nu(x)\bigr) = C_K(N_x)\,\frac{d\nu}{d\HH^n|_E}(x),
\end{equation}
where $C_K(N_x)$ is defined by
\begin{equation}\label{eqckn}
C_K(N_x)= \int_{N_x^\bot} \frac{H(y+N_x)- H(y-N_x)}{2(|y|^2 + 1)^{n/2}}\,d\HH^n(y),
\end{equation}
where $N_x^\bot$ is the hyperplane orthogonal to $N_x$ through the origin.
\end{theorem}

We 
 remark that for \rf{eqckn} to hold, we assume $\HH^n$ to be defined with a normalization factor so that
it coincides with the $n$-dimensional Lebesgue measure on any hyperplane. Further, we understand that
$\frac{d\nu}{d\HH^n|_E}$ is the density of the absolute continuous part of $\nu$ with respect to $\HH^n|_E$.
In the particular case when $K$ is the  $n$-dimensional Riesz kernel, i.e.,
$K(x) = \frac x{|x|^{n+1}}$, the integrand in \rf{eqckn} coincides with $N_x$ times some multiple of the Poisson kernel and
it easily follows that 
$$
C_K(N_x) = \frac{\omega_{n}}2\,N_x,
$$
where $\omega_{n}$ is the $n$-dimensional volume of the unit sphere in $\R^{n+1}$. For additional remarks, as well as for the proof of Theorem \ref{teo-jump}, see \cite{Tolsa-jump}.
\vv

\subsection{NTA and Reifenberg flat domains}
\label{secprelim-nta}

 Given $\Omega\subset \mathbb{R}^{n+1}$,
we say that $\Omega$ satisfies the {\it Harnack chain condition} if there are positive constants $c$ and $R$ such that
for every $\rho>0$, $\Lambda\geq 1$, and every pair of points
$x,y \in \Omega$ with $\dist(x,\partial\Omega),\,\dist(y,\partial\Omega) \geq\rho$ and $|x-y|<\Lambda\,\rho\leq R$, there is a chain of
open balls
$B_1,\dots,B_m \subset \Omega$, $m\leq C(\Lambda)$,
with $x\in B_1,\, y\in B_m,$ $B_k\cap B_{k+1}\neq \varnothing$
and $c^{-1}\diam (B_k) \leq \dist (B_k,\partial\Omega)\leq c\,\diam (B_k).$  The chain of balls is called
a {\it Harnack chain}. Note that if such a chain exists, then any positive harmonic function $u:\Omega\to\R$ satisfies 
\[u(x)\approx u(y),\]
with the implicit constant depending on $m$ and $n$.
For $C\geq 2$, $\Omega$ is a {\it $C$-corkscrew domain} if for all $\xi\in \partial\Omega$ and $r\in(0,R)$ there are two balls of radius $r/C$ contained in $B(\xi,r)\cap \Omega$ and $B(\xi,r)\backslash \Omega$ respectively. If $B(x,r/C)\subset B(\xi,r)\cap \Omega$, we call $x$ an {\it (interior) corkscrew point} for the ball $B(\xi,r)$. Finally, we say that $\Omega$ is {\it $C$-non-tangentially accessible (or $C$-NTA, or just NTA)} if it satisfies the Harnack chain condition and it is a $C$-corkscrew domain.   Also, $\Omega$ is
{\it two-sided $C$-NTA} if both $\Omega$ and $\Omega_{\rm {ext}}:=(\overline{\Omega})^{c}$ are $C$-NTA. 

NTA domains were introduced by Jerison and Kenig in \cite{Jerison-Kenig}. In that work, the behavior of harmonic measure in this type of domains was studied in detail. Among other results, the authors showed that harmonic measure is doubling in NTA domains, and its support
coincides with the whole boundary.

Given a set $E\subset\R^{n+1}$, $x\in \mathbb{R}^{n+1}$, $r>0$, and $P$ an $n$-plane, we set
\begin{equation}\label{eqDE}
D_{E}(x,r,P)=r^{-1}\max\left\{\sup_{y\in E\cap B(x,r)}\dist(y,P), \sup_{y\in P\cap B(x,r)}\dist(y,E)\right\}.
\end{equation}
We also define
\begin{equation}\label{eqDE2}
D_{E}(x,r)=\inf_{P}D_{E}(x,r,P)
\end{equation}
where the infimum is over all $n$-planes $P$. For a given ball $B=B(x,r)$, we will also write
$D_{E}(B)$ instead of $D_{E}(x,r)$.
Given $\delta, R>0$, set $E$ is {\em $(\delta,R)$-Reifenberg flat} (or just $\delta$-Reifenberg flat) if $D_{E}(x,r)<\delta$ for all $x\in E$ and $0<r\leq R$, and it is {\it vanishing Reifenberg flat} if 
\[
\lim_{r\rightarrow 0} \sup_{x\in E} D_{E}(x,r)=0.\]

Let $\Omega\subset \R^{n+1}$ be an open set, and let $0<\delta<1/2$. We say that $\Omega$
is a $(\delta,R)$-Reifenberg flat domain (or just $\delta$-Reifenberg flat) if it satisfies the following conditions:
\begin{itemize}
\item[(a)] $\partial\Omega$ is $(\delta,R)$-Reifenberg flat.

\item[(b)] For every $x\in\partial \Omega$ and $0<r\leq R$, denote by $P(x,r)$ an $n$-plane that minimizes $D_{E}(x,r)$. Then one of the connected components of 
$$B(x,r)\cap \bigl\{x\in\R^{n+1}:\dist(x,P(x,r))\geq 2\delta\,r\bigr\}$$
is contained in $\Omega$ and the other is contained in $\R^{n+1}\setminus\Omega$.
\end{itemize}
If, additionally,  $\partial\Omega$ is vanishing Reifenberg flat, then $\Omega$ is said to be vanishing Reifenberg flat, too.
It is well known that if $\Omega$ is a $\delta$-Reifenberg flat domain, with $\delta$ small enough,
then it is also an NTA domain (see \cite{Kenig-Toro-duke}).

{
Given two NTA domains $\Omega^+\subset \R^{n+1}$ and $\Omega^-=\R^{n+1}\setminus \overline{\Omega^+}$, 
we say that $\Omega^+$ and $\Omega^-$ have  joint big pieces of chord-arc subdomains if for any ball $B$ centered in $\partial\Omega^+$ with radius at most $\diam \partial\Omega^+$ there are two chord-arc domains
$\Omega_{B}^s\subset \Omega^s$, with $s=+,-$, such that $\HH^n(\partial\Omega_B^1\cap 
\partial\Omega_B^2\cap B)\gtrsim r(B)^n$. Recall that
a chord-arc domain is an NTA domain $\Omega\subset\R^{n+1}$ whose surface measure $\HH^n|_{\partial\Omega}$ is $n$-AD-regular.
}

 \vv


\subsection{The space $\vmo$}

Given a Radon measure $\mu$ in $\R^{n+1}$, $f\in L^1_{loc}(\mu)$, and $A\subset \R^{n+1}$, we write
$$m_{\mu,A}(f) = \avint_A f\,d\mu=
\frac1{\mu(A)}\int_A f\,d\mu.$$
Assume $\mu$ to be doubling.
We say that $f\in \vmo(\mu)$ if
\begin{equation}\label{defvmo}
\lim_{r\rightarrow 0}\sup_{x\in \supp \mu}\,\, \avint_{B(x,r)} \left| f- m_{\mu,B(x,r)}f\,d\mu\right|^{2}d\mu =0.
\end{equation}
It is well known  that the space VMO  coincides with the closure of the set of bounded uniformly continuous functions on $\supp \mu$ in the BMO norm. 


\vv

\subsection{Dyadic lattices and densities}
To prove Theorem \ref{teo1} we will use a dyadic decomposition of $\partial\Omega^+$ to obtain the VMO estimates by iteration arguments. In \cite{Christ} Michael Christ introduced a dyadic decomposition of the support of a doubling Radon measure in certain metric spaces which in particular applies to our case. We state below the precise result applied to our particular situation
when the metric space is the boundary of the domain $\Omega^+$ in Theorem \ref{teo1} and the doubling measure is $\omega^+$.

\begin{theorem}[{\cite[Theorem 11]{Christ}}] \label{teo-christ}
Let $\Omega^+$ and $\omega^+$ be as in Theorem \ref{teo1}.
There exist a family $\DD$ of relatively open subsets of $\partial\Omega^+$  and constants $0<r_0<1$, $0<a_1,\eta, C_1,C_2<\infty$ such that $\mathcal{D}=\bigcup_{k\in\Z} \mathcal{D}_k$ with $\mathcal{D}_k=\{Q^{i}\}_{i\in I_k}$, and the following holds:
\begin{enumerate}
\item[(a)] For every $k\in\Z$ we have $\displaystyle\omega^+\bigg(\partial\Omega^+\setminus \bigcup_{Q\in\mathcal{D}_k} Q\bigg) =0$.
\item[(b)] For every $k_0\leq k_1$ and $Q_j \in \mathcal{D}_{k_j}$ for $j\in \{0,1\}$, then either $Q_1\subset Q_0$ or $Q_1\cap Q_0=\varnothing$. 
\item[(c)] For each $Q_1\in \mathcal{D}_{k_1}$ and each $k_0<k_1$ there exists a unique cube $Q_0\in \mathcal{D}_{k_0}$ such that $Q_1\subset Q_0$.
\item[(d)] For $Q\in\mathcal{D}_k$ there are $z_Q\in Q$ and balls $B_Q=B(z_Q,a_1r_0^k)$ and $\wt B_Q= B(z_Q,C_1r_0^k)$ such that $B_Q\cap\partial \Omega^+ \subset Q\subset \wt B_Q$.
\item[(e)] For $Q\in\mathcal{D}_k$ we have $\omega^+(\{x\in Q: \dist(x, \partial \Omega^+\setminus Q)\leq tr_0^k\})\leq  C_2 t^\eta\,\omega^+(Q)$ for every $t>0$. 
\end{enumerate}
\end{theorem}

We say that $Q\in\mathcal{D}_k$ is a \emph{dyadic cube} of generation $k$, and write $\ell(Q):=2C_1r_0^k$. We call $\ell(Q)$  the side length of $Q$.

Whenever we want to control the $\bmo(\omega^+)$ norm of a function, it is not enough to study $\avint_Q |f-m_{\omega^+,Q}f|^2\,d\omega^+$ in dyadic cubes $Q$, since there may be balls that are not included in any of those cubes with a comparable diameter. In the Euclidean space it is enough to take dilations of the cubes. The drawback is that these dilated cubes do not have a lattice structure. To deal with this technical issue we will take unions of neighboring cubes, so that every such union decomposes in dyadic cubes at every scale. The definition in the Euclidean space of neighboring cubes is quite simple: two cubes are neighboring if their closures intersect. In Christ's decomposition we have to be a little more careful.

We say that two cubes $Q,S\in \DD_k$ are {\em neighbors}, writing $S\in\mathcal{N}(Q)$, if $\frac32 \wt B_Q\cap \frac 32 \wt B_S\neq \varnothing$. We write
 $\mathcal{ND}_k:=\left\{\bigcup_{S\in\mathcal{N}(Q)}S\right\}_{Q\in\mathcal{D}_k}$ and $\mathcal{ND}:=\bigcup_{k} \mathcal{ND}_k$. We say that $P\in\mathcal{ND}_k$ is an extended cube of generation $k$, and write $\ell(P):=2C_1r_0^k$.

Many of our arguments on cubes will need to be applied both to the dyadic and the extended cubes.
We write $\wh\DD_k=\DD_k\cup\mathcal{ND}_k$ and $\wh\DD=\DD\cup\mathcal{ND}$. We refer as {\em cubes} to both the dyadic and the extended cubes.

We need also to introduce ``dilations'' of cubes. Given $Q\in \wh\DD_k$ and $\Lambda >1$, we write
$$\Lambda Q = \{x\in \partial\Omega^+: \dist(x,Q) < (\Lambda-1)\ell(Q)\}.$$
and $\ell( \Lambda Q ):=2C_1(\Lambda r_0)^k$. Obviously, we also define $1Q\equiv Q$.

Let $\mu$ be a Radon measure in $\R^{n+1}$. 
Given a ball $B\subset\R^{n+1}$, we denote
\begin{equation}\label{sec2.6}
\Theta_\mu(B) = \frac{\mu(B)}{r(B)^n},\qquad P_\mu(B) = \sum_{j\geq0} 2^{-j}\,\Theta_\mu(2^jB).
\end{equation}
So $\Theta_\mu(B)$ is the $n$-dimensional density of $\mu$ on $B$ and $P_\mu(B)$ is some kind of smoothened version of this density. Analogously, given $Q\in\wh\DD$ and $\Lambda\geq1$, we denote
$$\Theta_{\mu}(\Lambda Q) = \frac{\mu(\Lambda Q)}{\ell(\Lambda Q)^n}\qquad P_\mu(\Lambda Q) = \sum_{j\geq0} 2^{-j}\,\Theta_\mu(2^j\Lambda Q).$$ 
\vv


\subsection{The two-phase problem for harmonic measure in domains satisfying the CDC} \label{sec2cdc}
 

For $n\geq2$, let $\Omega\subset\R^{n+1}$ be open.
We say that the {\em capacity density condition} (or CDC) holds if there exists constants $c(\Omega),R(\Omega)>0$ such that  
$$\capp(B(x,r)\cap \Omega^c)\geq c(\Omega)\,r^{n-1}\quad\mbox{ for all $x\in \partial\Omega$ and $0<r\leq R(\Omega)$,}
$$
where $\capp$ stands for the Newtonian capacity.
By the corkscrew condition, any NTA domain satisfies the CDC.

\vv

Next we record the precise result from \cite{AMT-cpam} regarding the solution of the two-phase problem for harmonic measure for domains satisfying the CDC: 

\begin{theorem}\label{t:AMT}
For $n\geq 2$, let $\Omega^+\subset \R^{n+1}$ be  open and let $\Omega^- = \bigl(\,\overline{\Omega^+}\,\bigr)^c$. Assume that $\Omega^+,\Omega^-$ are both connected, satisfy the CDC, and $\partial\Omega^+ = \partial\Omega^-$.
Let $\omega^\pm$ be the respective harmonic measures of $\Omega^\pm$.
Let $E\subset\partial\Omega^+$ be such that $\omega^+$ and $\omega^-$ are mutually absolutely continuous in $E$.
Then there exists an $n$-rectifiable subset $E'\subset E$ such that all points from $E'$ are tangent points for
$\partial\Omega^+$ and $E'$ has full harmonic measure in $E$ (that is, $\omega^+(E\setminus E') = \omega^-(E\setminus E')=0$), and moreover $\omega^+$, $\omega^-$, and $\HH^n$ are mutually absolutely continuous on $E'$.
\end{theorem}
\vv

Consider now the particular case of two NTA domains $\Omega^+,\Omega^-\subset\R^{n+1}$ satisfying the assumptions of Theorem \ref{teo1}.
By Theorem \ref{t:AMT} there exists a Borel set $F\subset\partial\Omega^+$ satisfying the following:
\begin{itemize}
\item[(a)] $F$ is $n$-rectifiable, has full harmonic measure (both for $\omega^+$ and $\omega^-$), and $\omega^+$, $\omega^-$, and $\HH^n$ are mutually absolutely continuous on $F$,

\item[(b)] all points from $F$ are tangent points for $\partial\Omega^+$, and 
\item[(c)] $F$ is is dense in $\partial\Omega^+$.
\end{itemize}
The first two statements are a direct consequence of Theorem \ref{t:AMT}, while the last one follows from the fact that 
$\supp\omega^+=\supp\omega^- = \partial\Omega^+$.

By the definition of tangent points, for every $x\in F$ and $a\in (0,1)$, there exists some $r=r(a,x)$ such 
that
$$X_a^+(x,N_x,r) \cup X_a^-(x,N_x,r)\subset \Omega^+\cup\Omega^-,$$
where $N_x$ is one of the two possible choices of the normal to the tangent $n$-plane of $\partial\Omega^+$ at $x$, and to shorten notation we wrote
$$X_a^\pm(x,N_x,r):= X_a^\pm(x,N_x)\cap B(x,r).$$
 Assume 
$a\ll1$ (to be chosen in a moment depending on the NTA constant of $\Omega^\pm$).
By connectedness, either $X_a^+(x,N_x,r)\subset \Omega^+$ or $X_a^+(x,N_x,r)\subset\Omega^-$, and the same happens for
$X_a^-(x,N_x,r)$. On the other hand, $X_a^+(x,N_x,r)$ and $X_a^-(x,N_x,r)$ cannot be both contained in
$\Omega^+$, because otherwise $\Omega^-\cap B(x,r)\subset \R^{n+1}\setminus X_a(x,N_x,r)$, which would violate the interior corkscrew
condition for $\Omega^-$ assuming $a$ small enough. Analogously, $X_a^+(x,N_x,r)$ and $X_a^-(x,N_x,r)$ are not both contained in $\Omega^-$. Thus, we may (and will) assume, by interchanging $N_x$ by $-N_x$ if necessary, that
\begin{equation}\label{eqnormal*}
X_a^+(x,N_x,r)\subset \Omega^+\quad \text{ and }\quad X_a^-(x,N_x,r)\subset \Omega^-.
\end{equation}
Further, it is immediate to check that once $N_x$ is chosen so that this happens, then the same property will hold for all 
cones with arbitrary aperture $a\in(0,1)$ and small enough radius.

\vv


\subsection{Riesz transform and jump identities}\label{secjump}

 Given a signed Radon measure $\nu$ in $\R^{d}$ we consider the $n$-dimensional Riesz
transform
$$\RR\nu(x) = \int \frac{x-y}{|x-y|^{n+1}}\,d\nu(y),$$
whenever the integral makes sense (for example, when $\nu$ has bounded support and $x\not\in \supp \nu$). For $\ve>0$, the $\ve$-truncated Riesz transform is given by 
$$\RR_\ve \nu(x) = \int_{|x-y|>\ve} \frac{x-y}{|x-y|^{n+1}}\,d\nu(y),$$
and we set
$$\RR_{*} \nu(x)= \sup_{\ve>0} |\RR_\ve \nu(x)|.$$
If $\mu$ is a fixed Radon measure and $f\in L^1_{loc}(\mu)$, we also write
$$\RR_\mu f = \RR(f\mu),\quad \RR_{\mu,\ve} f = \RR_\ve(f\mu),\quad \RR_{\mu,*} f= \RR_{*} (f\mu),$$
whenever these notions make sense. We say that $\RR_\mu$ is bounded in $L^2(\mu)$ if the
operators $\RR_{\mu,\ve}$ are bounded in $L^2(\mu)$ uniformly on $\ve>0$. Recall that if $\mu$ has $n$-polynomial growth and $\RR_\mu$ is bounded in $L^2(\mu)$, then $\RR$ is bounded from the
space of finite signed Radon measures $M(\R^d)$ to $L^{1,\infty}(\mu)$ (with its norm bounded above depending on $\sup_{\ve>0}\|\RR_{\mu,\ve}\|_{L^2(\mu)\to L^2(\mu)}$ and the $n$-polynomial growth of $\mu$).
This means that, for every
$\nu\in M(\R^d)$ and every $t>0$,
$$\mu\big(\{x\in\R^{d}:|\RR_\ve\nu(x)|>t\}\big)\leq C\frac{\|\nu\|}t,$$ 
with $C$ uniform on $\ve>0$. Recall also that if $E\subset \R^d$ is $n$-rectifiable, then
the principal values 
$$\pv \RR\nu(x) = \lim_{\ve\to0}\RR_\ve\nu(x)$$
exist for $\HH^n$-a.e.\ $x\in E$. See \cite[Chapters 2 and 8]{Tolsa-llibre} for the detailed proofs of the latter results, for example. Abusing notation, we will also write $\RR\nu(x)$ instead of $\pv \RR\nu(x)$.

Assume now that we are under the assumptions of Theorem \ref{teo1}. Let $F$ be the set of tangent points
for $\partial\Omega^+$ described just after Theorem \ref{t:AMT}. Recall that this set is $n$-rectifiable, has full measure for $\omega^+$ and $\omega^-$, and both $\omega^+$ and $\omega^-$ are mutually absolutely continuous with $\HH^n$ on $F$. Consider an arbitrary Borel subset $F'\subset F$ such that $\HH^n(F')
<\infty$. From the definition, it is clear that the tangent points for $\partial\Omega^+$ that belong to $F'$ are also tangent points for $F'$. By the discussion at the end of 
Section \ref{sec2cdc}, to each $x\in F'$ we can assign the inner normal vector
$N_x$ to $\partial\Omega^+$ so that \rf{eqnormal*} holds for $r>0$ small enough.
Then, according to Theorem \ref{teo-jump}, for $\HH^n$-a.e.\ $x\in F'$ (and thus for $\omega^+$-a.e. and $\omega^-$-a.e.\ $x\in F'$) and any signed Radon measure $\nu$, the non-tangential limits $\RR^+\nu(x)$ and $\RR^-\nu(x)$ defined in \rf{eqdtt} 
exist and satisfy
\begin{equation}\label{eqariesz}
 \frac12 \bigl(\RR^+\nu(x) + \RR^-\nu(x)\bigr)= \pv \RR\nu(x)\equiv \RR\nu(x),
 \end{equation}
 and
\begin{equation}\label{eqbriesz}
\RR^+\nu(x) - \RR^-\nu(x) = \omega_{n} \frac{d\nu}{d\HH^n|_{F'}}(x)\,N_x,
\end{equation}
 taking into account that $C_K(N_x) = \frac{\omega_{n}}2\,N_x$
for the Riesz kernel. Observe now that, since $F'$ is $n$-rectifiable and has finite
$\HH^n$ measure,
$$\lim_{r\to0} \frac{\HH^n(F'\cap B(x,r))}{c_n r^n} = 1\quad\mbox{ for $\HH^n$-a.e.\ $x\in F'$,}$$
where $c_n$ is the $n$-dimensional volume of the unit ball in $\R^n$.
As a consequence, for $\HH^n$-a.e.\ $x\in F'$,
$$\frac{d\nu}{d\HH^n|_{F'}}(x) = \lim_{r\to0} \frac{\nu(B(x,r))}{\HH^n(F'\cap B(x,r))} = 
\lim_{r\to0} \frac{\nu(B(x,r))}{c_n r^n} =: c_n^{-1}\,\Theta^n(x,\nu),$$
where the last identity consists of the definition of the $n$-dimensional density of $\nu$ at $x$.
So we can rewrite \rf{eqbriesz} as follows, for $\HH^n$-a.e.\ $x\in F'$:
\begin{equation}\label{eqbriesz'}
\RR^+\nu(x) - \RR^-\nu(x) = \frac{\omega_{n}}{c_n}\, \Theta^n(x,\nu)\,N_x.
\end{equation}
Finally, notice that since $F'$ is an arbitrary Borel subset of $F$ with finite $\HH^n$ measure, the identities \rf{eqariesz} and \rf{eqbriesz'} hold for $\HH^n$-a.e.\ $x\in F$, or equivalently for $\omega^+$-a.e.\ and  $\omega^-$-a.e.\ $x\in\partial\Omega^+$.

Observe also that for all $x\in\partial\Omega^+$ which are tangent points for $\partial\Omega^+$ (in particular for $x\in F$), if $\supp\nu\subset
\partial\Omega^+$, then, for all $y\in X_a(x)$ close enough to $x$, and $b\in(0,a)$, we have
$\RR_{b|x-y|}\nu(y)= \RR\nu(y)$. Thus,
\begin{equation}\label{eqbrieszzz}
\RR^+\nu(x) = \lim_{X_a^+(x)\ni y\to x} \RR\nu(y),\qquad \RR^-\nu(x) = \lim_{X_a^-(x)\ni y\to x} \RR\nu(y).
\end{equation}
\vv


\subsection{Riesz transform and rectifiability}

For a signed measure $\nu$ in $\R^{n+1}$, we consider the maximal operator
\begin{equation}\label{eqMN}
\mathcal M_n\nu(x)= \sup_{r>0}\frac{|\nu|(B(x,r))}{r^n}.
\end{equation}
Given an $n$-plane $L\subset \R^{n+1}$, a (positive) measure $\mu$, and a ball $B\subset\R^{n+1}$, we denote
$$\beta_{\mu,1}^L(B) = \frac1{r(B)^n}\int_B\frac{\dist(x,L)}{r(B)}\,d\mu(x).$$
This coefficient measures how close is $\mu$ to the $n$-plane $L$ in the ball $B$.

The following theorem is a consequence of the main result in \cite{GT}. For the precise statement below and the
arguments that show how to deduce this from \cite{GT}, see \cite[Theorem 3.3]{AMT-cpam}.

\begin{theorem} \label{teo-gt}
Let $\mu$ be a Radon measure in $\R^{n+1}$ and $B\subset \R^{n+1}$ a ball with $\mu(B)>0$ so that the following conditions
hold:
\begin{itemize}
\item[(a)] For some constant $C_0>0$, $P_\mu(B) \leq C_0\,\Theta_\mu(B)$.

\item[(b)] There is some $n$-plane $L$ passing through the center of $B$ such that, for some constant $0<\delta_0\ll 1$, $\beta_{\mu,1}^L(B)\leq \delta_0\,\Theta_\mu(B)$.

\item[(c)] For some constant $C_1>0$, there is $G_B\subset B$
such that
$$\MM_n(\chi_{2B}\mu)(x) + \RR_*(\chi_{2B}\,\mu)(x)\leq 
C_1\,\Theta_\mu(B)\quad \mbox{ for all $x\in G_B$}$$
and
$$\mu(B\setminus G_B)\leq \delta_0 \,\mu(B).$$

\item[(d)] For some constant $0<\tau_0\ll1$,
$$\int_{G_B} |\RR\mu - m_{\mu,G_B}(\RR\mu)|^2\,d\mu \leq \tau_0 \,\Theta_\mu(B)^2\mu(B).$$
\end{itemize}

Then there exists some constant $\theta>0$ such that if $\delta_0,\tau_0$ are small enough (with $\theta,\delta_0,\tau_0$ depending on $C_0$ and $C_1$),
then there is a uniformly $n$-rectifiable set $\Gamma\subset\R^{n+1}$ such that
$$\mu(G_B\cap \Gamma)\geq \theta\,\mu(B).$$
The uniform rectifiability constants of $\Gamma$ depend on all the constants above.
\end{theorem}
\vv

\subsection{A $T1$ theorem for Riesz transforms involving suppressed kernels}

The following theorem follows easily from the $Tb$ theorem with suppressed kernels of Nazarov, Treil, 
and Volberg \cite{NTV-prep}.

\begin{theorem}\label{teot1}
Let $\mu$ be a Radon measure in { $\R^{n+1}$}. Let $G\subset \supp\mu$.
Suppose that the $n$-dimensional Riesz transform satisfies
$$
\mathcal M_n\mu(x) + \RR_*\mu(x)\leq C_0
\quad\mbox{ for all $x\in G$.}
$$
Then $\RR_{\mu|_G}$ is bounded in $L^2(\mu|_G)$ with $\|\RR_{\mu|_G}\|_{L^2(\mu|_G)\to L^2(\mu|_G)}\leq
c\,C_0$,
with $c$ depending only on $n$.
\end{theorem}

The precise arguments to reduce this result to the aforementioned theorem of Nazarov, Treil, and
Volberg are quite similar to the ones from the proof of Theorem 3.3 from \cite{AMT-cpam}.
 However, for the
convenience of the reader we show the details. { Further, we remark that the theorem also holds assuming that the ambient space is $\R^d$, with $d\geq n$, instead of $\R^{n+1}$.}

\begin{proof} 
For $p_1,p_2>0$ to be fixed below, consider the sets
$$E^1_{p_1} = \{ x \in \R^{n+1}: \mathcal M_n \mu(x)> p_1C_0\}$$ and
$$E^2_{p_2} = \{ x \in \R^{n+1}: \RR_* \mu(x)>p_2 C_0\}.$$

 For $x\in E^1_{p_1}$, we denote 
$$\rho_1(x) = \sup\bigl\{r>0 : \mu(B(x,r))> p_1C_0\, r^n\}
$$
and for $x\in E^2_{p_2}$,
$$\rho_2(x) =\sup\bigl\{r >0 : |\RR_r\mu(x)|
> p_2C_0\bigr\}.$$
Define
$$H_i = \bigcup_{x\in E^i_{p_i}} B(x,\rho_i(x)),\,\, i=1,2.$$
Note that $H_1$ and $H_2$ are open sets. 

Next we will show that, for $p_1$ and $p_2$ big enough, $H_1 \cup H_2\subset G^c$. 
Notice first that if $y \in H_1$, then there is $x \in  E^1_{p_1}$ so that $y \in B(x,\rho_1(x))$, and so 
$$\mu(B(y, 2\rho_1(x))) \ge \mu(B(x, \rho_1(x))) \ge p_1\,C_0 \rho_1(x)^n =p_1\,C_0  2^{-n}[2\rho_1(x)]^n.$$
We conclude that
$H_1  \subset G^c$, if we choose $p_1$ so that $p_1 > 2^n $. 

We turn our attention to $H_2$. If  $y \in H_2 \setminus H_1$, then there exists $x \in E^2_{p_2}$ so that $y \in B(x, \rho_2(x) )$. We will prove that 
\begin{equation}\label{e:diff.Riesz}
|\RR_{\rho_2(x)} \mu(x) -\RR_{\rho_2(x)} \mu(y) | \leq C_n p_1 C_0,
\end{equation} 
where $C_n >0$ is some absolute constant depending only on the dimension. Indeed, we have that  
\begin{align*}
|\RR_{\rho_2(x)} &\mu(x) -\RR_{\rho_2(x)} \mu(y) | \\
& \leq |\RR_{\rho_2(x)} ( \chi_{ B(y, 2 \rho_2(x) )}\mu )(x) | + |\RR_{\rho_2(x)} ( \chi_{B(y, 2 \rho_2(x))}\mu) (y)| \\
&\quad +|\RR_{\rho_2(x)}( \chi_{\R^{n+1} \setminus B(y, 2 \rho_2(x))} \mu)(x) -\RR_{\rho_2(x)} (\chi_{\R^{n+1} \setminus B(y, 2 \rho_2(x) )}\mu)(y)| \\
& =:I_1+I_2+I_3.
\end{align*}
Notice now that 
$$ I_1+I_2 \leq C \frac{\mu(B(y, 2 \rho_2(x)))}{\rho_2(x)^n} \leq   C2^n p_1 C_0,$$
where the second inequality follows form the fact that $y \not \in H_1$. It just remains to handle $I_3$. To this end, we write
\begin{align*}
I_3 &= |\RR( \chi_{\R^{n+1} \setminus B(y, 2 \rho_2(x))} \mu)(x) -\RR(\chi_{\R^{n+1} \setminus B(y, 2 \rho_2(x) )}\mu)(y)|\\
&\leq \wt C \int_{\R^{n+1} \setminus B(y, 2 \rho_2(x))} \frac{|x-y|}{|z-y|^{n+1}}\,d \mu (z) \\
&\leq \wt C \sum_{j \geq 1} \frac{\rho_2(x)}{(2^j \rho_2(x))^{n+1}} \,\mu(B(y, 2^{j+1}\rho_2(x)))\leq \wt C\, 2^n p_1 C_0,
\end{align*}
where in the last inequality we used that $y \not \in H_1$. This concludes the proof of \eqref{e:diff.Riesz}. Therefore, 
since $|\RR_{\rho_2(x)} \mu(x)| > p_2  C_0$, we have that $H_2 \setminus H_1 \subset 
G^c$, if we choose $p_2$ so that $p_2 - C_n p_1> 1$.

Let $H=H_1 \cup H_2$ and consider the $1$-Lipschitz function
$$\Phi(x) = \dist(x,H^c) \geq \max(\rho_1(x), \rho_2(x)),$$
and the associated ``suppressed kernel''
$$K_\Phi(x,y) = \frac{x-y}{\bigl(|x-y|^2 + \Phi(x)\,\Phi(y)\bigr)^{(n+1)/2}}.$$
We consider the operator $\RR_{\Phi,\mu}$ defined by
$$\RR_{\Phi,\mu}f(x) = \int K_\Phi(x,y)\,f(y)\,d\mu(y),$$
and its $\ve$-truncated version (for $\ve>0$)
$$\RR_{\Phi,\ve,\mu}f(x) = \int_{|x-y|>\ve} K_\Phi(x,y)\,f(y)\,d\mu(y).$$
We also set
$$\RR_{\Phi,*,\mu}f(x) = \sup_{\ve>0} \RR_{\Phi,\ve,\mu}f(x).$$
We say that $\RR_{\Phi,\mu}$ is bounded in $L^2(\mu)$ if the operators $\RR_{\Phi,\ve,\mu}$
are bounded in $L^2(\mu)$ uniformly on $\ve>0$.

We now prove that
\begin{equation}\label{eq:suppRieszbound}
\RR_{\Phi,*,\mu} 1(x) \leq C(p_1,p_2)\,C_0,
\end{equation}
 for all $x\in\R^{n+1}$. To do so, we need the following lemma whose proof can be found in \cite[Lemma 5.5]{Tolsa-llibre}.
 
\begin{lemma}\label{lem:Tolsa-suppress}
 Let $x \in \R^{n+1}$ and $r_2 \geq 0$ so that $\mu(B(x,r)) \leq A_1 r^n$ for $r \geq r_2$ and $|\RR_\ve \mu (x)| \leq A_2$ for $\ve \geq r_2$. If $\Phi(x) \geq r_2$, then  $|\RR_{\Phi,\ve, \mu} 1(x)| \leq C\,A_1+  A_2$ for all $\ve>0$ and some constant $C$ depending only on $n$. 
 \end{lemma}
 
By Lemma \ref{lem:Tolsa-suppress} for $A_1=p_1 C_0$, $A_2= p_2C_0$ and $r_2=\max\{\rho_1(x), \rho_2(x)\}$, we obtain \eqref{eq:suppRieszbound}. We further apply the $Tb$ theorem for suppressed operators by Nazarov, Treil, and Volberg \cite{NTV-prep} (see also Corollary 5.33 in \cite{Tolsa-llibre}) and it follows then that
$\RR_{\Phi,\mu}:L^2(\mu) \to L^2(\mu)$ is bounded  with norm 
$$\|\RR_{\Phi,\mu}\|_{L^2(\mu)\to L^2(\mu)}\lesssim  C_0.$$
Since $\Phi$ vanishes on $G\subset H^c$, we have that $\RR_{\mu|_G}:L^2(\mu|_{G}) \to L^2(\mu|_{G})$ is bounded and $$\|\RR_{\mu|_G}\|_{L^2(\mu|_G)\to L^2(\mu|_G)}\lesssim C_0.$$ 
\end{proof}
\vv


\section{Proof of (a) $\Rightarrow$ (b) in Theorem \ref{teo1}}\label{sec3}

Throughout this section we assume that we are under the assumptions of Theorem \ref{teo1} (a), unless stated otherwise, and we denote
$$h= \frac{d\omega^-}{d\omega^+}.$$
We allow the constants denoted by $c$ or $C$ and other implicit constants in the relation $\lesssim$
to depend on the NTA constants of $\Omega^\pm$, and also on the ratio {$\dist(p^\pm,\partial \Omega^+)/\diam(\partial \Omega^+)$} (recall that $p^\pm$ is the pole for the harmonic measure $\omega^\pm$). Without loss of generality, we can think that $p^\pm$ is deep inside $\Omega^\pm$, {so that
$\dist(p^\pm,\partial\Omega^+)\gg R$, where $R$ is the constant appearing in the $(\delta,R)$-Reifenberg flatness. }

 At the end of the current section we will show how to deduce
Corollary \ref{coro2} from the implication (a) $\Rightarrow$ (b) in Theorem \ref{teo1}.

\subsection{The function $h$ as a Muckenhoupt weight}

\begin{lemma}\label{lem3.1}
Let $\Omega^+\subset\R^{n+1}$ be an NTA domain and let $\Omega^-= \R^{n+1}\setminus \overline{\Omega^+}$.
Denote by $\omega^+$ and $\omega^-$ the respective harmonic measures with poles $p^+\in\Omega^+$ and $p^-\in\Omega^-$.
Suppose also that $\omega^+$ and $\omega^-$ are mutually absolutely continuous and suppose that $h=\frac{d\omega^-}{d\omega^+}$ satisfies
$$\log h  \in \vmo(\omega^+).$$
{Then, for $\ell_0>0$ small enough, every $Q_0\in\wh\DD$ with $\ell(Q_0)\leq \ell_0$ satisfies that  $\chi_{Q_0}h\in A_2(\chi_{Q_0}\omega^+)$}.
\end{lemma}

Let us remark that $\omega^+|_{Q_0}$ is a doubling measure. This follows easily from the fact that $\omega^+$ is doubling and the properties of the lattice $\DD$.

\begin{proof}
Given $f\in L^1_{\rm loc}(\omega^+)$ and $Q\in\wh\DD$,  we write
$$\|{f}\|_{*,Q} := \sup_{P\in \wh\DD: P\subset Q} \avint_P |f(x)-m_{\omega^+,P}f|\, d\omega^+ (x).$$
From the John-Nirenberg inequality we know that
$$\avint_Q \exp\bigg(\dfrac{|f-m_{\omega^+,Q} f|}{C \|{f}\|_{*,{Q}}}\bigg) \, d\omega^+\leq 2$$
for $C$ big enough (see \cite[Theorem 0.4]{Buckley} for instance).

Being $\log h\in \vmo$ is equivalent to
\begin{equation}\label{eqVMOforh}
\avint_P |\log h- m_{\omega^+,P}(\log h)| \, d\omega^+\leq \ve(\ell(P))
\end{equation}
for every $P\in \wh\DD$ with $\ve(\ell)\xrightarrow{\ell\to 0} 0$. In particular, $\|{\log h}\|_{*,Q} \leq \ve(\ell(Q))$ for every $Q\in\wh\DD$.

We want to see that 
$$\left(\avint_{B\cap Q_0} h\,d\omega^+\right)\left( \avint_{B\cap Q_0}  h^{-1} \,d\omega^+\right)\leq C$$
for every ball $B$ centered in $Q_0$, which is equivalent to showing the same inequality with both
integrals over all possible $Q\in\wh\DD$ contained in $Q_0$.

Let us write $a_{Q} := e^{\avint_{Q}\log h\, d\omega^+}$. Applying the John-Nirenberg inequality to $f=\log h$ and $Q\in\wh \DD$ contained in $Q_0$, with {$\ell(Q_0)$} small enough, we get
$$\avint_{Q} e^{|\log (h/a_{Q})|} \, d\omega^+\leq \avint_{Q} e^{\frac{|\log h- \log a_{Q}|}{C \|{\log h}\|_{*,{Q}}}}\, d\omega^+\leq 2, $$
that is
\begin{equation}\label{eqJohnNirembergOmega}
\int_{{Q}\cap \{h\geq a_{Q}\}} \frac{h}{a_{Q}} \, d\omega^+ + \int_{{Q}\cap \{h<a_{Q}\}} \frac{a_{Q}}h \, d\omega^+  \leq 2\omega^+({Q}).
\end{equation}
In particular, we obtain
\begin{align*}
 \omega^-({Q}) 
 	& =\int_{Q}  h \,d\omega^+   \leq   \int_{Q\cap \{h\geq a_{Q}\}} h \, d\omega^+ + \int_{{Q}\cap \{h<a_{Q}\}} h \, d\omega^+ 
	 \leq  3a_{Q} \,\omega^+({Q}).
\end{align*}

On the other hand, Jensen's inequality gives
$$a_{Q}=e^{\avint_{Q}\log h\, d\omega^+}\leq \avint_{Q}  h \,d\omega^+ =\frac{\omega^-({Q})}{\omega^+({ Q})},$$
so
\begin{equation}\label{eqAQisLikeDensity0}
a_{Q}\approx\frac{\omega^-({Q})}{\omega^+({Q})}.
\end{equation}

Estimate \rf{eqJohnNirembergOmega} also gives 
$$\int_{Q}  \frac1h\,d\omega^+   = \int_{{Q}\cap \{h< a_{Q}\}} \frac{1}{h} \, d\omega^+ + \int_{{Q}\cap \{h\geq a_{Q}\}} \frac1h \, d\omega^+  \leq  3\frac{ \omega^+({Q})}{a_{Q}}\approx \frac{\omega^+({ Q})^2}{\omega^-({Q})}.$$
Therefore,
$$\left(\avint_{Q} h\,d\omega^+\right)\left( \avint_{Q}  h^{-1} \,d\omega^+\right)\lesssim \frac{\omega^-({Q})}{\omega^+({ Q})} \frac{\omega^+({Q})}{\omega^-({Q})}= 1,$$
as wished.
\end{proof}

\begin{rem}
The same calculations above hold for any dilation $\Lambda Q$ of $Q\in\wh D$ (with $\Lambda>1$) such that $\Lambda Q\subset Q_0$.
In particular, the coefficient
$$a_{\Lambda Q}=e^{\avint_{\Lambda Q}\log h\, d\omega^+}$$
satisfies
\begin{equation}\label{eqAQisLikeDensity}
a_{\Lambda Q}\approx\frac{\omega^-({\Lambda Q})}{\omega^+({\Lambda Q})},
\end{equation}
as in \rf{eqAQisLikeDensity0}, with constants independent of $\Lambda$.
\end{rem}
\vv


\subsection{The good sets $G_{\Lambda Q}$ and $\wt G_{\Lambda Q}$}
For every $Q\in \wh\DD$, $\Lambda\geq1$, and $\delta_1\in (0,1/2)$, let us define the good set
$$G_{\Lambda Q}=\left\{x\in {\Lambda Q}: \left| \frac{h(x)}{a_{\Lambda Q}}-1\right|\leq {\delta_1} \right\}.$$
We remark that below we will fix $\Lambda$ big enough, and then $\delta_1$ small enough depending on $\Lambda$ and other parameters. 

As $\ell({\Lambda Q})\to 0$ it turns out that $\frac{\omega^+(G_{\Lambda Q})}{\omega^+(\Lambda Q)}\to 1$.
Indeed, by Chebyshev's inequality, for any ${\delta_1}\leq 1/2$ we get
\begin{align*}
\omega^+({\Lambda Q}\setminus G_{\Lambda Q})
	& =\omega^+ \left\{x\in {\Lambda Q}: \left| \frac{h(x)}{a_{\Lambda Q}}-1\right|> {\delta_1} \right\}\\
		&\leq \omega^+ \left\{x\in {\Lambda Q}: \left| \log \frac{h(x)}{a_{\Lambda Q}} \right|> {\delta_1}/2 \right\} \leq \frac 2{{\delta_1}} \int_{\Lambda Q} \left| \log \frac{h(x)}{a_{\Lambda Q}} \right|\, d\omega^+.
\end{align*}
and using  Jensen's inequality and \rf{eqVMOforh}  we get that
\begin{align}\label{eqSmallComplement}
\frac{\omega^+({\Lambda Q}\setminus G_{\Lambda Q})}{\omega^+({\Lambda Q})}
	& \leq \frac 2{\delta_1}\avint_{\Lambda Q} \left| \log \frac{h(x)}{a_{\Lambda Q}} \right|\, d\omega^+
		\leq \frac{2}{\delta_1}\ve(C\ell({\Lambda Q})).
\end{align}
\vv

\begin{lemma}\label{lemReverseHolderH}
There exists $\ell_1(\delta_1)>0$ small enough such that if
$\ell({\Lambda Q})\leq\ell_1(\delta_1)$,  then
\begin{equation}\label{eql1*}
\int_{\Lambda Q} |h-a_{\Lambda Q}|\, d\omega^+\lesssim \delta_1 \omega^-({\Lambda Q}).
\end{equation}
\end{lemma}

\begin{proof}
Using the definition of $G_{\Lambda Q}$, \rf{eqAQisLikeDensity} and  \rf{eqSmallComplement} we get
\begin{align*}
\int_{\Lambda Q} |h-a_{\Lambda Q}|\, d\omega^+
	& =\int_{G_{\Lambda Q}} |h-a_{\Lambda Q}|\, d\omega^++\int_{{\Lambda Q}\setminus G_{\Lambda Q}} |h-a_{\Lambda Q}|\, d\omega^+\\
	& \leq {\delta_1} a_{\Lambda Q}\, \omega^+({\Lambda Q}) + a_{\Lambda Q} \,\omega^+({\Lambda Q}\setminus G_{\Lambda Q}) + \int_{{\Lambda Q}\setminus G_{\Lambda Q}} h\, d\omega^+\\
	&\lesssim \left( {\delta_1} +  \frac2{\delta_1} \ve(C\ell({\Lambda Q}))\right) \omega^-({\Lambda Q})   + \int_{{\Lambda Q}\setminus G_{\Lambda Q}} h\, d\omega^+.	
\end{align*}
To control the last term, we recall that $h|_{Q_0}$ is a local $A_2$ weight for any $Q_0\in\wh \DD$ with small 
enough side length, and so $h|_{\Lambda Q}$ satisfies a reverse H\"older inequality with exponent $p$ (depending on the $A_2$ constant) 
if $\ell(\Lambda Q)$ is small enough too. Using also \rf{eqSmallComplement} we obtain
\begin{align*}
\int_{{\Lambda Q}\setminus G_{\Lambda Q}} h\, d\omega^+	
	& \leq \omega^+( {\Lambda Q}\setminus G_{\Lambda Q})^{\frac1{p'}} \left(\avint_{{\Lambda Q}} h^p\, d\omega^+	\right)^\frac1p \omega^+({\Lambda Q})^\frac1p\\
	& \lesssim \left(\frac{2}{\delta_1}\ve(\ell({C\Lambda Q}))\right)^{\frac1{p'}} \omega^+({\Lambda Q})  \avint_{\Lambda Q} h\, d\omega^+
		 \lesssim \delta_1 \omega^-({\Lambda Q}),
\end{align*}
assuming $\ell({\Lambda Q})$ small enough for the last inequality.
\end{proof}
\vv

\begin{rem}\label{rem3.4}
Applying the preceding lemma to $2\Lambda Q$, we infer that
$$\int_{\Lambda Q} |h-a_{2\Lambda Q}|\, d\omega^+\leq \int_{2\Lambda Q} |h-a_{2\Lambda Q}|\, d\omega^+\lesssim \delta_1 \omega^-({2\Lambda Q}) \approx \delta_1 \omega^-({\Lambda Q}).$$
Together with \rf{eql1*}, this implies that
\begin{equation}\label{eql1**}
|a_{\Lambda Q}- a_{2\Lambda Q}|\leq C\delta_1a_{\Lambda Q}.
\end{equation}
\end{rem}
\vv

Next we consider the set
\begin{equation}\label{eqglambdaq}
\widetilde{G}_{\Lambda Q}=\left\{x\in {\Lambda Q}: M_{\omega^+}(h\chi_{2{\Lambda Q}})(x)\leq 2 a_{\Lambda Q}\mbox{ and }  M_{\omega^+}(\chi_{2{\Lambda Q}\setminus G_{2{\Lambda Q}}})(x)\leq \frac12 \right\}.
\end{equation}
Here we denoted by $M_{\omega^+}$ the centered maximal Hardy-Littlewood operator
$$M_{\omega^+}f(x) = \sup_{r>0} \frac1{\omega^+(B(x,r))}\int_{B(x,r)}|f|\,d\omega^+.$$

\begin{lemma}\label{lemgq}
Let $0<{\delta_1}<1/2$. If $\ell({\Lambda Q})\leq \ell_1({\delta_1})$, then
\begin{equation}\label{eq*843}
\omega^+(\widetilde{G}_{\Lambda Q})\geq (1-  \ve_0)\,\omega^+({\Lambda Q}),
\end{equation}
where $\ve_0= C\delta_1$.
\end{lemma}
Notice that $\widetilde{G}_{\Lambda Q}$ depends on $\delta_1$ because of the dependence of ${G}_{2\Lambda Q}$ on $\delta_1$. Further, abusing notation we allow the constant $\ell_1(\delta_1)$
to be smaller than the analogous constant in the previous appearances.

\begin{proof} 
We will estimate $\omega^+(\Lambda Q\setminus \wt G_{\Lambda Q})$. First, note that by the
weak $(1,1)$ inequality for the maximal operator, Lemma \ref{lemReverseHolderH}, and 
\rf{eql1**}, we get 
\begin{align*}
\omega^+\big(\left\{ x: M_{\omega^+}(h\chi_{2{\Lambda Q}})(x)> 2 a_{\Lambda Q}\right\}\big)
	& \leq \omega^+\big(\left\{ M_{\omega^+}((h-a_{\Lambda Q})\chi_{2{\Lambda Q}})(x)>  a_{\Lambda Q}\right\}\big)\\
	& \lesssim\frac{\| (h-a_{\Lambda Q} )\chi_{2{\Lambda Q}}\|_{L^1( \omega^+)}}{a_{\Lambda Q}} \\
	& \lesssim\frac{\| (h-a_{2\Lambda Q} )\chi_{2{\Lambda Q}}\|_{L^1( \omega^+)} + |a_{\Lambda Q} 
	- a_{2\Lambda Q}|\,\omega^+(2\Lambda Q)}{a_{\Lambda Q}} \\
		& \lesssim \delta_1 \frac{\omega^-(2{\Lambda Q})}{a_{{\Lambda Q}}}\approx \delta_1\,\omega^+(\Lambda Q).
\end{align*}
On the other hand, again by the weak $(1,1)$ boundedness of $M_{\omega^+}$ and by \rf{eqSmallComplement} we obtain
\begin{align*}
\omega^+\big(\big\{ x: M_{\omega^+}(\chi_{2{\Lambda Q}\setminus G_{2{\Lambda Q}}})(x)> \tfrac12 \big\}\big)
	& \lesssim \| \chi_{2{\Lambda Q}\setminus G_{2{\Lambda Q}}} \|_{L^1( \omega^+)}
		={ \omega^+(2{\Lambda Q}\setminus G_{2{\Lambda Q}})}\\
	& \lesssim \frac{\ve(C\ell(2{\Lambda Q}))\,\omega^+(2{\Lambda Q})  }{\delta_1}.
\end{align*}
Combining the previous estimates we get, for ${\delta_1}$ small enough,
\begin{align*}
\omega^+(\widetilde{G}_{\Lambda Q})
	& \geq \omega^+({\Lambda Q}) - \omega^+\big(\left\{ x: M_{\omega^+}(h\chi_{2{\Lambda Q}})(x) > 2 a_{\Lambda Q}\right\}\big)- \omega^+ \big(\big\{ x: M_{\omega^+}(\chi_{2{\Lambda Q}\setminus G_{2{\Lambda Q}}})(x)> \tfrac12 \big\}\big) \\ 
	& \geq\left(1-c\delta_1- \frac{C \ve(C\ell({\Lambda Q}))}{{\delta_1}}\right)  \omega^+({\Lambda Q}).
\end{align*}
Choosing for instance $\ell_1({\delta_1})$ so that $C\ve(C\ell_1)\leq ({{\delta_1}})^2$, then for every ${\Lambda Q}$ with $\ell({\Lambda Q})\leq \ell_1$ we derive
\begin{align*}
\omega^+(\widetilde{G}_{\Lambda Q})
	& \geq \left(1-C\delta_1\right)  \omega^+({\Lambda Q}).
\end{align*}
\end{proof}
\vv


\subsection{The Riesz transform of $\omega^+$}

In this section we will estimate the oscillation of the Riesz transforms $\RR^+ \omega^+$ and $\RR \omega^+$ on $\wt G(\Lambda Q)$. First we need to prove a few auxiliary results.

\begin{lemma}\label{lem44}
Let  ${\delta_1}\in (0,1/2)$. For every $x\in \widetilde{G}_{\Lambda Q}$  with $\ell({\Lambda Q})\leq \ell_1(\delta_1)$, and for $B=B(x,r)$ with $0<r<\ell({\Lambda Q})$ we have 
\begin{itemize}
\item[(a)] $\displaystyle \frac{\omega^-({\Lambda Q})}{\omega^+({\Lambda Q})} \approx \frac{\omega^-(B)}{\omega^+(B)} $.
\item[(b)] $\displaystyle \frac{\omega^s(B)}{r^n}\leq C\frac{\omega^s({\Lambda Q})}{\ell({\Lambda Q})^n}$ for $s\in\{+,-\}$.
\end{itemize}
\end{lemma}

\begin{proof}
Let $x\in \widetilde{G}_{\Lambda Q}$.
By the definition of $\widetilde{G}_{\Lambda Q}$ and \rf{eqAQisLikeDensity}, we get
$$\frac{\omega^-(B)}{\omega^+(B)}=\avint_B h\,d\omega^+
\leq M_{\omega^+}(h\chi_{2{\Lambda Q}})(x)\leq 2 a_{\Lambda Q}\approx  \frac{\omega^-({\Lambda Q})}{\omega^+({\Lambda Q})}.$$
Also by the definition of $G_{2{\Lambda Q}}$ and \rf{eql1**},
$$\frac{\omega^-(B)}{\omega^+(B)}\geq \frac{1}{\omega^+(B)}\int_{B\cap G_{2{\Lambda Q}}} h\, d\omega^+\geq a_{2\Lambda Q} \frac{\omega^+ (B\cap G_{2{\Lambda Q}})}{2\omega^+(B)}  $$
Using the fact that for $x\in \widetilde{G}_{\Lambda Q}$ we have $M_{\omega^+}(\chi_{2{\Lambda Q}\setminus G_{2{\Lambda Q}}})(x)\leq \frac12$, we get that 
$$\frac{\omega^-(B)}{\omega^+(B)} \geq \frac{a_{2\Lambda Q}}2 \bigg(1-\frac12\bigg) \approx \frac{\omega^-({\Lambda Q})}{\omega^+({\Lambda Q})},$$
proving {(a)}.

On the other hand, let $g^{\pm}$ be the Green function of $\Omega^{\pm}$ with pole at $p^\pm$ and
consider the functional
$$\gamma(x,r):=\prod_{s\in\{+,-\}} \frac1{r^2} \left(\int_{B(x,r)} \frac{|\nabla g^s(y)|{^2}}{|x-y|^{n-1}}\, dm(y)\right) $$
By the well known Alt-Caffarelli-Friedman monotonicity formula \cite[Lemma 5.1]{ACF84}, this functional is non-decreasing in $r$. Therefore, $\gamma(x,r)\leq \gamma(x,\ell({\Lambda Q}))$. Moreover, by \cite[Theorem 3.13]{KPT09} we have that
\begin{equation}\label{eqGammaComparableHarmonic}
\gamma(x,r)^\frac12 \approx \prod_{s\in\{+,-\}} \frac{\omega^s(B(x,r))}{r^n} = \frac{\omega^-(B(x,r))}{\omega^+(B(x,r))}\frac{\omega^+(B(x,r))^2}{r^{2n}}.
\end{equation}
Combining this estimate with the monotonicity of $\gamma$ we get
$$
 \frac{\omega^-(B(x,r))}{\omega^+(B(x,r))}\frac{\omega^+(B(x,r))^2}{r^{2n}}
 	\lesssim{ \frac{\omega^-({\Lambda Q})}{\omega^+({\Lambda Q})} }\frac{\omega^+({\Lambda Q})^2}{\ell({\Lambda Q})^{2n}}.
$$
Together with (a), this gives (b) for $s=+$. The case $s=-$ follows analogously.
\end{proof}
\vv

{
\begin{lemma}\label{lemfacil}
For some $\delta>0$ and {$R>0$}, let $\Omega\subset\R^{n+1}$ be a bounded $(\delta,R)$-Reifenberg flat domain with bounded boundary and $\omega$ its associated harmonic measure with pole in $p\in\Omega$ such that {$R\leq\dist(p,\partial\Omega)\leq \diam(\Omega)$.}
Given $\gamma_0>0$  {and $A>1$ if $M=M(\gamma_0)$ is big enough and if $\delta =\delta(\gamma_0,M)$ is small enough, every $x\in\partial\Omega$ satisfies} 
\begin{equation}\label{eq0384}
C^{-1}A^{n(1-\gamma_0)}\omega(B(x,r))\leq  \omega(B(x,Ar)) \le CA^{n(1+\gamma_0)}\omega(B(x,r))\quad\mbox{ for all $r\leq (AM)^{-1} R$.}
\end{equation}
Also, if $\delta$ is small enough, then there exists some constant $\wt M$ depending on $\diam(\partial\Omega)/R$
such that
$$P_\omega(B(x,r))\leq C\,\Theta_\omega(B(x,r))\quad \quad\mbox{ for all $0<r\leq \wt M^{-1} R$},$$
where $C$ is some absolute constant.
\end{lemma}

\begin{proof}
The first assertion is an immediate consequence of \cite[Theorem 4.1]{Kenig-Toro-duke}. 
Indeed, this theorem asserts that, given any $\beta>0$, there exists $M'>1$ big enough and $\delta>0$ small enough such that for any $(\delta,R)$-Reifenberg flat domain $\Omega\subset\R^{n+1}$, its associated
harmonic measure $\omega=\omega^p$ satisfies
$$(1-\beta)2^{n}\omega(B(x,s))\leq \omega(B(x,2s))\leq (1+\beta)2^n\omega(B(x,s))\quad\mbox{ for all $x\in\partial \Omega$, $0<s\leq R/M'$},$$
assuming also $|p-x|> M's$.
Then, letting $s=2^k r$ for some $k\ge1$, we deduce that
$$(1-\beta)2^{n}\omega(B(x,2^{k-1}r))\leq \omega(B(x,2^{k} r))\leq (1+\beta)2^n\omega(B(x,2^{k-1}r))\quad\mbox{ for $x\in\partial \Omega$, $0<2^kr\leq R/M'$}.$$
As a consequence, by iteration,
$$(1-\beta)^k2^{nk}\omega(B(x,r))\leq \omega(B(x,2^{k} r))\leq (1+\beta)^k2^{nk}\omega(B(x,r))\quad\mbox{ for  $x\in\partial \Omega$, $0<2^kr\leq R/M'$}.$$
Hence, in the particular case $A=2^k$, we infer that
$$ (1-\beta)^{k}A^n \omega(B(x,r))\leq \omega(B(x,A r))\leq (1+\beta)^k A^n\omega(B(x,r))\quad\mbox{ for  $x\in\partial \Omega$, $0<2^kr\leq R/M'$},$$ which clearly implies \rf{eq0384} {with $C=1$ by choosing $\beta=1-2^{-n\gamma_0}$}. The case when $A\neq2^k$ for any $k$ follows also from the previous estimate.

To prove the second statement of the lemma, assuming $\delta$ small enough, we take $M=M_0$ big enough such that \rf{eq0384} holds with $A=2$ and $\gamma_0=1/(2n)$ (so the smallness of $\delta$ and $M_0$ depend just on $n$).
Next let $n_0$ be the largest integer such that $2^{n_0}r\leq R/M_0$. Then, writing $B=B(x,r)$, we have that
\begin{align}\label{eqpois45}
P_{\omega}(B) =\sum_{j\geq 0} 2^{-j} \Theta_\omega(2^jB) & \leq \sum_{0\leq j< n_0} 2^{-j}\frac{C 2^{j n(1+\gamma_0)}\omega(B)}{2^{jn} r^n} + \sum_{j\geq n_0} 2^{-j}\Theta_\omega(2^jB)\\
& \leq C
 \sum_{j\geq 0} 2^{j(\frac12-1)} \Theta_\omega(B) + 2^{-n_0}\sup_{j\geq n_0} \Theta_\omega(2^jB)\nonumber\\
& \leq C  \Theta_\omega(B) + 2^{-n_0}\sup_{j\geq n_0} \Theta_\omega(2^jB).\nonumber
\end{align}
To estimate the last term, for each $j \geq n_0$ we write
$$\Theta_\omega(2^jB) \leq \frac{\omega(2^jB)}{r(2^{n_0}B)^n}\leq \frac{1}{r(2^{n_0}B)^n}{\approx}
\frac{\omega(\partial\Omega)}{r(2^{n_0}B)^n}.
$$
Observe now that, { since $\Omega$ is an NTA domain whose NTA character depends on $\diam(\partial\Omega)/R$ (assuming $\delta$ small enough),} 
$$\omega(\partial\Omega)\leq C({\diam(\partial\Omega)/R,}\,\diam(\partial \Omega)/2^{n_0}{r})\,\omega(2^{n_0}B)
\leq C'(\diam(\partial \Omega)/R)\,\omega(2^{n_0}B).$$
Hence,
\begin{align*}
2^{-n_0}\sup_{j> n_0} \,\Theta_\omega(2^jB) & \leq 2^{-n_0}\,C'(\diam(\partial \Omega)/R)\,\frac{\omega(2^{n_0}B)}{r(2^{n_0}B)^n} \\
&\lesssim 2^{n_0(n\gamma_0-1)} \,C'(\diam(\partial \Omega)/R)\,\Theta_\omega(B)  = 2^{-n_0/2} \,C'(\diam(\partial \Omega)/R)\,\Theta_\omega(B). 
\end{align*}
So, for $n_0$ big enough depending on $\diam(\partial \Omega)/R$, the right hand side above is at most 
$ 2\Theta_\omega(B)$, which together with \rf{eqpois45} gives
$$P_{\omega}(B) \lesssim \Theta_\omega(B)$$
for $n_0$ big enough depending on $\diam(\partial \Omega)/R$. 
\end{proof}
}

\vv

\begin{rem}
For definiteness, we will assume that the Reifenberg flat constant $\delta$ of $\Omega^\pm$ is small enough so that
we can take $n \gamma_0\leq1/2$ in the preceding lemma.
\end{rem}
\vv

\begin{rem}\label{rem17}
If $x\in\Omega^-$, then $K(x-\cdot) = \frac{x-\cdot}{|x-\cdot|^{n+1}}$ is harmonic in $\Omega^+$ and continuous 
in its closure. Hence, by the definition of harmonic measure,
$$\RR\omega^+(x) = K(x-p^+)\quad\mbox{ if $x\in\Omega^-$}.$$
Analogously, 
$$\RR\omega^-(x) = K(x-p^-)\quad\mbox{ if $x\in\Omega^+$}.$$
Hence, for any tangent point $x\in\partial\Omega^+$,
\begin{equation}\label{eqriesz79}
\RR^-\omega^+(x) = K(x-p^+)\quad\mbox{ and }\quad\RR^+\omega^-(x) = K(x-p^-).
\end{equation}
\end{rem}
\vv

\begin{lemma}\label{lem37}
Let $\Lambda\geq1$ and assume $\ell(\Lambda Q)\leq\ell_1(\delta_1)$. Then we have
\begin{equation}\label{eqpoin1}
\mathcal{M}_n(\chi_{2\Lambda Q} \omega^+)(x) + \RR_*(\chi_{2\Lambda Q}\,\omega^+)(x)\lesssim \Theta_{\omega^+}(\Lambda Q)
\quad\mbox{ for all $x\in  \wt G_{\Lambda Q}$.}
\end{equation}
The analogous estimate also holds replacing $\omega^+$ by $\omega^-$.
\end{lemma}

\begin{proof}
By Lemma \ref{lem44} we have
$$\sup_{0<r\leq \Lambda \ell(Q)} \frac{\omega^+(B(x,r))}{r^n} \lesssim \Theta_{\omega^+}(\Lambda Q)\quad\mbox{ for all $x\in \wt G_{\Lambda Q}$.}$$
So we only need to estimate $\RR_*(\chi_{2\Lambda Q}\,\omega^+)(x)$.
To this end, given $0<\ve\leq \Lambda \ell(Q)$, consider a point $x'_\ve\in B(x,\ve)\setminus \Omega^+$ such that $\dist(x'_\ve,\partial\Omega^+)\approx\ve$ (this point exists because of the exterior corkscrew condition of $\Omega^+$). 
Consider also the analogous point $x'_{\Lambda\ell(Q)}$.
Then we have, by Remark \ref{rem17},
$$\RR\omega^+(x'_\ve) = K(x'_\ve - p^+) \quad \mbox{ and }\quad \RR\omega^+(x'_{\Lambda\ell(Q)}) = K(x'_{\Lambda\ell(Q)} - p^+).
$$
By standard Calder\'on-Zygmund estimates, we have
$$|\RR_\ve\omega^+(x) - \RR\omega^+(x'_\ve)|\lesssim P_{\omega^+}(B(x,\ve)) $$
and by Lemma \ref{lemfacil} we get
$$ 
|\RR_\ve(\chi_{(2\Lambda Q)^c}\omega^+)(x'_{\Lambda\ell(Q)}) - \RR\omega^+(x'_{\Lambda\ell(Q)})|\lesssim P_{\omega^+}(\Lambda Q)\lesssim \Theta_{\omega^+}(\Lambda Q).$$
It is also easy to check that 
$$P_{\omega^+}(B(x,\ve)) \lesssim \mathcal{M}_n(\chi_{2\Lambda Q}\omega^+)(x) + P_{\omega^+}(\Lambda Q) \lesssim \Theta_{\omega^+}(\Lambda Q).$$
So we deduce that
\begin{align*}
|\RR_\ve(\chi_{2\Lambda Q}\omega^+)(x)| & = |\RR_\ve\omega^+(x)-\RR_\ve(\chi_{(2\Lambda Q)^c}\omega^+)(x)|\\
& {\lesssim} |K(x'_\ve - p^+) - K(x'_{\Lambda\ell(Q)} - p^+)| +  \Theta_{\omega^+}(\Lambda Q) \\&\lesssim \frac{\ell(\Lambda Q)}{\dist(p^+,\partial\Omega^+)^{n+1}}
+ \Theta_{\omega^+}(\Lambda Q).
\end{align*}
Observe now that, by the first statement in Lemma \ref{lemfacil} (with $ n\gamma_0=1/2$),
\begin{align*}
\frac{\ell(\Lambda Q)}{\dist(p^+,\partial\Omega^+)^{n+1}} &
\approx \frac{\ell(\Lambda Q)\,}{\diam(\partial\Omega^+)^{n+1}} \,\omega^+(\partial\Omega^+)\\
& \lesssim \frac{\ell(\Lambda Q)\,}{\diam(\partial\Omega^+)^{n+1}} \,\frac{\diam(\partial\Omega^+)^{n+1/2}}{\ell(\Lambda Q)^{n+1/2}}\,\omega^+(\Lambda Q)=
\frac{\ell(\Lambda Q)^{1/2}}{\diam(\partial\Omega^+)^{1/2}}\,\Theta_{\omega^+}(\Lambda Q).
\end{align*}
From the preceding estimates we infer that
$$|\RR_\ve(\chi_{2\Lambda Q}\omega^+)(x)|\lesssim \Theta_{\omega^+}(\Lambda Q)\quad\mbox{ for $x\in \wt G_{\Lambda Q}$ and $0<\ve\leq \Lambda\ell(Q)$.}$$
Using also that $\RR_\ve(\chi_{2\Lambda Q}\omega^+)(x)=0$ for all $x\in \Lambda Q$ when $\ve>\diam(2\Lambda Q)$, we
deduce that $\RR_*(\chi_{2\Lambda Q}\,\omega^+)(x)\lesssim \Theta_{\omega^+}(\Lambda Q)$
for all $x\in  \wt G_{\Lambda Q}$, and the proof of the lemma is concluded.
\end{proof}

\vv

\begin{lemma}\label{lemROmegaRestricted}
Let $\Lambda\geq1$ and and assume $\ell(\Lambda Q)\leq\ell_1(\delta_1)$. 
The operator $\RR_{\omega^+}$ is bounded in $L^2(\omega^+|_{\wt G_{\Lambda Q}})$ with 
$$\|\RR_{\omega^+}\|_{L^2(\omega^+|_{\wt G_{\Lambda Q}})\to L^2(\omega^+|_{\wt G_{\Lambda Q}})}\lesssim\Theta_{\omega^+}(\Lambda Q).$$
The analogous statement also holds replacing $\omega^+$ by $\omega^-$.
\end{lemma}

\begin{proof}
This is an immediate consequence of Lemma \ref{lem37} and Theorem \ref{teot1} applied to $\mu=
\omega^+|_{2\Lambda Q}$ and $G= \wt G_{\Lambda Q}$.
\end{proof}

\vv
\begin{lemma}\label{lem-normal}
Given $\ve'>0$,
assume that $\Lambda$ is big enough (depending on $\ve'$), $\delta_1$ small enough (depending {also} on $\Lambda$), and suppose that $\ell(Q)\leq\ell_2(\delta_1,\Lambda,\ve')$. 
 Then
$$\int_{Q\cap \wt G_{\Lambda Q}} \big|\Theta^n(x,\omega^+) N(x) - m_{\omega^+,Q\cap\wt G_{\Lambda Q}}(\Theta^n(\cdot,\omega^+) N(\cdot))\big|^2\,d\omega^+(x) \leq \ve'\,\Theta_{\omega^+}(Q)^2\,\omega^+(Q).$$
\end{lemma}

\begin{proof}
Recall that, by \rf{eqbriesz'}, for $\omega^+$-a.e.\ $x$,
\begin{equation}\label{eqbriesz99}
\RR^+\omega^+(x) - \RR^-\omega^+(x) = c_n' \Theta^n(x,\omega^+)N(x),
\end{equation}
for some absolute constant $c_n'>0$. Notice that the second identity in \rf{eqriesz79}
implies that
\begin{equation}\label{eqrr98}
\RR^+\!\omega^+(x) = \RR^+\!\omega^+(x) - a_{\Lambda Q}^{-1}\big(\RR^+\!\omega^-(x) - K(x-p^-)\big) = 
\RR^+\big((1-a_{\Lambda Q}^{-1}h)\omega^+\big)(x) + a_{\Lambda Q}^{-1}K(x-p^-).
\end{equation}
Therefore,
\begin{align*}
c_n' \Theta^n(x,\omega^+)N(x) &=\RR^+\omega^+(x) - \RR^-\omega^+(x) \\
& = \RR^+\big((1-a_{\Lambda Q}^{-1}h)\omega^+\big)(x) + a_{\Lambda Q}^{-1}K(x-p^-) - K(x-p^+) =: f.
\end{align*}

Denote
$$m_Q = m_{\omega^+,Q\cap\wt G_{\Lambda Q}}\big(\RR^+\big(\chi_{\R^{n+1}\setminus \Lambda Q}(1-a_{\Lambda Q}^{-1}h)\omega^+\big)\big) + a_{\Lambda Q}^{-1}K(z_Q-p^-) - K(z_Q-p^+).$$
To prove the lemma, it suffices to show that
\begin{equation}\label{eqeve''}
\int_{Q\cap\wt G_{\Lambda Q}} \big|f - m_Q\big|^2\,d\omega^+ \leq \ve'\,\Theta_{\omega^+}(Q)^2\,\omega^+(Q).
\end{equation}
To this end, we split
\begin{align}\label{eqi1-4}
&\int_{Q\cap\wt G_{\Lambda Q}} \big|f - m_Q\big|^2\,d\omega^+ \\ & \quad\lesssim
\int_{Q\cap\wt G_{\Lambda Q}} \big|\RR^+\big(\chi_{\Lambda Q}(1-a_{\Lambda Q}^{-1}h)\omega^+\big)\big|^2\,d\omega^+
\nonumber\\
&\qquad+
\int_{Q\cap\wt G_{\Lambda Q}} \!\big| \RR^+\big(\chi_{\R^{n+1}\setminus \Lambda Q}(1-a_{\Lambda Q}^{-1}h)\omega^+\big) - 
m_{\omega^+,Q\cap\wt G_{\Lambda Q}}\big(\RR^+\big(\chi_{\R^{n+1}\setminus \Lambda Q}(1-a_{\Lambda Q}^{-1}h)\omega^+\big)\big) \big|^2\,d\omega^+\nonumber\\
& \qquad + a_{\Lambda Q}^{-2} 
\int_{Q} \big|K(x-p^-) - K(z_Q-p^-) \big|^2\,d\omega^+(x)\nonumber\\
& \qquad+  
\int_{Q} \big|K(x-p^+) - K(z_Q-p^+) \big|^2\,d\omega^+(x) =: I_1 + I_2+ I_3+ I_4.\nonumber
\end{align}

To estimate the term $I_1$ notice first that, for all $x\in Q\cap\wt G_{\Lambda Q}$, using the jump formulas \rf{eqariesz}, \rf{eqbriesz99}, and Lemma \ref{lem37},
\begin{align*}
\big|\RR^+\big(\chi_{\Lambda Q}&(1-a_{\Lambda Q}^{-1}h)\omega^+\big)(x)\big|  \leq 
\big|\RR^+\big(\chi_{\Lambda Q}\omega^+\big)(x)\big| + a_{\Lambda Q}^{-1}
\big|\RR^+\big(\chi_{\Lambda Q}\omega^-\big)(x)\big|\\
&\leq \big|\RR\big(\chi_{\Lambda Q}\omega^+\big)(x)\big| + \frac{c_n'}2 \Theta^n(x,\omega^+)
+ a_{\Lambda Q}^{-1}
\big|\RR\big(\chi_{\Lambda Q}\omega^-\big)(x)\big| + a_{\Lambda Q}^{-1}\frac{c_n'}2 \Theta^n(x,\omega^-)\\
& \leq \RR_*\big(\chi_{2\Lambda Q}\omega^+\big)(x) + \frac{c_n'}2\,\mathcal M_n(\chi_{2Q}\omega^+)(x) + 
a_{\Lambda Q}^{-1}\RR_*\big(\chi_{2\Lambda Q}\omega^-\big)(x) + 
a_{\Lambda Q}^{-1}\frac{c_n'}2\,\mathcal M_n(\chi_{2Q}\omega^-)(x)\\
& \lesssim  \Theta_{\omega^+}(\Lambda Q) +  a_{\Lambda Q}^{-1}\Theta_{\omega^-}(\Lambda Q)\approx C(\Lambda) \Theta_{\omega^+}(Q).
\end{align*}

Recall also that $\RR_{\omega^+}$ is bounded in $L^2(\omega^+|_{\wt G_{\Lambda Q}})$ with norm at most $C\Theta_{\omega^+}(\Lambda Q)$ (see Lemma \ref{lemROmegaRestricted}), and thus also bounded from {the space of finite signed Radon measures} $M(\R^{n+1})$ to $L^{1,\infty}(\omega^+|_{\wt G_{\Lambda Q}})$ with norm
at most $C'\Theta_{\omega^+}(\Lambda Q)$ (using also that $\mathcal M_n\left(\chi_{\Lambda Q}\omega^+\right)(x)\lesssim
\Theta_{\omega^+}(\Lambda Q)$ in $\wt G_{\Lambda Q}$, see \cite[Theorem 2.16 and Remark 2.17]{Tolsa-llibre}).
Then, given $t>0$, if we let
$$E_t:= \big\{x\in  Q\cap \wt G_{\Lambda Q}: \big|\RR^+\big(\chi_{\Lambda Q}(1-a_{\Lambda Q}^{-1}h)\omega^+\big)\big|>t\,\Theta_{\omega^+}(Q)\big\},$$
using also Lemma \ref{lemReverseHolderH}, we get
\begin{align*}
\omega^+(E_t)& \lesssim\frac{\Theta_{\omega^+}(\Lambda Q)}{t\,\Theta_{\omega^+}(Q)}
\,\|\chi_{\Lambda Q}(1-a_{\Lambda Q}^{-1}h)\|_{L^1(\omega^+)} \leq C(\Lambda)\frac{a_{\Lambda Q}^{-1}}t\,\|\chi_{\Lambda Q}(a_{\Lambda Q}-h)\|_{L^1(\omega^+)}\\
& \leq C(\Lambda)\frac{a_{\Lambda Q}^{-1}
}t\, \delta_1\omega^-(\Lambda Q)
\approx 
\frac{C(\Lambda) \delta_1}t\,\omega^+(Q).
\end{align*}
Then, 
\begin{align*}
I_1  &\leq \int_{Q\cap \wt G_{\Lambda Q}\setminus E_t} \big|\RR^+\big(\chi_{\Lambda Q}(1-a_{\Lambda Q}^{-1}h)\omega^+\big)\big|^2\,d\omega^+ +
\int_{E_t} \big|\RR^+\big(\chi_{\Lambda Q}(1-a_{\Lambda Q}^{-1}h)\omega^+\big)\big|^2\,d\omega^+
\\
& \leq t^2\,\Theta_{\omega^+}(Q)^2\omega^+(Q) +  C(\Lambda)\Theta_{\omega^+}(Q)^2\,\frac{\delta_1}t\,\omega^+(Q).
\end{align*}
Hence, choosing $t=(\delta_1)^{1/2}$, we obtain
$$I_1\lesssim (\delta_1 + C(\Lambda)(\delta_1)^{1/2})\Theta_{\omega^+}(Q)^2\omega^+(Q) \lesssim 
C(\Lambda)(\delta_1)^{1/2}\Theta_{\omega^+}(Q)^2\omega^+(Q).$$

Next we will estimate the term $I_2$ in \rf{eqi1-4}.
To this end, notice that for all $x,x'\in Q$ we have
\begin{align}\label{eqxx'}
\big| \RR^+\big(\chi_{\R^{n+1}\setminus \Lambda Q}(1-a_{\Lambda Q}^{-1}h)\omega^+\big)(x) &\mbox{}-
 \RR^+\big(\chi_{\R^{n+1}\setminus \Lambda Q}(1-a_{\Lambda Q}^{-1}h)\omega^+\big)(x')\big|\\
 &\lesssim 
\int_{\R^{n+1}\setminus \Lambda Q}\frac{|x-x'|}{|x-y|^{n+1}}\,d(\omega^+ + a_{\Lambda Q}^{-1}\omega^-)(y).\nonumber
\end{align}
By Lemma \ref{lemfacil},
\begin{align*}
\int_{\R^{n+1}\setminus \Lambda Q}\frac{\ell(Q)}{|x-y|^{n+1}}\,d\omega^+ & \lesssim \Lambda^{-1}
P_{\omega^+}(\Lambda Q) \lesssim \Lambda^{-1}\Theta_{\omega^+}(\Lambda Q) \lesssim \Lambda^{-1/2}\,\Theta_{\omega^+}(Q) ,
\end{align*}
and analogously replacing $\omega^+$ by $\omega^-$.
Hence,
$$\int_{\R^{n+1}\setminus \Lambda Q}\frac{|x-x'|}{|x-y|^{n+1}}\,d(\omega^+ + a_{\Lambda Q}^{-1}\omega^-)
\lesssim \Lambda^{-1/2}\Theta_{\omega^+}(Q) + a_{\Lambda Q}^{-1}\Lambda^{-1}\Theta_{\omega^-}(\Lambda Q)
\lesssim \Lambda^{-1/2}\Theta_{\omega^+}(Q).$$
Plugging this estimate into \rf{eqxx'} and averaging over all $x'\in Q\cap\wt G_{\Lambda Q}$, we infer that
the integrand in the term $I_2$ is at most $C\Lambda^{-1}\Theta_{\omega^+}(Q)^2$, and thus
$$I_2\lesssim \Lambda^{-1}\Theta_{\omega^+}(Q)^2\omega^+(Q).$$

Now we turn our attention to $I_4$.
Observe that, by Lemma \ref{lemfacil}, for all $x\in Q$,
\begin{align}\label{eqdkj341}
\big|K(x-p^+) - K(z_Q-p^+) \big| & \lesssim \frac{\ell(Q)}{\dist(p^+,\partial\Omega^+)^{n+1}}
\approx \frac{\ell(Q)}{\diam(\partial\Omega^+)^{1/2}}\,\frac{\omega^+(\partial\Omega^+)}{\diam(\partial\Omega^+)^{n+1/2}}\\
& \lesssim \frac{\ell(Q)}{\diam(\partial\Omega^+)^{1/2}}\,\frac{\omega^+(Q)}{\ell(Q)^{n+1/2}} =
\frac{\ell(Q)^{1/2}}{\diam(\partial\Omega^+)^{1/2}}\,\Theta_{\omega^+}(Q).\nonumber
\end{align}
Hence,
$$I_4\lesssim \frac{\ell(Q)}{\diam(\partial\Omega^+)}\,\Theta_{\omega^+}(Q)^2\,\omega^+(Q).$$
For $I_3$ the arguments are similar: as in \rf{eqdkj341}, we have
$$\big|K(x-p^-) - K(z_Q-p^-) \big|\lesssim \frac{\ell(Q)^{1/2}}{\diam(\partial\Omega^+)^{1/2}}\,\Theta_{\omega^-}(Q).$$
Therefore,
$$I_3\lesssim a_{\Lambda Q}^{-2}\frac{\ell(Q)}{\diam(\partial\Omega^+)}\,\Theta_{\omega^-}(Q)^2\,\omega^+(Q)
\lesssim C(\Lambda) \frac{\ell(Q)}{\diam(\partial\Omega^+)}\,\Theta_{\omega^+}(Q)^2\,\omega^+(Q).$$

Gathering the estimates obtained for $I_1,\ldots,I_4$, the estimate \rf{eqeve''} follows, with
$$\ve'\approx C(\Lambda)(\delta_1)^{1/2} + \Lambda^{-1} + \frac{C(\Lambda)\ell(Q)}{\diam(\partial\Omega^+)},$$
which is as small as wished if $\Lambda$ is taken big enough and then $\ell(Q)$ and $\delta_1$
small enough.
\end{proof}

\begin{lemma}\label{lem-pvriesz}
Given $\ve''>0$,
assume that $\Lambda$ is big enough (depending on $\ve''$), $\delta_1$ small enough (depending {also} on $\Lambda$), and suppose that $\ell(Q)\leq\ell_2(\delta_1,\Lambda,\ve'')$. 
Then
$$\int_{Q\cap\wt G_{\Lambda Q}} \big|\RR\omega^+ - m_{\omega^+,Q\cap\wt G_{\Lambda Q}}(\RR\omega^+)\big|^2\,d\omega^+ \leq \ve''\,\Theta_{\omega^+}(Q)^2\,\omega^+(Q).$$
\end{lemma}

We remark that, strictly speaking, the constant $\ell_2$ above may differ from the analogous one in Lemma \ref{lem-normal}.

\begin{proof}
From \rf{eqariesz}, \rf{eqbriesz99}, and the first identity in \rf{eqriesz79} we infer that, for
$\omega^+$ a.e.\ $x\in\partial\Omega^+$,
$$\RR\omega^+(x) - c_n' \Theta^n(x,\omega^+)N(x) = \RR^-\omega^+(x) = K(x-p^+).$$
Therefore, for $\omega^+$ a.e.\ $x,x'\in\partial\Omega^+$,
$$\RR\omega^+(x) - \RR\omega^+(x') = c_n' \Theta^n(x,\omega^+)N(x) - c_n' \Theta^n(x',\omega^+)N(x') + K(x-p^+) -
K(x'-p^+).$$
Averaging for $x'\in Q\cap\wt G_{\Lambda Q}$, we deduce that
\begin{align*}
\RR\omega^+(x) - m_{\omega^+,Q\cap\wt G_{\Lambda Q}}\bigl(\RR\omega^+)  &= c_n' 
\big(\Theta^n(x,\omega^+)N(x) - m_{\omega^+,Q\cap\wt G_{\Lambda Q}}\big(\Theta^n(\cdot,\omega^+)N\big)\big) \\
&\quad+ K(x-p^+) - m_{\omega^+,Q\cap\wt G_{\Lambda Q}}\big(
K(\cdot -p^+)\big).
\end{align*}
From the estimate \rf{eqdkj341}, it follows easily that
$$\big|K(x-p^+) - m_{\omega^+,Q\cap\wt G_{\Lambda Q}}\big(
K(\cdot -p^+)\big)\big| \lesssim
\frac{\ell(Q)^{1/2}}{\diam(\partial\Omega^+)^{1/2}}\,\Theta_{\omega^+}(Q).
$$
Thus,
\begin{align*}
\big|\RR\omega^+(x) - m_{\omega^+,Q\cap\wt G_{\Lambda Q}}\bigl(\RR\omega^+)\big| & \lesssim
\big|\Theta^n(x,\omega^+)N(x) - m_{\omega^+,Q\cap\wt G_{\Lambda Q}}\big(\Theta^n(\cdot,\omega^+)N(\cdot)\big)\big|\\
&\quad+\frac{\ell(Q)^{1/2}}{\diam(\partial\Omega^+)^{1/2}}\,\Theta_{\omega^+}(Q).
\end{align*}
Then, by Lemma \ref{lem-normal},
$$\int_{Q\cap \wt G_{\Lambda Q}} 
\big|\RR\omega^+ - m_{\omega^+,Q\cap\wt G_{\Lambda Q}}\bigl(\RR\omega^+)\big|^2\,d\omega^+
\lesssim \ve'\,\Theta_{\omega^+}(Q)^2\,\omega^+(Q) + \frac{\ell(Q)}{\diam(\partial\Omega^+)}\,\Theta_{\omega^+}(Q)^2\,\omega^+(Q),$$
which proves the lemma, with $\ve''= C\big(\ve' +\frac{\ell(Q)}{\diam(\partial\Omega^+)}\big)$.
\end{proof}

\vv

\subsection{Non-degeneracy of the density of $\omega^+$ in a big piece of $Q\cap\wt G_{\Lambda Q}$}

Consider $Q\in\wh\DD_k$, for some $k\in\Z$, and $G_{\Lambda Q}$, $\wt G_{\Lambda Q}$ as above. 
Observe that if $Q\in \mathcal{ND}_k$ and $Q'$ is any of the cubes from $\DD_k$ that forms $Q$, then
$\Theta_{\omega^+}(Q')\approx\Theta_{\omega^+}(Q)$, by the doubling property of $\omega^+$.

Given $0<\tau\ll1$, we denote by $\LLD_\tau$ the family of maximal cubes $P\in\bigcup_{j>k}\DD_j$ such that $P\subset Q$ and 
$\Theta_{\omega^+}(P) \leq \tau \Theta_{\omega^+}(Q)$. The notation $\LLD$ stands for ``low density".
We also denote $$\LD_\tau = \bigcup_{P\in\LLD_\tau} P.$$

Notice that if $P\in\LLD_\tau$, then
$$\Theta_{\omega^+}(P)\approx \tau\,\Theta_{\omega^+}(Q).$$
Indeed, by definition $\Theta_{\omega^+}(P)\leq \tau\,\Theta_{\omega^+}(Q)$, and by the maximality of $P$, the father $P'\in\DD$ of $P$ satisfies $\Theta_{\omega^+}(P')>\tau \,\Theta_{\omega^+}(Q)$ and, since $\omega^+$ is doubling, 
$\Theta_{\omega^+}(P)\approx \Theta_{\omega^+}(P')$.

\begin{lemma}\label{lemLD}
For all $\ve_1>0$, there exists some constant $\tau=\tau(\ve_1)\in(0,1/10)$ such that, for $\ell(Q)$ small enough,
$$\omega^+(\LD_{\tau})\leq \ve_1\,\omega^+(Q).$$
\end{lemma}

\begin{proof}
We choose $\tau$ of the form $\tau=\lambda^M$, for some $0<\lambda\ll1$ and some integer $M\gg1$ to be fixed
below. {Let us emphasize that $\lambda$ will depend only on $n$ and other fixed parameters, while $M$ will depend on $\tau$ and thus on $\ve_1$.}
For { $k\geq 1$}, we denote
$$\LLD^k = \LLD_{\lambda^k}, \qquad\LD^k =\bigcup_{P\in \LLD^k}P.$$
We also set $\LLD^0=\{Q\}$ and $\LD^0= Q$.
Observe that any cube from $\LLD^k$ is contained in some cube from $\LLD^{k-1}$, and so 
 $\LD^k\subset \LD^{k-1}$.
 The lemma is an easy consequence of the following:
 \begin{claim}\label{claim}
Suppose that $\lambda$ is small enough {depending only on $n$}. Then there exists some $\eta\in (0,1)$ such that for every  {$k\geq 0$} and every $P\in \LLD^k$,
$$\omega^+(P\cap \LD^{k+1})\leq \eta\,\omega^+(P).$$
\end{claim}
From this claim it follows that
$$\omega^+(\LD^{k+1})\leq \eta\,\omega^+(\LD^k),$$
and thus
$$\omega^+(\LD^M)\leq \eta^M\,\omega^+(\LD^0) = \eta^M\omega^+(Q),$$
which proves the lemma if $M$ is big enough.

\vv
To prove the claim above we intend to apply Theorem \ref{teo-gt} to the ball $B=\frac12B_P$ (recall that $B_P$ is the ball associated with $P$ introduced in Theorem \ref{teo-christ}) and
the measure $\omega^+$. First we will check that the assumptions in Theorem \ref{teo-gt} hold.
The second statement in Lemma \ref{lemfacil} ensures that $P_{\omega^+}(B)\leq C\,\Theta_{\omega^+}(B)$ {if $\ell(P)$ is small enough},
and thus the assumption (a) in Theorem \ref{teo-gt} is satisfied. On the other hand, the condition (b) is an immediate consequence of the $\delta$-Reifenberg flatness of $\Omega^+$, assuming $\delta$ small enough.

To check the condition (c) in Theorem \ref{teo-gt}, for some $A>1$ to be fixed below we take $G_B = \wt G_{A P} \cap B$, with $\wt G_{A P}$ defined in \rf{eqglambdaq}
 (with $Q,\Lambda$ replacing $P,A$). 
Observe that Lemma \ref{lemgq} implies that 
$$\omega^+(B\setminus G_B) \leq \omega^+(P\setminus \wt G_{A P}) \lesssim \ve_0\,\omega^+(P) \approx \ve_0\,\omega^+(B).$$
Further, we have
$$|\RR_\ve (\chi_{2 B}\omega^+)(x)|\lesssim |\RR_\ve (\chi_{2A P}\omega^+)(x)| + |\RR_{r(B)}(\chi_{2A P}\omega^+)(x)|
+ \Theta_{\omega^+}(B) \quad\mbox{ for all $x\in B$,}$$
and then it easily follows that
$$\RR_* (\chi_{2 B}\omega^+)(x)\lesssim \RR_* (\chi_{2A P}\omega^+)(x) + \Theta_{\omega^+}(B) \quad\mbox{ for all $x\in B$.}$$
Thus, by Lemma \ref{lem37},  
with $Q=P$, we obtain
$$\mathcal{M}_n(\chi_{2B} \omega^+)(x) + \RR_*(\chi_{2B}\,\omega^+)(x)\lesssim   \Theta_{\omega^+}(B) + \Theta_{\omega^+}(AP)\quad\mbox{ for all $x\in G_B$.}$$
By Lemma \ref{lemfacil}, we have 
$$C^{-1}A^{-\gamma_0 n}\Theta_{\omega^+}(P)\leq \Theta_{\omega^+}(AP)\leq CA^{\gamma_0 n} \Theta_{\omega^+}(P)$$
with $\gamma_0$ as small as wanted if the Reifenberg flat constant is small enough. In particular, if the Reifenberg constant is small enough (depending only on $A$), then $A^{\gamma_0n}=2$ and thus $\Theta_{\omega^+}(AP)\approx \Theta_{\omega^+}(P)\approx \Theta_{\omega^+}(B)$ and so {(c) holds with constant $C_1=2C$ for $C$ as in \rf{eq0384}}.

The last assumption (d) is a direct consequence of Lemma \ref{lem-pvriesz} applied to $P$. Indeed,
\begin{align*}
\int_{G_B} |\RR\omega^+(x) - m_{\omega^+,G_B}(\RR\omega^+)|^2\,d\omega^+(x) 
& \leq \int_{G_B} |\RR\omega^+(x) - m_{\omega^+,P\cap \wt G_{AP}}(\RR\omega^+)|^2\,d\omega^+(x)\\
&\leq \ve''\,\Theta_{\omega^+}(P)^2\,\omega^+(P)\approx \ve''\,\Theta_{\omega^+}(B)^2\,\omega^+(B),
\end{align*}
with $\ve''$ as small as wished (assuming $\ell(Q)$ and $\delta_1$ small enough and $A$ big enough).

\vv
The application of Theorem \ref{teo-gt} ensures the existence of 
a uniformly $n$-rectifiable set $\Gamma\subset\R^{n+1}$ such that
$$\omega^+(G_B\cap \Gamma)\geq \theta\,\omega^+(B),$$
for some fixed $\theta>0$, with the UR constants of $\Gamma$ uniformly bounded. Claim \ref{claim} is
an easy corollary of this fact. Indeed, let $I$ denote the subfamily of cubes from $\LLD^{k+1}$ which 
intersect $G_B\cap \Gamma$ (and thus 
are contained in $P$). Consider a subfamily $J\subset I$ such that
\begin{itemize}
\item the balls $2\wt B_R$, $R\in J$, are pairwise disjoint, and
\item $\bigcup_{R'\in I} {R'}\subset \bigcup_{R\in J} 6\wt B_{R}.$ 
\end{itemize}
Then, using the fact that $\Theta_{\omega^+}(6\wt B_R)\approx\Theta_{\omega^+}(R)
\lesssim \lambda \Theta_{\omega^+}(P)$ for $R\in J$, we get
\begin{equation}\label{eqfif1}
\omega^+(G_B\cap \Gamma\cap \LD^{k+1}) \leq \sum_{R\in J} \omega^+(6\wt B_R) \lesssim
\lambda\,\Theta_{\omega^+}(P)
\sum_{R\in J} \ell(R)^n.
\end{equation}
By the $n$-AD regularity of $\Gamma$ and the fact that $\wt B_R\cap\Gamma \neq\varnothing$ for $R\in J$, we derive
$$\ell(R)^n\approx \HH^n(\Gamma\cap 2\wt B_R),$$
and thus, using the fact that the balls $2\wt B_R$ are disjoint and contained in some fixed multiple of $B_P$,  and the $n$-AD regularity of $\Gamma$ again, we get
$$\sum_{R\in J} \ell(R)^n \approx \sum_{R\in J}\HH^n(\Gamma\cap 2\wt B_R) \lesssim \HH^n(\Gamma\cap C B_P) \approx \ell(P)^n.$$
Plugging this estimate into \rf{eqfif1} and choosing $\lambda\ll\theta$, we obtain
$$\omega^+(G_B\cap \Gamma\cap \LD^{k+1})\leq C \lambda\,\omega^+(P) \leq C' \lambda\,\omega^+(B)\leq \frac\theta2\,\omega^+(B) \leq \frac12 \,\omega^+(G_B\cap \Gamma).$$
Thus,
$$\omega^+(G_B\cap \Gamma\setminus \LD^{k+1})\geq \frac12 \omega^+(G_B\cap \Gamma)\geq \frac\theta2\,\omega^+(B)\approx \theta\,\omega^+(P).$$
In particular, this shows that $\omega^+(P\setminus \LD^{k+1})\gtrsim \theta\,\omega^+(P)$ and proves the claim, and the lemma.
\end{proof}

\vv

Notice that in the argument above the Reifenberg flatness constant asked for $\Omega^+$ does not depend on
$\tau$. Indeed, in the application of Theorem \ref{teo-gt} during the proof of the claim we fixed $C_0$ and $C_1$ fitting the constants appearing in Lemma \ref{lemfacil}, which are universal. Thus, to check that  the assumption $(d)$ of the same theorem is satisfied, we chose $\varepsilon''=\tau_0(C_0,C_1)$ in the notation of that theorem. Then $A=\Lambda(\varepsilon'')$ was fixed according to Lemma \ref{lem-pvriesz}, in terms of these universal constants and the NTA parameters of the domain. This made $\gamma_0=\frac{1}{n\log_2(A)}$ needed in \rf{eq0384} a constant just depending on the NTA parameters and, therefore, the Reifenberg constant $\delta$ also depends on the NTA parameters so that Lemma \ref{lemfacil} and condition $(b)$ of Theorem \ref{teo-gt} can be applied. The parameter $\delta_1$ used to define the good set needs to satisfy $\delta_1\leq\delta_1(\varepsilon'',A)$ when applying Lemma \ref{lem-pvriesz} with $\Lambda=A$ but it also needs to satisfy Lemma \ref{lemgq} with $\delta_1\leq C\delta_0(C_0,C_1)$ so that part $(c)$ in Theorem \ref{teo-gt} could be checked. Finally, the side length of the cube needed to be very small so that Lemmas \ref{lemgq}, \ref{lemfacil}, \ref{lem37} and \ref{lem-pvriesz} can be applied, that is, $\ell(Q)\leq \min\{A^{-1}\ell_1(\delta_1), \wt M^{-1} R,\ell_2(\delta_1,A,\varepsilon'')\}$ (depending on the NTA parameters, on $R/\diam(\partial\Omega^+)$, and on the VMO character of $N_{\Omega^+}$). Note that also $\theta$ and the uniform rectifiability constants are universal, and  $\eta$ depends on the NTA parameters. To end, the NTA parameters can be thought to be uniformly bounded if $\delta$ is small enough by \cite[Theorem 3.1]{Kenig-Toro-duke}. 

This is an important point in our proof, because
to show that 
$$\lim_{\ell(Q)\to0}
\avint_{Q} \big|N - m_{\omega^+,Q}(N)\big|^2\,d\omega^+\to 0,$$
we need to take $\tau\to0$.

\vv


\subsection{End of the proof of (a) $\Rightarrow$ (b) in Theorem \ref{teo1}}

Given two non-zero vectors $u,v\in\R^{n+1}$, we have
\begin{equation}\label{equv0}
\left|\frac u{|u|} - \frac v{|v|}\right| = \frac{\big| |v|(u-v) + v(|v|-|u|)\big|}{|u|\,|v|} 
\leq \frac{|u-v|}{|u|} + 
\frac{\big||v|-|u|\big|}{|u|} \leq 2 \frac{|u-v|}{|u|}.
\end{equation}
Obviously, the same estimate is valid in the case $v=0$, replacing $\frac v{|v|}$ by $0$.
Applying this inequality with $u=\Theta^n(x,\omega^+) N(x)$, $v=m_{\omega^+,Q\cap\wt G_{\Lambda Q}}(\Theta^n(\cdot,\omega^+) N(\cdot))$ for $x\in Q\cap\wt G_{\Lambda Q}\setminus \LD_\tau$, we infer that
\begin{multline*}
\int_{Q\cap \wt G_{\Lambda Q}\setminus \LD_\tau} \big|N(x) - C_Q\big|^2\,d\omega^+(x)\\
\leq 
\int_{Q\cap \wt G_{\Lambda Q}\setminus \LD_\tau} \frac4{\Theta^n(x,\omega^+)^2}\big|\Theta^n(x,\omega^+) N(x) - m_{\omega^+,Q\cap\wt G_{\Lambda Q}}(\Theta^n(\cdot,\omega^+) N(\cdot))\big|^2\,d\omega^+(x),
\end{multline*}
where $C_Q = \frac{m_{\omega^+,Q\cap \wt G_{\Lambda Q}}(\Theta^n(\cdot,\omega^+) N(\cdot))}{|m_{\omega^+,Q\cap \wt G_{\Lambda Q}}(\Theta^n(\cdot,\omega^+) N(\cdot))|}$ if $m_{\omega^+,Q\cap \wt G_{\Lambda Q}}(\Theta^n(\cdot,\omega^+) N(\cdot))\neq 0$ and $C_Q=0$ otherwise. Using the fact that $\Theta^n(x,\omega^+)\geq \tau\,\Theta_{\omega^+}(Q)$ in 
$Q\cap \wt G_{\Lambda Q}\setminus \LD_\tau$ and Lemma \ref{lem-normal}, we obtain
\begin{multline*}
\int_{Q\cap \wt G_{\Lambda Q}\setminus \LD_\tau} \big|N(x) - C_Q\big|^2\,d\omega^+(x) \\
\leq \frac4{\tau^2 \,\Theta_{\omega^+}(Q)^2}
\int_{Q\cap \wt G_{\Lambda Q}} \big|\Theta^n(x,\omega^+) N(x) - m_{\omega^+,Q\cap\wt G_{\Lambda Q}}(\Theta^n(\cdot,\omega^+) N(\cdot))\big|^2\,d\omega^+(x) \leq \frac{4\ve'}{\tau^2}\,\omega^+(Q).
\end{multline*}
Therefore, taking also into account Lemmas \ref{lemgq} and  \ref{lemLD},
\begin{align*}
\int_{Q} \big|N - C_Q\big|^2\,d\omega^+ &  \leq 4\omega^+(Q\setminus (\wt G_{\Lambda Q}\setminus \LD_\tau)) + \int_{Q\cap \wt G_{\Lambda Q}\setminus \LD_\tau}\!\! \big|N - C_Q\big|^2\,d\omega^+\\
& \leq 4\omega^+(Q\setminus \wt G_{\Lambda Q}) + 4 \omega^+(\LD_\tau) + \int_{Q\cap \wt G_{\Lambda Q}\setminus \LD_\tau} \big|N - C_Q\big|^2\,d\omega^+\\
&  \lesssim 4\ve_0\omega^+(Q) + 4 \ve_1\omega^+(Q) + \frac{4\ve'}{\tau(\ve_1)^2}\,\omega^+(Q),
\end{align*}
where $\ve_0= C\delta_1$.
Thus, given any $\ve_2>0$, choosing appropriately the parameters $\ve_0$, $\ve_1$, and $\ve'$, and taking $\ell(Q)$ small enough, we infer that
$$\int_{Q} \big|N - m_{\omega^+,Q}(N)\big|^2\,d\omega^+ \leq  \int_{Q} \big|N - C_Q\big|^2\,d\omega^+ \leq\ve_2\,\omega^+(Q).$$

Given any ball $B$ centered in $\partial\Omega^+$ with small enough radius, there exists some $Q\in\wh \DD$ such 
that 
$$B\cap\partial\Omega^+\subset Q\quad \mbox{ and } \quad\ell(Q)\approx r(B).$$
It follows then that
$$\lim_{r\to 0} \sup_{B:r(B)\leq r} \avint_{B} \big|N - m_{\omega^+,B}(N)\big|^2\,d\omega^+ 
\lesssim \lim_{\ell\to 0} \sup_{Q\in\wh D:\ell(Q)\leq \ell} \avint_{Q} \big|N - m_{\omega^+,Q}(N)\big|^2\,d\omega^+ =0,$$
where all the balls $B$ in the first supremum are assumed to be centered in $\partial\Omega^+$.
So we have  $N\in\vmo(\omega^+)$.
\fiproof

\vv

\subsection{Proof of Corollary \ref{coro2}}\label{seccoro2}

Let $\Omega^+$, $\Omega^-$ be as in Corollary \ref{coro2}. We have to show that $N\in\vmo(\HH^n|_{\partial\Omega^+})$. Since we are assuming that $\Omega^+$ is a chord-arc domain, it follows that
$\omega^+$ is an $A_\infty$ weight with respect to the surface measure $\sigma\equiv\HH^n|_{\partial\Omega^+}$, by
results due independently to David and Jerison \cite{DJ} and to Semmes \cite{Semmes}. 

Consider an arbitrary ball $B$ centered in $\partial\Omega^+$ with $r(B)\leq\diam(\Omega^+)$.
By Theorem \ref{teo1}, we know that
$$\int_{B} \big|N - m_{\omega^+,B}(N)\big|^2\,d\omega^+\leq \ve_3(r(B))\,\omega^+(B),$$
with $\ve_3(r)\to0$ as $r\to0$.
Let
$$E= \big\{x\in B\cap\partial\Omega^+: \big|N(x) - m_{\omega^+,B}(N)\big| >\ve_3(r(B))^{1/4}\big\}.$$
By Chebyshev we deduce that 
$$\omega^+(E)\leq  \ve_3(r(B))^{1/2}\,\omega^+(B).$$
Hence, by the $A_\infty$ property of $\omega^+$,
given an arbitrary $\ve_4>0$, if $r(B)$ is small enough (and thus $\ve_3(r(B))^{1/2}$ small enough), 
then $\sigma(E)\leq \ve_4\,\sigma(B)$. Therefore, 
\begin{align*}
\int_{B} \big|N - m_{\sigma,B}(N)\big|^2\,d\sigma &\leq
\int_{B} \big|N - m_{\omega^+,B}(N)\big|^2\,d\sigma\\
& \leq 4 \sigma(E) + \int_{B\setminus E} \big|N - m_{\omega^+,B}(N)\big|^2\,d\sigma\\
& \leq 4\ve_4\,\sigma(B) + \ve_3(r(B))^{1/2}\,\sigma(B),
\end{align*}
which shows that
$$\lim_{r\to 0} \sup_{B:r(B)\leq r} \avint_{B} \big|N - m_{\sigma,B}(N)\big|^2\,d\sigma 
 =0$$
(with the balls $B$ in the supremum centered in $\partial\Omega^+$),
or equivalently, that $N\in \vmo(\sigma)$.
\fiproof

\vv

\subsection{A final result}

Essentially the same arguments used to prove (a) $\Rightarrow$ (b) in Theorem \ref{teo1} and Corollary \ref{coro2} give the following.

\begin{theorem}\label{teofifi} 
Let $\Omega^+\subset\R^{n+1}$ be a bounded NTA domain and let $\Omega^-= \R^{n+1}\setminus \overline{\Omega^+}$ be an NTA domain as well.
Denote by $\omega^+$ and $\omega^-$ the respective harmonic measures with poles $p^+\in\Omega^+$ and $p^-\in\Omega^-$.
Suppose that $\Omega^+$ is a $\delta$-Reifenberg flat domain, with $\delta>0$ small enough, and 
that $\omega^+$ and $\omega^-$ are mutually absolutely continuous. Let $r_0\in (0,\diam(\Omega^+))$.
For every $\ve>0$ there exists $\eta>0$ depending on $\ve$ and $r_0$ such that if
$$\avint_B \Big|\log\frac{d\omega^-}{d\omega^+}- m_{\omega^+,B}\Big(\log\frac{d\omega^-}{d\omega^+}\Big)
 \Big|\,d\omega^+ \leq \eta$$
for all balls $B$ centered in $\partial\Omega^+$ with radius at most $r_0$,
then the inner normal $N(x)$ exists at $\omega^+$-almost every $x\in \partial\Omega^+$ and
$$\avint_B |N-m_{\omega^+,B}N|\,d\omega^+\leq \ve$$
for any ball $B$ centered in $\partial\Omega^+$ with radius small enough. 

If, additionally, $\partial\Omega^+$ is $n$-AD regular, then also
$$\avint_B |N-m_{\HH^n|_{\partial\Omega^+},B}N|\,d\HH^n|_{\partial\Omega^+} \leq\ve$$
for any ball $B$ centered in $\partial\Omega^+$ with radius small enough. 

\end{theorem}

\vv



\section{Proof of $\text{(c)}\Rightarrow \text{(a)}$ in Theorem \ref{teo1}} \label{secctoa}

We assume that we are under the conditions of Theorem \ref{teo1} (c). So we suppose that
$\Omega^+$ is vanishing Reifenberg flat, the inner normal $N$ belongs to $\vmo(\omega^+)$, and either $\dfrac{d\omega^+}{d\omega^-}\in B_{3/2}(\omega^-)$ or
$\dfrac{d\omega^-}{d\omega^+}\in B_{3/2}(\omega^+)$.
Our objective
is to show that, given $\tau>0$, for any ball $B_0$ centered { in} $\partial\Omega$,
\begin{equation}\label{eqrev1}
\avint_{B_0} \left|\log\frac{d\omega^-}{d\omega^+} - m_{B_0,\omega^+}\Big(\log\frac{d\omega^-}{d\omega^+}\Big)\right|\,d\omega^+ \leq \tau
\end{equation}
if $r(B_0)$ is small enough. 
First we will prove this assuming that $\dfrac{d\omega^+}{d\omega^-}\in B_{3/2}(\omega^-)$, and in the last Section \ref{seclast} we will explain the arguments in the case $\dfrac{d\omega^-}{d\omega^+}\in B_{3/2}(\omega^+)$.

\subsection{Stopping cubes}\label{secstop}
To prove the estimate \rf{eqrev1} we intend to construct some approximation domains for $\Omega^+$ and $\Omega^-$ and apply to them Theorem A from Kenig and Toro. To this end, in this section we need to introduce 
some stopping cubes.  

According to \cite[Theorem 1.3]{AMT-quantcpam},
both $\omega^+$ and $\omega^-$ have very big pieces of uniformly $n$-rectifiable measures. This means
that, for every $\ve\in (0,1)$ and
for every ball $B$ centered { in} $\partial\Omega$ 
with radius at most $\diam(\partial\Omega)$, there exists uniformly $n$-rectifiable sets $E^+,E^-$, with UR constants possibly depending on $\ve$, and subsets $F^\pm\subset E^\pm$ such that
\begin{equation}\label{equr00}
\omega^\pm(B\setminus F^\pm)\leq \ve\,\omega^\pm(B)
\end{equation}
and
\begin{equation}\label{equr01}
\omega^\pm(D)\approx_\ve \HH^n(D)\,\Theta_{\omega^\pm}^n(B)\quad \mbox{ for all $D\subset F^\pm$.}
\end{equation}

We consider the lattice $\DD$ of Christ cubes from $\partial\Omega$.
Given 
 a ball $B_0$ centered { in} $\partial\Omega$ with $r(B_0)\leq \diam(\Omega^+)$ and two parameters $\delta\in(0,1)$ and $A\gg1$ to be chosen below, we
consider some stopping cubes defined as follows: we say that $Q\in \sss(B_0)$ if $Q\in\DD$ is a maximal cube contained in $2B_0$ such that one of the following options holds:
\begin{itemize}
\item $\Theta_{\omega^+}(Q)> A\Theta_{\omega^+}(B_0)$. We write $Q\in\HHD^+(B_0)$.
\item $\Theta_{\omega^+}(Q)\leq \delta\Theta_{\omega^+}(B_0)$.  We write $Q\in\LLD^+(B_0)$.
\end{itemize}
We denote
$$\HD^+(B_0) =\bigcup_{Q\in\HHD^+(B_0)}Q,\qquad \LD^+(B_0) =\bigcup_{Q\in\LLD^+(B_0)}Q$$
and
$$G(B_0)= \tfrac32 B_0 \setminus \bigcup_{Q\in\sss(B_0)} Q.$$

\vv
\begin{lemma}\label{lem4.1}
For any $\ve'>0$, if $A$ is big enough and $\delta$ small enough, then
$$\omega^+(\HD^+(B_0) \cup \LD^+(B_0))\leq \ve'\,\omega^+(B_0).$$
\end{lemma}

Let us remark that $A$ and $\delta$ depend on $\ve'$ and the $A_\infty$ relation between $\omega^+$
and $\omega^-$.

\begin{proof}
First we estimate $\omega^+(\HD^+(B_0))$. Let $F^+$, $E^+$ and $\ve$ be as in \rf{equr00}, with $B$ replaced by $2B_0$ and $\ve$ to be chosen below.
We claim that if $Q\in\HHD^+(B_0)$ and $A$ is big enough, then 
$$\omega^+(Q\cap F^+) < \frac12\,\omega^+(Q).$$
In fact, observe that if $\omega^+(Q\cap F^+) \geq \frac12\,\omega^+(Q)$, then
$$\omega^+(Q)\leq 2\omega^+(Q\cap F^+)\leq C(\ve)\Theta_{\omega^+}(B_0)\,\HH^n(Q\cap F^+)\leq C(\ve)\Theta_{\omega^+}(B_0)\,\ell(Q)^n,$$
and thus 
$$\Theta_{\omega^+}(Q)\leq C(\ve)\Theta_{\omega^+}(B_0).$$
So $Q\not \in\HHD^+(B_0)$ if $A$ is chosen big enough (depending on $\ve$) and the claim follows.
Then we deduce
$$\sum_{Q\in \HHD^+(B_0)}\omega^+(Q)\leq 2\sum_{Q\in \HHD^+(B_0)}\omega^+(Q\setminus F^+)\leq 2\omega^+(2B_0\setminus F^+)\leq C\ve\omega^+(B_0).$$

Next we estimate $\omega^+(\LD^+(B_0))$. We write
$$\sum_{Q\in \LLD^+(B_0)}\omega^+(Q) = \sum_{Q\in \LLD^+(B_0)}\omega^+(Q\setminus F^+)+ \sum_{Q\in \LLD^+(B_0)}\omega^+(Q\cap F^+).$$
The first sum on the right hand side is at most
$$\omega^+(2B_0\setminus F^+)\leq c\ve\omega^+(B_0).$$
To deal with the sum $\sum_{Q\in \LLD^+(B_0)}\omega^+(Q\cap F^+)$, denote by $J$ the family of cubes $Q\in\LLD^+(B_0)$ such that $Q\cap F^+\neq\varnothing$, and consider a subfamily $J_0\subset J$ such that the balls
$2\wt B_Q$ (defined in Theorem \ref{teo-christ}), with $Q\in J_0$, are pairwise disjoint, while
$$\bigcup_{Q\in J} Q\subset \bigcup_{Q\in J_0} 10\wt B_Q.$$
Then, using the doubling property of $\omega^+$, we obtain
$$\sum_{Q\in \LD^+(B_0)}\omega^+(Q\cap F^+) \leq \sum_{Q\in J}\omega^+(Q)\leq  \sum_{Q\in J_0}\omega^+(10\wt B_Q)
\leq C\sum_{Q\in J_0}\omega^+(Q)\leq C\delta\,\Theta_{\omega^+}(B_0)\sum_{Q\in J_0}\ell(Q)^n.$$
Now we take into account that, for $Q\in J_0$, $\wt B_Q\cap E^+\neq\varnothing$, and by the $n$-AD regularity of $E^+$,
$\ell(Q)^n \lesssim_\ve \HH^n(2\wt B_Q\cap E^+)$. Thus, using that the balls $2B_Q$ are disjoint,
\begin{align*}
C\delta\,\Theta_{\omega^+}(B_0)\sum_{Q\in J_0}\ell(Q)^n & \leq C(\ve)\delta\,\Theta_{\omega^+}(B_0)\sum_{Q\in J_0}\HH^n(2\wt B_Q\cap E^+)\\
& \leq C(\ve)\delta\,\Theta_{\omega^+}(B_0)\HH^n(4B_0 \cap E^+)\\&\leq 
 C(\ve)\delta\,\Theta_{\omega^+}(B_0)\,r(B_0)^n =  C(\ve)\delta\,\omega^+(B_0).
 \end{align*}
Therefore,
$$\sum_{Q\in \LLD^+(B_0)}\omega^+(Q) \leq (c\ve + C(\ve)\delta)\omega^+(B_0).$$

Altogether, we have
$$\omega^+(\HD^+(B_0) \cup \LD^+(B_0))\leq (c\ve + C(\ve)\delta)\omega^+(B_0)\lesssim\ve \omega^+(B_0),$$
assuming $\delta=\delta(\ve)$ small enough (and also $A=A(\ve)$ big enough).
\end{proof}
\vv

\subsection{Construction of the approximating domains}\label{sec4.2}

Next we will follow an idea from \cite{AMT-singular}. To this end, first we have to introduce some
Whitney type cubes with restricted size.

Given an open set $V\subsetneq \R^{n+1}$ and $K\geq4$, we denote by $\WW_{K,r_0}(V)$ the set of maximal dyadic cubes $Q\subset V$ such that $\diam KQ\leq r_{0}$ and $K Q\cap V^{c}=\varnothing$. These cubes have disjoint interiors and can be easily shown to satisfy the following properties:
\begin{enumerate}
\item[(a)] $\min\{r_{0},\dist(Q,V^{c})\}/K\lesssim \ell(Q)\lesssim \min\{r_{0},\dist(Q,V^{c})\}/K$, where $\ell(Q)$ denotes the side length of the cube.
\item[(b)]  If $Q,R\in \WW_{K,r_0}(V)$ and $\frac K4 Q\cap \frac K4 R\neq\varnothing$, then $\ell(Q)\approx_{K,n}\ell(R)$.
\item[(c)] $\sum_{Q\in \WW_{K,r_0}(V)}\chi_{\frac K4 Q}\lesssim_{K,n}\chi_{V}$.
\end{enumerate}

Consider the open set $V=\R^{n+1}\setminus \overline{G(B_0)}$ and, for some constant 
$0<\tau_0<1/100$, the associated Whitney cubes $\WW_{\tau_0^{-2},r(B_0)}(V)$. Denote by $\WW_0 $ the family of cubes $Q\in\WW_{\tau_0^{-2},r(B_0)}(V)$ such that
$Q\cap\partial\Omega^+\neq\varnothing$. Notice that
$$\ell(Q)\lesssim \tau_0^2\,\dist(Q,G(B_0))\qquad\mbox{for all $Q\in \WW_0 $}$$
and
$$\partial \Omega^+\setminus \overline{G(B_0)} \subset \bigcup_{S\in \WW_0 } S.$$
For each $S\in \WW_0 $, fix some point $z_S\in S\cap\partial \Omega$ and set
\begin{equation}\label{eqbs*}
B_S = B(z_S,\tau_0\,\min\{r(B_0),\dist(S,G(B_0))\}).
\end{equation}
Notice  that 
$$\ell(S)\approx \tau_0\,r(B_S)\approx \tau_0^2
\min\{r(B_0),\dist(S,G(B_0))\}.
$$
Then we consider the domains
$$\Omega^+_b = \Omega^+ \cup \bigcup_{S\in\WW_0} B_S,\qquad
\Omega^+_s =  \Omega^+ \setminus\overline{ \bigcup_{S\in\WW_0} B_S},$$
$$\Omega^-_b = \Omega^- \cup \bigcup_{S\in\WW_0} B_S,\qquad
\Omega^-_s =  \Omega^- \setminus\overline{ \bigcup_{S\in\WW_0} B_S}.$$
The subindex $b$ stands for ``big'' and $s$ for ``small". { Notice that the domains $\Omega_b^\pm$ are obtained by increasing $\Omega^\pm$, while $\Omega_s^\pm$ are obtained by 
reducing $\Omega^\pm$. Further, by {\cite[Lemma 2.2]{AMT-singular}},
$$\overline{G(B_0)}\subset \partial \Omega^+ \cap \partial \Omega^\pm_s \cap \partial \Omega^\pm_b.$$

Let us remark that the use of approximating interior or exterior domains is not new in potential theory. They can be constructed by different methods. The method we use here has the advantage of being quite straightforward and producing Reifenberg flat domains, as the next lemma shows.}

\begin{lemma}\label{lemreif1}
Let $\rho_0=r(B_0)$, and 
let $\tau_0>0$ be small enough. There exists $\delta_{0}=\delta_{0}(\tau_0)>0$ such that if $\Omega\subset \R^{n+1}$ is $(\tau_1,\rho_0)$-Reifenberg flat for some $\tau_1\in(0,\delta_{0})$, then
$\Omega^\pm_b$ and $\Omega^\pm_s$ are $(c\tau_0^{1/2},\rho_0/2)$-Reifenberg flat. Further,
for each $S\in\WW_0$, $10 B_S\cap \partial\Omega^\pm_b$ and $10 B_S\cap\partial\Omega^\pm_s$ are $c\tau_0^{1/2}$-Lipschitz graphs with 
respect {to} the best approximating $n$-plane for $\partial\Omega^+$ in $B_S$.
\end{lemma}
This result is an immediate consequence of Lemmas 2.2 and 2.3 from \cite{AMT-singular}.
\vv

\subsection{Some properties of the approximating domains}\label{sec4.3}

\begin{lemma}\label{lemcreix}
Suppose that $\Omega^+$ is $(\tau_1,r_0)$-Reifenberg flat for some $0<\tau_1\leq\tau_0$,
and that $\Lambda_0 r(B_0)\leq r_0$ for some $\Lambda_0\geq1$.
Denote $\sigma =\HH^n|_{\partial\Omega^+_b}$. If $\tau_1$ is small enough and $\Lambda_0$ big enough (both depending on $\tau_0$), then
\begin{equation}\label{eqsigma49}
\sigma(B(x,r))\leq C(A,\delta) r^n
\quad \mbox{ for all $x\in  \partial\Omega^+_b$, $0<r\leq r(B_0)$.}
\end{equation}
An analogous estimate holds for $\HH^n|_{\partial\Omega^+_s}$.
\end{lemma}

{Remark that the preceding lemma implies that $\Omega^+_b$ and $\Omega^+_s$ are chord-arc domains and 
thus the surface measures $\HH^n|_{\partial\Omega^+_b}$ and $\HH^n|_{\partial\Omega^+_s}$ are $n$-rectifiable.}

\begin{proof}

Consider a ball $B=B(x,r)$, with $x\in \partial\Omega^+_b$, $0<r\leq r(B_0)$.
Suppose first that $B$ is centered in $10B_0$.
By the stopping conditions, it follows that 
{
$$\omega^+(B(y,s))\approx_{A,\delta} \Theta_{\omega^+}(B_0)\,s^n\quad \mbox{ for all $y\in \overline{G(B_0)}$, $0<s\leq r(B_0)$.}$$
This can be easily deduced from the fact that any such ball $B(y,s)$ contains a cube $Q\in\DD$ such that
$\ell(Q)\approx s$ and $Q\not\subset\HD^+(B_0)\cup \LD^+(B_0)$.
}
Then we infer that
\begin{equation}\label{equa*492}
\sigma(B(x,r)\cap \overline{G(B_0)})\approx_{A,\delta} \Theta_{\omega^+}(B_0)^{-1}\omega^+(B(x,r)\cap \overline{G(B_0)})\lesssim_{A,\delta}r^n.
\end{equation}

Denote by $I_1$ the family of cubes $S\in\WW_0$ such that $B_S\cap B\neq\varnothing$ and $r(B_S)\leq r(B)$ (with $B_S$ defined in \rf{eqbs*})
and by $I_2$ the family of cubes $S\in\WW_0$ such that $B_S\cap B\neq\varnothing$ and $r(B_S)> r(B)$.
Consider subfamilies $\wt I_1\subset I_1$ and $\wt I_2\subset I_2$ such that, for $i=1,2$,
\begin{itemize}
\item the balls $\bar B_S$, $S\in \wt I_i$, are pairwise disjoint, and
\item  $B\cap \bigcup_{S\in I_i} \bar B_S \subset \bigcup_{S\in \wt I_i}5\bar B_S$,
\end{itemize}
Then we have
$$\sigma(B) \leq \sigma(B\cap \overline{G(B_0)}) + \sum_{S\in\wt I_1} \sigma(5\bar B_S\cap B) + \sum_{S\in\wt I_2} \sigma(5\bar B_S\cap B).$$
Observe that if $I_2\neq\varnothing$, then we can assume that $\wt I_2$ is made up of a single ball and that $\wt I_1$ is empty (since a ball $5B_S$, with $S\in \wt I_2$, suffices to cover $B$).

We will use now that, by \cite{AMT-singular}, $5\bar B_S\cap \partial\Omega^+_b$ is a Lipschitz graph (with slope at most $C\tau_0^{1/2}$ with respect to some suitable axis). Concerning the last sum on the right
hand side, since there is at most one ball $B_S$, $S\in \wt I_2$, 
$$\sum_{S\in\wt I_2} \sigma(5\bar B_S\cap B) 
\lesssim r(B)^n.$$
For the sum over $\wt I_1$, we write
$$\sum_{S\in\wt I_1} \sigma(5 \bar B_S\cap B) \lesssim \sum_{S\in\wt I_1} r(B_S)^n.$$ 
Now we use the fact that 
$$\Theta_{\omega^+}(\tau_0^{-2}S)\approx_{A,\delta} \Theta_{\omega^+}(B_0),$$
by the stopping conditions and the doubling property of $\omega^+$.
Also,
if $\tau_1$ is small enough and $\Lambda_0$ is big enough (both depending on $\tau_0$), by \cite[Theorem 4.1]{Kenig-Toro-duke} or Lemma \ref{lemfacil} we have 
$$\Theta_{\omega^+}(\tau_0^{-2}S)\approx \Theta_{\omega^+}(\tau_0^{-1}S)\approx\Theta_{\omega^+}(B_S),$$
and therefore
$$r(B_S)^n = \Theta_{\omega^+}(B_S)^{-1}\,\omega^+(B_S) \approx_{A,\delta} \Theta_{\omega^+}(B_0)^{-1}\,\omega^+(B_S).
$$
Thus, using that the balls $B_S$, $S\in\wt I_1$, are disjoint and contained in $3B$,
$$\sum_{S\in\wt I_1} r(B_S)^n\approx_{A,\delta} \Theta_{\omega^+}(B_0)^{-1}\,\sum_{S\in\wt I_1}\omega^+(B_S) \lesssim_{A,\delta} \Theta_{\omega^+}(B_0)^{-1}\,\omega^+(3B).$$
From the fact that $\wt I_1$ is non-empty, it follows easily that $\Theta_{\omega^+}(B) \approx_{A,\delta} \Theta_{\omega^+}(B_0)$ (applying \cite[Theorem 4.1]{Kenig-Toro-duke} or Lemma \ref{lemfacil} again), and so
$$\Theta_{\omega^+}(B_0)^{-1}\,\omega^+(3B) \approx \Theta_{\omega^+}(B_0)^{-1}\,\omega^+(B)\approx_{A,\delta} r(B)^n,$$
which implies that
$$\sum_{S\in\wt I_1} \sigma(5\bar B_S\cap B) \lesssim_{A,\delta} r(B)^n.$$
Together with the previous estimates, this shows that \rf{eqsigma49} holds for any ball $B$
centered in $10B_0\cap\partial\Omega_b^+$ with radius at most $r(B_0)$.

Suppose now that $B$ is centered in $\partial\Omega_b^+\setminus 10 B_0$. 
Consider the families $\wt I_1$ and $\wt I_2$ defined as above, and denote $\wt I=\wt I_1\cup \wt I_2$. 
By construction, in this case all the balls $B_S$, $S\in\wt I$, have radius equal to $\tau_0 r(B_0)$ and are contained in a ball $B'$ concentric with $B$ with radius $r(B')=3r(B_0)$.
If $\tau_1$ is assumed small enough, then
the best approximating planes $P_S$ for all these balls will be very close in $B'$, and then $B\cap \partial\Omega_b^+$
will be a $C'\tau_0^{1/2}$-Lipschitz graph (since it is a $C\tau_0^{1/2}$-Lipschitz graph with respect
to $P_S$ in each ball $10B_S$). So  \rf{eqsigma49} also holds in this case.
\end{proof}

\vv

Given a measure $\mu$, for every $\ell>0$ and every function $f\in L^1_{loc}(\mu)$, we denote
$$\|f\|_{*,\ell,\mu} = \sup_{\begin{subarray}{l} x\in\partial\Omega^+\\0<r\leq \ell\end{subarray}} \avint_{B(x,r)}
|f - m_{\mu,B(x,r)} f|\,d\mu.$$

Our next objective is to prove the following.

\begin{lemma}\label{lemnorm}
Suppose that $\Omega^+$ is $(\tau_1,r_0)$-Reifenberg flat for some $0<\tau_1\leq\tau_0$,
and that $\Lambda_0 r(B_0)\leq r_0$ for some $\Lambda_0\geq1$. Suppose also that $\omega^+$ and $\omega^-$ are mutually absolutely continuous, that 
$\dfrac{d\omega^+}{d\omega^-}\in B_{3/2}(\omega^-)$ and that the inner normal $N$ belongs to $\vmo(\omega^+)$.
Denote $\sigma =\HH^n|_{\partial\Omega^+_b}$. If $\tau_1$ is small enough (depending on $\tau_0$ and $\Lambda_0$), then
the oscillation of the inner unit normal to $\partial\Omega^+_b$ 
in any ball $B$ centered { in} $\partial\Omega^+_b$ with radius $0<r(B)\leq r( B_0)$
satisfies
$$ \avint_{B\cap \partial\Omega^+_b} |N_{\Omega^+_b} -m_{B,\sigma} N_{\Omega^+_b}|\,d\sigma\lesssim C(A,\delta)\|N_{\Omega^+}\|_{*,r(10\Lambda_0 B_0),\omega^+} + \tau_0^{1/2} + \ve_1,$$
with $\ve_1(\tau_0)$ as small as wished if $\tau_0$ is small enough.
Analogous estimates hold for $\Omega^-_b$ and $\Omega^\pm_s$.
\end{lemma}

To prove this lemma we will need the following key result.

\begin{lemma}\label{keylemma}
Let $\Omega^+\subset\R^{n+1}$ be a bounded NTA domain and let $\Omega^-= \R^{n+1}\setminus \overline{\Omega^+}$ be an NTA domain as well.
Suppose that $\Omega^+\subset\R^{n+1}$ is
$(\tau_0,r_0)$-Reifenberg flat for some 
$\tau_0>0$ and $r_0>0$.
Suppose also that $\omega^+$ and $\omega^-$ are mutually absolutely continuous, that 
$\dfrac{d\omega^+}{d\omega^-}\in B_{3/2}(\omega^-)$ and that the inner normal $N$ belongs to $\vmo(\omega^+)$.
Let $B$ be a ball centered { in}
$\partial\Omega^+$ with $\Lambda_0r( B)\leq r_0/4$.
 Let $L_B$ be a best approximating $n$-plane for $\partial\Omega^+\cap B$ and $N_B$ the unit normal to $L_B$ pointing to $\Omega^+$. For any $\ve_1>0$, $$\big|N_B - m_{B,\omega^+}N_{\Omega^+}\big| \leq\ve_1=\ve_1(\tau_0,r(B)),$$
with $\ve_1$ as small as wished if $\tau_0$ is small enough and $r(B)$ small enough,
\end{lemma}

We defer the proof of this result to the next subsection, and we show first how this can be
used to prove Lemma \ref{lemnorm}.

\begin{proof}[Proof of Lemma \ref{lemnorm}]
First consider a ball $B$ centered { in} $10 B_0\cap \partial\Omega^+_b$ with radius $0<r(B)\leq r(B_0)$.
Take the same families of cubes $I_1,I_2,\wt I_1,\wt I_2\subset \WW_0$ as in the proof of Lemma
\ref{lemcreix}. For any unit constant vector $C_B$ to be chosen below, 
we split
\begin{align}\label{eqspl63}
\int_{B\cap \partial\Omega^+_b} |N_{\Omega^+_b} -C_B|\,d\sigma & \leq 
\int_{B\cap \overline{G(B_0)}} |N_{\Omega^+_b} -C_B|\,d\sigma \\
& \quad + \sum_{S\in\wt I_1} \int_{5\bar B_S\cap B} |N_{\Omega^+_b} - C_B|\,d\sigma +  \sum_{S\in\wt I_2} \int_{5\bar B_S\cap B} |N_{\Omega^+_b} -C_B|\,d\sigma.\nonumber
\end{align}

Suppose first that $I_2\neq\varnothing$ and let $\wt I_2=\{S_0\}$ (recall that $\wt I_2$ has at most one cube and that $\wt I_1=\varnothing$ in this case). 
Denote by $N_{S_0}$ the inner unit normal to a best 
approximating hyperplane $L_{S_0}$ for $5\bar B_{S_0}\cap \partial\Omega^+_b$. By \cite{AMT-singular},
since $\partial\Omega^+_b$ is a Lipschitz graph with slope at most $C\tau_0^{1/2}$ over $L_{S_0}$, it follows
that
$$|N_{\Omega^+_b}(x)-N_{S_0}|\lesssim \tau_0^{1/2}\quad \mbox{ for $\sigma$-a.e.\ $x\in 5\bar B_{S_0}$.}$$
Therefore, choosing $C_B=N_{S_0}$,
$$
\int_{B\cap \partial\Omega^+_b} |N_{\Omega^+_b} -C_B|\,d\sigma\leq \int_{5\bar B_S\cap B} |N_{\Omega^+_b} -C_B|\,d\sigma \lesssim \tau_0^{1/2} \sigma(B).$$

In the case where $I_2=\varnothing$ it is immediate to check that $2B\cap \partial\Omega^+\neq\varnothing$.
Then we choose $C_B = \avint_{4B} N_{\Omega^+}\,d\omega^+$. 
Observe that, in $\overline{G(B_0)}$,
 $N_{\Omega^+_b}= N_{\Omega^+}$ {$\sigma$-a.e.}\ and $\sigma\approx_{A,\delta} \Theta_{\omega^+}
(B_0)^{-1}\omega^+$  (i.e., they are comparable measures) by \rf{equa*492}. Using also that $\Theta_{\omega^+} {(4B)}\approx_{A,\delta}\Theta_{\omega^+}(B_0)$ if $\overline{G(B_0)}\cap B\neq\varnothing$, we get
\begin{align}\label{eqgr*1}
\int_{B\cap \overline{G(B_0)}} |N_{\Omega^+_b} -C_B|\,d\sigma & \approx_{A,\delta} \Theta_{\omega^+}(B_0)^{-1}
\int_{B\cap \overline{G(B_0)}} |N_{\Omega^+} - m_{4B,\omega^+}N_{\Omega^+}|\,d\omega^+ \\
& \lesssim_{A,\delta} \|N_{\Omega^+}\|_{*,r(10B_0),\omega^+} 
\Theta_{\omega^+}(B_0)^{-1}\,  {\omega^+(4B)} \nonumber\\ &\approx_{A,\delta}\|N_{\Omega^+}\|_{*,r(10B_0),\omega^+}\sigma(B).
\nonumber
\end{align}

Consider now the sum over $S\in\wt I_1$ in \rf{eqspl63}. For each $S\in\wt I_1$, denote again  by $N_S$ the inner unit normal to the best 
approximating hyperplane $L_S$ for $5\bar B_S\cap \partial\Omega^+$. Then we write
\begin{align*}
\int_{B\cap 5\bar B_S} |N_{\Omega^+_b} -C_B|\,d\sigma &\leq 
\int_{5\bar B_S} |N_{\Omega^+_b} - N_S|\,d\sigma + \int_{5\bar B_S}\! |N_S - m_{B_S,\omega^+}N_{\Omega^+}|\,d\sigma\\
&\quad + \int_{5\bar B_S} |m_{B_S,\omega^+}N_{\Omega^+} - m_{4B,\omega^+}N_{\Omega^+}|\,d\sigma
.
\end{align*}
By Lemmas 2.2 and 2.3 from \cite{AMT-singular}, $|N_{\Omega^+_b} - N_S|\lesssim\tau_0^{1/2}$ on $5\bar B_S$ and, by Lemma \ref{keylemma}, we know that 
\begin{equation}\label{eqclau*92}
|N_S - m_{B_S,\omega^+}N_{\Omega^+}|\lesssim \ve_1.
\end{equation}
 So we get
\begin{align*}
\int_{B\cap 5\bar B_S} |N_{\Omega^+_b} -C_B|\,d\sigma & \lesssim (\tau_0^{1/2}+\ve_1) \sigma(5\bar B_S) + 
|m_{B_S,\omega^+}N_{\Omega^+} - m_{4B,\omega^+}N_{\Omega^+}|\,\sigma(5\bar B_S)\\
& \approx(\tau_0^{1/2}+\ve_1) \sigma(B_S) + \Theta_{\omega^+}(B_S)^{-1}\,
|m_{B_S,\omega^+}N_{\Omega^+} - m_{4B,\omega^+}N_{\Omega^+}|\,\omega^+(B_S).
\end{align*}
Here we have used that on $5\bar B_S$ the measure $\sigma$ is the surface measure on a Lipschitz graph and thus is
doubling on $5\bar B_S$ (with constants independent of $A$ and $\delta$). Next we take into account that 
\begin{align*}
|m_{B_S,\omega^+}N_{\Omega^+} - m_{4B,\omega^+}N_{\Omega^+}|\,\omega^+(B_S) &\leq \int_{B_S} \!|N_{\Omega^+}-m_{B_S,\omega^+}N_{\Omega^+}|\,d\omega^+ +  \int_{B_S} \!|N_{\Omega^+} - m_{4B,\omega^+}N_{\Omega^+}|\,d\omega^+\\
&\leq \|N_{\Omega^+}\|_{*,r(10 B_0),\omega^+} \omega^+(B_S) +  \int_{B_S} |N_{\Omega^+} - m_{4B,\omega^+}N_{\Omega^+}|\,d\omega^+
\end{align*}
and that $\Theta_{\omega^+}(B_S)\approx_{A,\delta}\Theta_{\omega^+}(B_0)$.
Then, summing over $S\in\wt I_1$ and using the disjointness of the balls $B_S$, we derive
\begin{align*}
\sum_{S\in\wt I_1} \int_{5\bar B_S\cap B} |N_{\Omega^+_b} - C_B|\,d\sigma & \lesssim (\tau_0^{1/2}+\ve_1)\sum_{S\in\wt I_1}\sigma(B_S) +  \frac{C(A,\delta)}{\Theta_{\omega^+}(B_0)}\|N_{\Omega^+}\|_{*,r(10 B_0),\omega^+} \sum_{S\in\wt I_1}\omega^+(B_S) \\
& \quad
+ \frac{C(A,\delta)}{\Theta_{\omega^+}(B_0)} \sum_{S\in\wt I_1} \int_{B_S} |N_{\Omega^+} - m_{4B,\omega^+}N_{\Omega^+}|\,d\omega^+\\
& \lesssim \big(\tau_0^{1/2}+ \ve_1 +  C(A,\delta)\,\|N_{\Omega^+}\|_{*,r(10 B_0),\omega^+} \big) \sigma(B) \\& \quad +
\frac{C(A,\delta)}{\Theta_{\omega^+}(B_0)}\int_{7B} |N_{\Omega^+} - m_{4B,\omega^+}N_{\Omega^+}|\,d\omega^+\\
&\lesssim \Big(\tau_0^{1/2} +\ve_1+  C(A,\delta)\,\|N_{\Omega^+}\|_{*,r(10B_0),\omega^+} \Big) \sigma(B).
\end{align*}
Together with \rf{eqgr*1}, this yields
$$\int_{B} |N_{\Omega^+_b} - C_B|\,d\sigma  \lesssim  \bigg(\tau_0^{1/2} + \ve_1 + C(A,\delta)\, \|N_{\Omega^+}\|_{*,r(10B_0),\omega^+} \bigg) \sigma(B),$$
and proves the lemma for balls $B$ centered in $10B_0\cap\partial\Omega_b^+$.

In the case where $B$ is centered in $\partial\Omega_b^+\setminus 10 B_0$,
we know that $B\cap \partial\Omega_b^+$
will be a $C'\tau_0^{1/2}$-Lipschitz graph (by the same arguments as in the end of the proof of Lemma 
\ref{lemcreix}), and thus the lemma also holds.
\end{proof}

\vv

\subsection{The proof of Lemma \ref{keylemma}} \label{sec4.4}

Let $B$ be a ball satisfying the assumptions of Lemma \ref{keylemma}, and consider some big constant
$\Lambda$ such that $1<\Lambda\ll\Lambda_0$.  Consider the set $G(\Lambda B)$
defined in Section \ref{secstop}, with $B_0$ interchanged with $\Lambda B$. Denote
$$d_B(x) = \min\big(r(B),\,\dist(x,G(\Lambda B))\big).$$
For each 
$\tau_h\in (\tau_0,1/10)$
 and $\ve_0\in(0,\tau_h/10)$ to be chosen later
we denote 
\begin{equation}\label{eqdefhh1}
h(x) = \max(\ve_0r(B),\tau_h\,d_B(x)).
\end{equation}
Notice that $h$ is Lipschitz with constant $\tau_h$. We remark that, to prove Lemma \ref{keylemma}, below we will need to choose the parameter $\tau_h$ small enough, and later we will take $\tau_0$ depending on $\tau_h$, so that in particular $\tau_0<\tau_h$.

Remark that if $x\in 2\Lambda B$ and $r\in[h(x),r(\Lambda B)]$, then
\begin{equation}\label{eqhd32*}
\Theta_\omega(B(x,r))\approx_{A,\delta}\Theta_\omega(\Lambda B)\approx \Theta_\omega(B),
\end{equation}
assuming $\Lambda_0/\Lambda$ big enough and $\tau_0$ small enough, depending on $\tau_h$. This follows easily from the stopping conditions in Section \ref{secstop} and Lemma \ref{lemfacil}.

Given $x\in\partial\Omega^+$, we consider a best approximating $n$-plane $L_0(x)$ for $\partial\Omega^+\cap B(x,\tau_h^{-1/2}h(x))$.
We denote by $N_0(x)$ the inner unit normal to $L_0(x)$, and we consider the function
$$H(x) := h(x)\,N_0(x).$$
Observe that $x\pm H(x)\in\Omega^\pm$ if $x\in\partial\Omega^+$
Next we consider the auxiliary kernels
$$K_1^\pm(x,y) = K(x\pm H(x)-y)\quad \mbox{ for $x\in\partial\Omega^+$, $y\in\R^{n+1}$,}
$$
and 
$$K_2^\pm(x,y) = K(x\pm H(y)-y) \quad \mbox{ for  $x\in\R^{n+1}$, $y\in\partial\Omega^+$},$$
where $K(\cdot)$ is the signed Riesz kernel of homogeneity $-n$. We denote by $\RR^\pm_1,\RR^\pm_2$ the respective associated integral operators. That is, for any finite measure $\nu$,
we write
$$\RR^\pm_1\nu(x) = \int K_1^\pm(x,y)\,d\nu(y), \qquad\RR^\pm_2\nu(x) = \int K_2^\pm(x,y)\,d\nu(y),$$
whenever the integrals make sense.

{ Notice that, for each $y\in\partial\Omega^+$, $K_2^+(\cdot,y)$ is harmonic in $\Omega^+$ and continuous in
$\overline{\Omega^+}$ (the last condition holds because $\ve_0>0$); and for each $x\in\partial\Omega^+$, $K_1^-(x,\cdot)$ is also harmonic in $\Omega^+$ and continuous in $\overline{\Omega^+}$. }
As a consequence, for each $z\in \Omega^+$, 
$\RR_2^+\omega^{+,z}$ is a function which is harmonic in $\Omega^+$ and continuous in $\overline{\Omega^+}$,
and thus
\begin{equation}\label{eqrriesz1}
\int \RR_2^+\omega^{+,z}\,d\omega^{+,z} = \RR_2^+\omega^{+,z}(z).
\end{equation}
Regarding the kernel $K_1^-$, we have $\RR_1^-\omega^{+,z}(x) = K_1^-(x,z)$ and so
\begin{align}\label{eqrriesz2}
\int \RR_1^-\omega^{+,z}(x)\,d\omega^{+,z}(x) & = \int K_1^-(x,z)\,d\omega^{+,z}(x) = \int K(x-H(x)-z)\,d\omega^{+,z}(x)
\\ & =
- \int K(z+H(x)-x)\,d\omega^{+,z}(x) = -\RR_2^+\omega^{+,z}(z).\nonumber
\end{align}
We will apply the identities \rf{eqrriesz1} and \rf{eqrriesz2} to a corkscrew point $z\in\Omega^+\cap B$. More precisely, we will choose $z$ to be the point $p_B\in \Omega^+$ such that both
$p_B$ and the center $x_B$ of the ball $B$ have the same orthogonal projection on a best approximating plane $L_B$ for $\partial \Omega^+$, and so that, moreover, $\dist(p_B,x_B)=r(B)/4$.  

First we need to introduce some notation and prove some auxiliary lemmas.
For any Radon measure $\nu$, we write
$$P_r\nu(x) = \int \frac {r}{r^{n+1} + |x-y|^{n+1}}\,d|\nu|(y).$$
Notice that 
$P_r\nu(x)\approx P_\nu(B(x,r))$, where $P_\nu(B(x,r))$ is defined in \eqref{sec2.6}.

In this section, from now on, to simplify notation, sometimes we will write $\omega$ instead of $\omega^+$ and
$\omega^z$ instead of $\omega^{+,z}$.

\vv

\begin{lemma}\label{lemcreix22}
Under the assumptions of Lemma \ref{keylemma},
for every $x\in\partial\Omega^+\cap2\Lambda  B$
\begin{equation}\label{eqcreix100}
P_r\omega^{p_B}(x)\lesssim \Theta_\omega (B(x,r))\,\frac{\omega^{p_B}(B(x,r(B)))}{\omega(B)} + \frac1{r(B)^n} \quad\mbox{ for all $r >0$,}
\end{equation}
assuming $\Lambda_0/\Lambda$ big enough and $\tau_0$ small enough.
\end{lemma}

\begin{proof}
Observe that, for any $x\in\partial\Omega^+\cap2\Lambda  B$ and every $r\geq r(B)$,
$$\Theta_{\omega^{p_B}}(B(x,r))\le\frac1{r^n}\le \frac1{r(B)^n}.$$
On the other hand, for $r\leq r(B)$, using the well known change of pole formula for NTA domains of Jerison and Kenig, we have
\begin{equation}\label{eqjk89}
\Theta_{\omega^{p_B}}(B(x,r))\approx \frac{\omega(B(x,r))\,\omega^{p_B}(B(x,r(B)))
}{\omega(B(x,r(B)))\,r^n}\,\approx \frac{\Theta_\omega(B(x,r))\,\omega^{p_B}(B(x,r(B)))
}{\omega(B)},
\end{equation}
where we took into account that $\omega(B(x,r(B)))\approx\omega(B)$, with a constant independent of $\Lambda$, because of the $(\tau_0,\Lambda_0r(B))$-Reifenberg flatness of $\Omega^+$, and we assume $\tau_0$ small enough and $\Lambda_0\gg\Lambda$ big enough (see Lemma \ref{lemfacil}).
Then we deduce
\begin{align*}
P_r\omega^{p_B}(x)& \lesssim \sum_{k\geq 0:2^kr\leq r(B)}2^{-k}\Theta_{\omega^{p_B}}(B(x,2^kr)) + 
\sum_{k\geq 0:2^kr> r(B)}2^{-k}\frac1{(2^kr)^n}\\
& \lesssim
\frac{\omega^{p_B}(B(x,r(B)))}{\omega(B)} P_r\omega(x) + \frac1{r(B)^n}\lesssim \frac{\omega^{p_B}(B(x,r(B)))}{\omega(B)}\,\Theta_\omega(B(x,r)) + \frac1{r(B)^n},
\end{align*}
using again the Reifenberg flatness of $\Omega^+$ and Lemma \ref{lemfacil} for the last inequality.
\end{proof}

\vv

\begin{lemma}\label{lemcreix2}
Under the assumptions of Lemma \ref{keylemma},
for every $x\in\partial\Omega^+$,
\begin{equation}\label{eqcreix10}
 P_r\omega^{p_B}(x)\lesssim_{A,\delta} \frac1{r(B)^n} \quad\mbox{\!for $r\in [h(x),r(B)]$,}
\end{equation}
assuming $\Lambda_0/\Lambda$ big enough and $\tau_0$ small enough.
Further, in the case $x\in \partial\Omega^+\setminus 2\Lambda  B$, the implicit constant in 
\rf{eqcreix10} is independent of $A$ and $\delta$, assuming $\tau_0$ small enough (depending on $\tau_h$ and $\Lambda $).
\end{lemma}

Remark that the estimates in this lemma may depend on $A$, unlike in the preceding lemma.

\begin{proof}
First we show \rf{eqcreix10} for $x\in 2\Lambda  B$. In this case, by \rf{eqcreix100}, the trivial estimate $\omega^{p_B}(B(x,r(B)))\leq 1$, and the fact that $\Theta_\omega (B(x,r))\approx_{A,\delta}
\Theta_\omega (B)$ by  \rf{eqhd32*},  we deduce
$$
P_r\omega^{p_B}(x)\lesssim \frac{\Theta_\omega (B(x,r))
}{\omega(B)} + \frac1{r(B)^n} \lesssim_{A,\delta} \frac1{r(B)^n}.
$$

To estimate $P_r\omega^{p_B}(x)$ in the case $x\in \partial\Omega^+\setminus 2\Lambda B$,
we take into account that $d_B(x)\approx r(B)$ for all $x\in\partial\Omega^+\setminus 2\Lambda B$, so that for any $r\in[h(x),r(B)]$, arguing as in \rf{eqjk89},
\begin{align*}
\Theta_{\omega^{p_B}}(B(x,r)) & \approx \frac{\omega(B(x,r))\,\omega^{p_B}(B(x,r(B)))
}{\omega(B(x,r(B)))\,r^n} =  \frac{\Theta_\omega(B(x,r))\,\omega^{p_B}(B(x,r(B)))
}{\omega(B(x,r(B)))}\\
&\approx \frac{\Theta_\omega(B(x,r(B)))\,\omega^{p_B}(B(x,r(B)))}{\omega(B(x,r(B)))} = \frac{\omega^{p_B}(B(x,r(B)))}{r(B)^n}\leq \frac1{r(B)^n}.
\end{align*}
Therefore,
$$P_{r} \omega^{p_B}\lesssim \sum_{k\geq0: 2^kr\leq r(B)} 2^{-k}\Theta_{\omega^{p_B}}(B(x,2^kr)) + \sum_{k\geq0: 2^kr> r(B)} 2^{-k}\frac1{r(B)^n}\lesssim \frac1{r(B)^n}.$$ 
\end{proof}

\vv
\begin{lemma}\label{lemdifer23}
Under the assumptions of Lemma \ref{keylemma},
$$\int \big|\RR_1^+\omega^{p_B} - \RR_2^+\omega^{p_B}\big|\,d\omega^{p_B}
\lesssim_{A,\delta} \frac{\tau_0+\tau_h^{1/2}}{r(B)^n},$$
assuming $\Lambda_0/\Lambda$ big enough and $\tau_0$ small enough.
\end{lemma}

\begin{proof}
We have
\begin{align*}
\int \big|\RR_1^+\omega^{p_B} - \RR_2^+\omega^{p_B}\big|\,d\omega^{p_B} & \leq 
 \iint \big|K(x+H(x)-y)-K(x+H(y)-y)\big|\,d\omega^{p_B}(y)d\omega^{p_B}(x)\\
 & = \iint_{|x-y|\geq \tau_h^{-1/2}h(x)}\cdots + \iint_{|x-y|< \tau_h^{-1/2}h(x)}\cdots =: I_1 + I_2.
\end{align*}
First we estimate $I_1$. Observe that, for $x,y$ in the domain of integration of $I_1$,
$$|h(x) - h(y)|\leq \tau_h |d_B(x) - d_B(y)| \leq \tau_h|x-y|.$$
Thus, since $|x-y|\geq \tau_h^{-1/2}h(x)$,
$$|H(x)-H(y)| \leq |H(x)| + |H(y)| = h(x)+h(y)\leq 2h(x) + |h(x)-h(y)| \leq 3\tau_h^{1/2}|x-y|\ll|x-y|.$$
Hence,
$$\big|K(x+H(x)-y)-K(x+H(y)-y)\big|\lesssim \frac{|H(x)-H(y)|}{|x-y|^{n+1}}\lesssim 
\frac{h(x) + h(y)}{|x-y|^{n+1}}.
$$
Then, using also Lemma \ref{lemcreix2}, we get
\begin{align*}
I_1& \lesssim \iint_{|x-y|\geq \tau_h^{-1/2}h(x)} \frac{h(x) + h(y)}{|x-y|^{n+1}}\,d\omega^{p_B}(y)d\omega^{p_B}(x)\\
& \!\leq\!\iint_{|x-y|\geq \tau_h^{-1/2}h(x)} \frac{h(x)}{|x-y|^{n+1}}\,d\omega^{p_B}(y)d\omega^{p_B}(x)
+ \iint_{|x-y|\geq \tau_h^{-1/2}h(y)/2} \frac{h(y)}{|x-y|^{n+1}}\,d\omega^{p_B}(x)d\omega^{p_B}(y)\\
& \!\leq 2 \tau_h^{1/2}\int P_{ \tau_h^{-1/2}h(x)}\omega^{p_B}(x)\,d\omega^{p_B}(x)\lesssim_{A,\delta}
\frac{ \tau_h^{1/2}}{r(B)^n}.
\end{align*}

Next we consider the integral $I_2$. The points $x,y$ in the domain of integration $I_2$ satisfy $|x-y|<\tau_h^{-1/2}h(x)$ and thus
$$|h(x)-h(y)| \leq \tau_h\, |d_B(x)-d_B(y)|\leq \tau_h\, |x-y|<\tau_h^{1/2}h(x).$$
In particular, this implies 
$$\frac12 h(x)\leq h(y)\leq 2h(x).$$
Then we deduce
\begin{align*}
|H(x)-H(y)| & = |h(x)N_0(x)- h(y)N_0(y)|\leq |h(x)-h(y)| + h(y)\,|N_0(x)-N_0(y)| \\
&\leq \big(\tau_h^{1/2}+
2|N_0(x)-N_0(y)|\big)\, h(x).
\end{align*}
Observe now that,  since $|x-y|<\tau_h^{-1/2}h(x)\approx \tau_h^{-1/2} h(y)$, by the Reifenberg flatness, we have
$$|N_0(x)-N_0(y)|\lesssim \tau_0,$$
and so 
$$|H(x)-H(y)|\lesssim \big(\tau_h^{1/2}+\tau_0
\big)\, h(x) \approx \big(\tau_h^{1/2}+\tau_0
\big)\, h(y) \ll \min\big(|x+H(x)-y|,\ |x+H(y)-y|\big)$$
(to check this, notice that if $|x-y|\leq 2\max(h(x),h(y))$, then $H(x)$ and $H(y)$ are ``almost perpendicular" to $x-y$ and thus $|x+H(x)-y|\geq |H(x)|=h(x)$, and the same happens replacing $H(x)$ by $H(y)$).
Therefore,
\begin{equation}\label{eqana7*}
\big|K(x+H(x)-y)-K(x+H(y)-y)\big|\lesssim \frac{|H(x)-H(y)|}{|x+H(x)-y|^{n+1}}\lesssim 
\frac{\big(\tau_h^{1/2}+\tau_0
\big)\, h(x)}{|x-y|^{n+1} + h(x)^{n+1}}.
\end{equation}
Then, by Lemma \ref{lemcreix2}, we obtain
\begin{align*}
I_2& \lesssim \iint \frac{\big(\tau_h^{1/2}+\tau_0\big)\, h(x)}{|x-y|^{n+1} + h(x)^{n+1}}\,d\omega^{p_B}(y)d\omega^{p_B}(x) \lesssim \big(\tau_h^{1/2}+\tau_0\big) \int P_{h(x)}\omega^{p_B}(x)\,d\omega^{p_B}(x)\\
& \lesssim_{A,\delta}
\frac{\tau_h^{1/2}+\tau_0}{r(B)^n}.
\end{align*}
Together with the estimate we obtained for $I_1$, this proves the lemma.
\end{proof}

\vv

\begin{lemma}\label{lemreverse}
Under the assumptions of Lemma \ref{keylemma} and the above notation,
$$\int_{G(\Lambda B)} \big(\RR_1^+\omega^{p_B} - \RR_1^-\omega^{p_B}\big)\,d\omega^{p_B} = 2\,\RR_2^+(\omega^{p_B})(p_B) + \frac{C_B(\tau_0,\tau_h,\delta,A,\Lambda)}{r(B)^n},$$
with 
$$|C_B(\tau_0,\tau_h,\delta,A,\Lambda)|\leq\ve$$
if $\tau_0,\tau_h,\delta$ are small enough and $\Lambda_0,A,\Lambda$ big enough. 
\end{lemma}

\begin{proof}
Recall that $G(\Lambda B)\subset \frac32 \Lambda B$.
We write
\begin{align}\label{eqali6329}
\int_{G(\Lambda B)} \big(\RR_1^+\omega^{p_B} - \RR_1^-\omega^{p_B}\big)\,d\omega^{p_B}
 & = \int\big(\RR_2^+\omega^{p_B} - \RR_1^-\omega^{p_B}\big)\,d\omega^{p_B} \\
 & \quad+
 \int\big(\RR_1^+\omega^{p_B} - \RR_2^+\omega^{p_B}\big)\,d\omega^{p_B}\nonumber\\
&\quad - \int_{\partial\Omega^+\setminus G(\Lambda B)} \big(\RR_1^+\omega^{p_B} - \RR_1^-\omega^{p_B}\big)\,d\omega^{p_B} \nonumber
\\ &=: I_1 + I_2 - I_3.\nonumber
\end{align}
Observe now that, by \rf{eqrriesz1} and \rf{eqrriesz2},
$$I_1 = 2\,\RR_2^+\omega^{p_B}(p_B).$$
Also, by Lemma \ref{lemdifer23},
\begin{equation}\label{eqi2**}
|I_2|\lesssim_{A,\delta}\frac{\tau_0+\tau_h^{1/2}}{r(B)^n}.
\end{equation}
Finally, to estimate $I_3$ notice that, for any $x\in\partial\Omega^+$,
\begin{align*}
\big|\RR_1^+\omega^{p_B}{(x)} - \RR_1^-\omega^{p_B}{(x)}\big| & \leq \int \big|K(x+H(x)-y)-K(x-H(x)-y)\big|\,
d\omega^{p_B}(y)\lesssim P_{h(x)}\omega^{p_B}(x),
\end{align*}
by arguments analogous to the one in \rf{eqana7*}.
Thus,
$$|I_3|\lesssim \int_{\partial\Omega^+\setminus 2\Lambda B} P_{h(\cdot)}\omega^{p_B}\,d\omega^{p_B} + 
\int_{2\Lambda B\setminus G(\Lambda B)} P_{h(\cdot)}\omega^{p_B}\,d\omega^{p_B} =: I_{3,a} + I_{3,b}.$$
To estimate $I_{3,a}$ recall that $P_{h(x)} {\omega^{p_B}(x)}\lesssim\frac1{r(B)^n},$ by  \rf{eqcreix10} (with 
an implicit constant independent of $A$ and $\delta$).
Hence,
$$I_{3,a}\lesssim \frac1{r(B)^n}\,\omega^{p_B}(\partial\Omega^+\setminus 2\Lambda B).$$

Concerning  $I_{3,b}$ notice that, for any $x\in\partial\Omega^+\cap 2\Lambda B$, by Lemma \ref{lemcreix22},
\begin{align*}
P_{h(x)}\omega^{p_B}(x) \lesssim\frac{\Theta_\omega (B(x,h(x)))}{\omega(B)} + \frac1{r(B)^n}.
\end{align*}
Thus,
\begin{align*}
I_{3,b}&\lesssim \frac1{\omega(B)}\int_{\Lambda B\setminus G(\Lambda B)} \Theta_\omega (B(x,h(x)))\,d\omega^{p_B}(x) +
\frac1{\omega(B)}\int_{2\Lambda B \setminus \Lambda B} \Theta_\omega (B(x,h(x)))\,d\omega^{p_B}(x) 
\\ &
\quad + \frac1{r(B)^n}\,\omega^{p_B}\big(\partial\Omega^+ \setminus(\Lambda B\cap G(\Lambda B))\big) \\
& =: J_1 + J_2+ J_3.
\end{align*}

First we estimate the term $J_1$. By H\"older's inequality, for any $p>1$ we have
\begin{equation}\label{eqhold32}
\int_{\Lambda B\setminus G(\Lambda B)} \Theta_\omega (B(x,h(x)))\,d\omega^{p_B}(x)\leq 
\left(\int_{\Lambda B} \Theta_\omega (B(x,h(x)))^p\,d\omega^{p_B}(x)\right)^{1/p} \omega^{p_B}(\Lambda B\setminus G(\Lambda B))^{1/p'}.
\end{equation}
To show the last term is small, consider a Besicovitch covering of $\Lambda B\setminus G(\Lambda B)$
with balls $B_i$ with radius equal to $r(B)$. Using that $\omega(B)\approx\omega(B_i)$ by Lemma \ref{lemfacil} we get
\begin{align}\label{eqas723}
\omega^{p_B}(\Lambda B\setminus G(\Lambda B)) & \approx \sum_i \omega^{p_B}(B_i\cap \Lambda B\setminus G(\Lambda B))
\approx \sum_i \frac{\omega(B_i\cap \Lambda B\setminus G(\Lambda B))\,\omega^{p_B}(B_i)}{\omega(B_i)}\\
&\lesssim \sum_i \frac{\omega(B_i\cap \Lambda B\setminus G(\Lambda B))}{\omega(B)} \lesssim 
\frac{\omega(\Lambda B\setminus G(\Lambda B))}{\omega(B)}\leq C(\Lambda)\,\ve',\nonumber
\end{align}
with $\ve'=\ve'(A,\delta)$, by Lemma \ref{lem4.1}.
Next we estimate the integral on the right hand side of \rf{eqhold32}. To this end, for
$x\in\Lambda B$ we write
\begin{align*}
\Theta_\omega (B(x,h(x))) & = \big(\Theta_{\omega^+} (B(x,h(x)))\,\Theta_{\omega^-} (B(x,h(x)))\big)^{1/2}
\,\left(\frac{\Theta_{\omega^+} (B(x,h(x)))}{\Theta_{\omega^-} (B(x,h(x)))}\right)^{1/2}\\
& \lesssim 
\big(\Theta_{\omega^+} (\Lambda B)\,\Theta_{\omega^-} (\Lambda B)\big)^{1/2}
\,\left(\frac{\Theta_{\omega^+} (B(x,h(x)))}{\Theta_{\omega^-} (B(x,h(x)))}\right)^{1/2},
\end{align*}
by the ACF monotonicity formula (see the proof of Lemma \ref{lem44}).
Denoting by $M_{\omega^-}$ the centered maximal Hardy-Littlewood operator with respect to $\omega^-$ and 
$f:=\frac{d\omega^+}{d\omega^-}$, we get (since $h(x)\leq r(\Lambda B)$),
$$\Theta_\omega (B(x,h(x)))\lesssim \big(\Theta_{\omega^+} (\Lambda B)\,\Theta_{\omega^-} (\Lambda B)\big)^{1/2}
\,M_{\omega^-}(f\chi_{2\Lambda B})(x)^{1/2}.$$
Also, arguing as in \rf{eqas723}, it follows that
\begin{equation}\label{eq198}
\omega^{p_B}(\Lambda B\cap F)\lesssim \frac{\omega^+(\Lambda B\cap F)}{\omega^+(B)}
\quad \mbox{ for all $F\subset\R^{n+1}$},
\end{equation}
and thus, for any function $g$ supported in $\Lambda B$,
$$\int g\,d\omega^{p_B}\lesssim \frac1{\omega^+(B)}  \int g\,d\omega^+ = \frac1{\omega^+(B)}  \int g\,f\,d\omega^-.$$
Altogether, we get
\begin{align}\label{eqali842}
\int_{\Lambda B} \Theta_\omega (B(x,h(x)))^p\,d\omega^{p_B}(x)& \lesssim
\frac{\big(\Theta_{\omega^+} (\Lambda B)\,\Theta_{\omega^-} (\Lambda B)\big)^{p/2}}
{\omega^+(B)}
\int_{\Lambda B} |M_{\omega^-}(f\chi_{2\Lambda B})|^{p/2}\,f\,d\omega^-\\
& \leq 
\frac{\big(\Theta_{\omega^+} (\Lambda B)\,\Theta_{\omega^-} (\Lambda B)\big)^{p/2}}
{\omega^+(B)}
\int_{\Lambda B} |M_{\omega^-}(f\chi_{2\Lambda B})|^{1+p/2}\,d\omega^-\nonumber\\
& \lesssim
\frac{\big(\Theta_{\omega^+} (\Lambda B)\,\Theta_{\omega^-} (\Lambda B)\big)^{p/2}}
{\omega^+(B)}
\int_{2\Lambda B} |f|^{1+p/2}\,d\omega^-,\nonumber
\end{align}
by the $L^{1+p/2}(\omega^-)$ boundedness of $M_{\omega^-}$. Now, since $f\in B_{3/2}(\omega^-)$,
by Gehring's lemma it follows also that $f\in B_{\ve+3/2}(\omega^-)$, for some $\ve>0$. Hence, for $p$
close enough to $1$, we have
$$\avint_{2\Lambda B} |f|^{1+p/2}\,d\omega^-\lesssim 
\left(\avint_{2\Lambda B} f\,d\omega^-\right)^{1+p/2} = \left(\frac{\omega^+(2\Lambda B)}{\omega^-(2\Lambda B)}\right)^{1+p/2}.$$
So we deduce
\begin{align}\label{eqali843}
\int_{\Lambda B} \Theta_\omega (B(x,h(x)))^p\,d\omega^{p_B}(x)& \lesssim
\frac{\big(\Theta_{\omega^+} (\Lambda B)\,\Theta_{\omega^-} (\Lambda B)\big)^{p/2}}
{\omega^+(B)}\,\left(\frac{\omega^+(\Lambda B)}{\omega^-(\Lambda B)}\right)^{1+p/2}\!\omega^-(\Lambda B)\\
& \lesssim_\Lambda \,\Theta_{\omega^+} (\Lambda B)^p.\nonumber
\end{align}
As a consequence,
$$J_1\lesssim_\Lambda (\ve')^{1/p'}\frac{\Theta_{\omega^+} (\Lambda B)}{\omega^+(B)}\approx_\Lambda (\ve')^{1/p'}\frac1{r(B)^n}.$$

Next we turn our attention to the term $J_2$. To this end we consider a Besicovitch covering of
$\partial\Omega^+\cap2\Lambda B\setminus \Lambda B$ with balls $B_i$ of radius $r(B)$.
We split 
$$J_2 \leq \sum_i \frac1{\omega^+(B)}\int_{B_i} \Theta_{\omega^+} (B(x,h(x)))\,d\omega^{p_B}(x).$$
Now we argue as we did to deal with the term $J_1$. However, instead of \rf{eq198}, we use the
more precise estimate
$$
\omega^{p_B}(B_i\cap F)\approx \frac{\omega^+(B_i\cap F)\,\omega^{p_B}(B_i)}{\omega^+(B_i)} \approx \frac{\omega^+(B_i\cap F)\,\omega^{p_B}(B_i)}{\omega^+(B)} \quad \mbox{ for all $F\subset\R^{n+1}$},$$
where again we used the fact that $\omega^+(B_i)\approx\omega^+(B)$. Then we obtain, for any function $g$ supported on $B_i$,
$$\int g\,d\omega^{p_B}\lesssim \frac{\omega^{p_B}(B_i)}{\omega^+(B)}  \int g\,d\omega^+ = \frac{\omega^{p_B}(B_i)}{\omega^+(B)}  \int g\,f\,d\omega^- .$$
Then, arguing as in \rf{eqali842} and \rf{eqali843}, choosing now $p=1$, for each ball $B_i$ we get
\begin{align*}
\int_{B_i} \Theta_{\omega^+} (B(x,h(x)))\,d\omega^{p_B}(x)& \lesssim \omega^{p_B}(B_i)\,
\frac{\big(\Theta_{\omega^+} (\Lambda B)\,\Theta_{\omega^-} (\Lambda B)\big)^{1/2}}
{\omega^+(B)}
\int_{B_i} |M_{\omega^-}(f\chi_{2B_i})|^{1/2}\,f\,d\omega_-\\
& \lesssim \omega^{p_B}(B_i)\,
\frac{\big(\Theta_{\omega^+} (\Lambda B)\,\Theta_{\omega^-} (\Lambda B)\big)^{1/2}}
{\omega^+(B)}
\int_{2B_i} |f|^{3/2}\,d\omega_-\\
&\lesssim \omega^{p_B}(B_i)\,
\frac{\big(\Theta_{\omega^+} (\Lambda B)\,\Theta_{\omega^-} (\Lambda B)\big)^{1/2}}
{\omega^+(B)}\,\left(\frac{\omega^+(2B_i)}{\omega^-(2 B_i)}\right)^{3/2}\!\omega^-(2B_i).
\end{align*}
Since $\omega^\pm(2B_i)\approx \omega^\pm(B_i)\approx\omega^\pm(B)$ 
and $\Theta_{\omega^\pm} (\Lambda B)\approx \Theta_{\omega^\pm} (B)$
(by the
Reifenberg flatness), we derive
$$\int_{B_i} \Theta_{\omega^+} (B(x,h(x)))\,d\omega^{p_B}(x) \lesssim \omega^{p_B}(B_i)\Theta_{\omega^+} (B).$$
Summing on $i$ and using the finite overlap of the balls $B_i$, we derive
$$J_2\lesssim \frac{\Theta_{\omega^+} (B)}{\omega^+(B)} \sum_i\omega^{p_B}(B_i)
\lesssim \frac1{r(B)^n}\,\omega^{p_B}(\partial \Omega^+\setminus \tfrac12 \Lambda B).$$

Gathering the estimates obtained for $J_1$ and $J_2$, we obtain
\begin{align*}
I_{3,b} &\lesssim C(\Lambda)(\ve')^{1/p'}\frac1{r(B)^n}
+ \frac1{r(B)^n}\,\omega^{p_B}(\partial \Omega^+\setminus \tfrac12 \Lambda B) +
\frac1{r(B)^n}\,\omega^{p_B}\big(\partial\Omega^+ \setminus(\Lambda B\cap G(\Lambda B))\big).
\end{align*}
and combining this with the estimate for $I_{3,a}$, we get
$$|I_3|\lesssim  \frac1{r(B)^n}\,\Big(C(\Lambda)(\ve')^{1/p'} + \omega^{p_B}\big(\partial\Omega^+ \setminus(\tfrac12\Lambda B\cap G(\Lambda B))\big)\Big).$$
Now we write
$$\omega^{p_B}\big(\partial\Omega^+ \setminus(\tfrac12\Lambda B\cap G(\Lambda B))\big)\leq \omega^{p_B}\big(\partial\Omega^+ \setminus\tfrac12\Lambda B\big) +
\omega^{p_B}\big(\Lambda B \setminus G(\Lambda B)).
$$
To deal with the first summand on the right hand side, we use the fact that $\dist(p_B,\partial\Omega^+\setminus \frac12\Lambda B)\approx r(\Lambda B)$, and by the H\"older continuity of $\omega^{(\cdot)}(\partial\Omega^+\setminus \frac12\Lambda B)$ { in 
$\frac14\Lambda B$} (see Lemma 4.1 \cite{Jerison-Kenig}), we have
$$
\omega^{p_B}\big(\partial\Omega^+ \setminus\tfrac12\Lambda B\big)
\lesssim \left(\frac{\dist(p_B,\partial\Omega^+)}{r(\Lambda B)}\right)^\eta
\approx  \Lambda^{-\eta},
$$
for some constant $\eta>0$ just depending on the NTA character of $\Omega^+$. Recall also that, by \rf{eqas723},
$\omega^{p_B}\big(\Lambda B \setminus G(\Lambda B)\big)\lesssim_\Lambda \ve'$.
Therefore,
$$|I_3|\lesssim  \frac1{r(B)^n}\,\big(C(\Lambda)(\ve')^{1/p'} + \Lambda^{-\eta} + \ve'\big) \lesssim \frac1{r(B)^n}\,\big(C(\Lambda)(\ve')^{1/p'} + \Lambda^{-\eta} \big).$$
In combination with \rf{eqali6329} and \rf{eqi2**}, and the fact that $I_1 = 2\,\RR_2^+\omega^{p_B}(p_B)$, this yields the lemma, with
$$C_B(\tau_0,\tau_h,\delta,A,\Lambda) = C(A,\delta)(\tau_0+\tau_h^{1/2}) + C(\Lambda)(\ve')^{1/p'} + C\Lambda^{-\eta},$$
and $\ve'=\ve'(A,\delta)$ as small as wished if $A$ is big enough and $\delta$ small enough.
\end{proof}

\vv

\begin{lemma}\label{lemcompactness}
There is an absolute constant $c_n>0$ such that for any {$\alpha>0$}, if $\tau_0,\tau_h,\ve_0,r(B)/r_0$ are small enough, then 
$$\Bigl|\RR_2^+\omega^{p_B}(p_B) - \frac{c_n}{r(B)^n} \,N_B\Big|\leq \frac{\alpha}{r(B)^n}.$$
\end{lemma}

Recall that we assume $\Omega$ to be $(\tau_0,r_0)$-Reifenberg flat.

\begin{proof} 
Suppose this fails for some $\alpha>0$ and some $c_n$ to be fixed in a moment. Consider a sequence of balls $B_k$ with $r(B_k)\leq { r_0}/k$ and suppose that  {$\Omega^+_k$}
is $(\tau_k, {r_0})$-Reifenberg flat, with $\tau_k\to0$. Let $\tau_{h,k}=10^{-k}$ and $\ve_{0,k}=100^{-k}$, say, and
suppose that
\begin{equation}\label{eqkkk*}
\Bigl|\RR_{2,k}^+\omega^{p_{B_k}}(p_{B_k}) - \frac{c_n}{r(B_k)^n} \,N_{B_k}\Big|> \frac{\alpha}{r(B_k)^n},
\end{equation}
where we denoted by $\RR_{2,k}^+$ the singular integral operator associated with the kernel $K_2^+$ defined choosing $\tau_h=\tau_{h,k}$ and $\ve_0=\ve_{0,k}$.

Let
$$\wt \Omega_k = \frac1{r(B_k)}\,( {\Omega^+_k} - x_{B_k}),$$
where $x_{B_k}$ is the center of $B_k$. Rotate $\wt \Omega_k$ to get a new domain $\Omega_k$ such that the best approximating $n$-plane for $\partial\Omega_k$ at $B(0,1)$ is horizontal.
 Further, 
the harmonic measure $\omega_k^{p_0}$ of $\Omega_k$ with pole at $p_0=(0,\ldots,0,1/4)$ equals the image measure of $\omega^{p_{B_k}}$ by the map $y\mapsto \frac1{r(B_k)}\,(y- x_{B_k})$.
From \rf{eqkkk*} we deduce that
$$\Bigl|\RR_{2,k}^+\omega_k^{p_0}(p_0) - c_n \,e_{n+1}\Big|> \alpha$$
for every $k$.
However, it is easy to check that $\Omega_k$ converges locally to the upper half-space $H=\{x\in\R^{n+1}: x_{n+1}>0\}$ in Hausdorff distance. This implies that, $\omega_k^{p_0}$ converges weakly to $\omega_H^{p_0}$ (where $\omega_H$ is the harmonic measure of $H$). Assuming this for the moment, we deduce that
$$|\RR\omega_H^{p_0}(p_0) - c_n \,e_{n+1}\Big|> \alpha,$$
since $p_0$ is far away from the boundary.
By symmetry, it is clear that $\RR\omega_H^{p_0}(p_0)= c_n' \,e_{n+1}$ for some $c_n'>0$.
So we get a contradiction if we choose precisely $c_n=c_n'$.

\vv
It remains to check that $\omega_k^{p_0}$ converges weakly to $\omega_H^{p_0}$. 
{ This follows by rather standard techiques. In fact, similar arguments have appeared in works such as 
\cite{Kenig-Toro-annals} or \cite{Kenig-Toro-crelle}, for example. However, for the convenience of the reader we will
give some details.} So let $g_k(\cdot,p_0)$ and $g_H(\cdot,p_0)$ be the respective Green functions with pole at $p_0$ of the domains $\Omega_k$ and $H$ (assuming, by definition, that they vanish identically in $(\Omega_k)^c$ and $H^c$, respectively). For any $\varphi\in C_c^\infty(\R^{n+1})$, we have
\begin{equation}\label{eqmhar27}
\vphi(p_0) - \int \vphi\,d\omega^{p_0}_k = \int \nabla g_k(x,p_0)\,\nabla\vphi(x)\,dx.
\end{equation}
The analogous identity holds replacing $\omega^{p_0}_k$ by $\omega^{p_0}_H$ and $g_k(x,p_0)$ by $g_H(x,p_0)$. From these identities it follows that, to show that $\omega_k^{p_0}$ converges weakly weakly to $\omega_H^{p_0}$,
it suffices to check that $g_k(\cdot,p_0)$ converges weakly to $g_H(\cdot,p_0)$ in $W^{1,q}_{loc}(\R^{n+1})$ for some $q>1$.

Choosing $q>1$ small enough, standard uniform bounds for the Green function, and the weak compactness of the unit ball of $W^{1,q}(\R^{n+1})$, any subsequence of $\{g_k(\cdot,p_0)\}_k$ has a subsequence $\{g_{k_j}(\cdot,p_0)\}_j$
converging weakly to some $f\in W^{1,q}_{loc}(\R^{n+1})$. It is easy to check that $f$ vanishes in the Sobolev sense in $H^c$ and by \rf{eqmhar27} it also satisfies 
$$\vphi(p_0) = \int \nabla f(x)\,\nabla\vphi(x)\,dx   \quad\mbox{ for all $\vphi\in C_c^\infty(H)$.}$$
From this identity and the analogous one for $g_H(\cdot,p_0)$, we deduce that
$$\int \nabla \big(f(x)-g_H(x,p_0)\big)\,\nabla\vphi(x)\,dx   \quad\mbox{ for all $\vphi\in C_c^\infty(H)$.}$$
So, by Weyl's lemma the function $F=f-g(\cdot,p_0)$ is harmonic in $H$. One easily checks that the function obtained by extending $F|_H$ antisymmetrically with respect to $\partial H$ 
is harmonic in $\R^{n+1}$ and vanishes at $\infty$, and thus it is identically zero\footnote{Instead of using this reflection argument, one can also prove that $f$ 
vanishes continuously at $\partial H$, by the uniform H\"older continuity of $g_k(\cdot,p_0)$ far away from $p_0$. This implies that $F$ is continuous in $\overline H$ and then, recalling that $F$ vanishes at $\infty$, one can appeal to a standard maximum principle to deduce that $F\equiv0$.}. So $f=g_H(\cdot,p_0)$ in $H$. This implies  
that $g_k(\cdot,p_0)$ converges weakly to $g_H(\cdot,p_0)$ in $W^{1,q}_{loc}(\R^{n+1})$, as wished.
\end{proof}

\vv

\begin{lemma}
{ Let $\ve>0$.}
Under the assumptions of Lemma \ref{keylemma} and the above notation, 
\begin{equation}\label{eqcl926}
\int_{G(\Lambda B)} \Theta^n(x,\omega^{p_B}) N_{\Omega^+}(x)\,d\omega^{p_B}(x) = \frac{C_n}{r(B)^n} N_B + \frac{C_B'(\tau_0,\tau_h,\delta,A,\Lambda)}{r(B)^n},
\end{equation}
for some absolute constant $C_n>0$,
with 
$$|C_B'(\tau_0,\tau_h,\delta,A,\Lambda)|\leq\ve$$
if $\tau_0,\delta,r(B)/r_0,\tau_h$ are small enough and $A,\Lambda,\Lambda_0$ big enough.
\end{lemma}

\begin{proof}
This is an easy consequence of Lemmas \ref{lemreverse} and \ref{lemcompactness} and the application of the jump formulas for the Riesz transforms, letting $\ve_0\to0$.
Indeed, momentarily write $\RR_{1,(\ve_0)}^\pm$ instead of $\RR_1^\pm$ to reflect the dependence of $\RR_1^\pm$ on the parameter $\ve_0$ in \rf{eqdefhh1}. Then
by the jump formulas for the Riesz transform, we know that
$$\lim_{\ve_0 \to 0}
\big(\RR^+_{1,(\ve_0)}\omega^{p_B}(x) - \RR_{1,(\ve_0)}^-\omega^{p_B}(x)\big) = C \,\Theta^n(x,\omega^{p_B})N_{\Omega^+}(x)\quad\mbox{ for $\omega$-a.e.\ $x\in G(\Lambda B)$}$$
{(see Sections \ref{sec2cdc} and \ref{secjump}).}
Also notice that for all $x\in  G(\Lambda B)$,
$$\big|\RR_{1,(\ve_0)}^+\omega^{p_B}(x) - \RR_{1,(\ve_0)}^-\omega^{p_B}(x)\big| \lesssim P_{h(x)}\omega^{p_B}(x)\lesssim_{A,\delta} \frac1{r(B)^n},$$
by Lemma \ref{lemcreix2}. Thus by the dominated convergence theorem, 
$$\lim_{\ve_0 \to 0}\int_{ G(\Lambda B)} \big(\RR_{1,(\ve_0)}^+\omega^{p_B} - \RR_{1,(\ve_0)}^-\omega^{p_B}\big)\,d\omega^{p_B} =
C\int_{ G(\Lambda B)} \Theta^n(x,\omega^{p_B})N_{\Omega^+}(x)\,d\omega^{p_B}(x).$$
\end{proof}

\vv

\begin{proof}[\bf Proof of Lemma \ref{keylemma}]
Denote
$$g(x) = \omega(2\Lambda B)\,r(B)^n\,\Theta^n(x,\omega^{p_B}) \frac{d\omega^{p_B}}{d\omega}(x) \,\chi_{G(\Lambda B)}(x).$$
By the change of pole formula, the doubling property of $\omega$ and Harnack's inequality, we have
$$\frac{d\omega^{p_B}}{d\omega}\approx_\Lambda \frac1{\omega(B)} \quad \mbox{ in $2\Lambda B$,}$$
and thus 
$$\Theta^n(x,\omega^{p_B})\approx_\Lambda \frac{\Theta^n(x,\omega)}{\omega(B)}\approx_{\Lambda,A,\delta}\frac1{r(B)^n}\quad \mbox{ in $ G(\Lambda B)$,}$$
by the stopping conditions. Thus
$$g(x)\approx_{\Lambda,A,\delta}1 \quad \quad \mbox{ in $G(\Lambda B)$,}$$
Observe now that the identity \rf{eqcl926} can be rewritten as follows:
$$\frac1{\omega(2\Lambda B)}\int_{2\Lambda B} g(x)\, N_{\Omega^+}(x)\,d\omega = C_n N_B + \wt C(\tau_0,\tau_h,\delta,A,\Lambda),$$
with $|\wt C(\tau_0,\tau_h,\delta,A,\Lambda)|\leq\ve$
if $\tau_0,\tau_h,\delta,r(B)/r_0$ are small enough and $A,\Lambda,\Lambda_0$ big enough. 
So we deduce that
$$\frac1{\omega(2\Lambda B)}\int_{2\Lambda B} g(x)\, \big(N_{\Omega^+}(x) - m_{2\Lambda B,\omega} N_{\Omega^+}\big)\,d\omega = \big(C_n N_B  -  c_{2\Lambda B} m_{2\Lambda B,\omega} N_{\Omega^+}\big)+ \wt C(\tau_0,\tau_h,\delta,A,\Lambda),$$
where
$$c_{2\Lambda B} = \frac1{\omega(2\Lambda B)}\int_{2\Lambda B} g\,d\omega\approx_{A,\delta}1.$$
As a consequence,
\begin{align*}
\big|C_n N_B  -  c_{2\Lambda B} m_{2\Lambda B,\omega} N_{\Omega^+}\big| & \leq \frac1{\omega(2\Lambda B)}\int_{2\Lambda B} g(x)\, \big|N_{\Omega^+}(x) - m_{2\Lambda B,\omega} N_{\Omega^+}\big|\,d\omega +  \wt C(\tau_0,\tau_h,\delta,A,\Lambda)\\
& \lesssim C(A,\delta,\Lambda)\,\|N_{\Omega^+}\|_{*,r(2\Lambda B) ,\omega^+} +   \wt C(\tau_0,\tau_h,\delta,A,\Lambda).
\end{align*}

From the previous estimate and the inequality \rf{equv0} (with $u=C_n N_B$, $v=c_{2\Lambda B} m_{2\Lambda B,\omega} N_{\Omega^+}$), we infer that
\begin{equation}\label{equv1}
\Big|N_B  -   \frac{m_{2\Lambda B,\omega} N_{\Omega^+}}{|m_{2\Lambda B,\omega} N_{\Omega^+}|}\Big|\lesssim C(A,\delta,\Lambda)\,\|N_{\Omega^+}\|_{*,r(2\Lambda B) ,\omega^+} +   \wt C(\tau_0,\tau_h,\delta,A,\Lambda).
\end{equation}
Now we claim that
\begin{equation}\label{equv2}
\Big|m_{B,\omega} N_{\Omega^+}  -   \frac{m_{2\Lambda B,\omega} N_{\Omega^+}}{|m_{2\Lambda B,\omega} N_{\Omega^+}|}\Big| \lesssim_\Lambda \|N_{\Omega^+}\|_{*,r(2\Lambda B) ,\omega^+}.
\end{equation}
{ In fact, by the triangle inequality,
$$\Big|m_{B,\omega} N_{\Omega^+}  -   \frac{m_{2\Lambda B,\omega} N_{\Omega^+}}{|m_{2\Lambda B,\omega} N_{\Omega^+}|}\Big|
 \leq \Big|m_{B,\omega} N_{\Omega^+}  - m_{2\Lambda B,\omega} N_{\Omega^+}\Big| + 
\Big|m_{2\Lambda B,\omega} N_{\Omega^+}  -   \frac{m_{2\Lambda B,\omega} N_{\Omega^+}}{|m_{2\Lambda B,\omega} N_{\Omega^+}|}\Big|.$$
By standard estimates, the first term on the right hand side does not exceed 
$C(\Lambda)\|N_{\Omega^+}\|_{*,r(2\Lambda B) ,\omega^+}$, while the second one equals
$$\Big|m_{2\Lambda B,\omega} \Big(N_{\Omega^+}  -   \frac{m_{2\Lambda B,\omega} N_{\Omega^+}}{|m_{2\Lambda B,\omega} N_{\Omega^+}|}\Big)\Big| \leq \avint_{2\Lambda B}\Big| N_{\Omega^+}(x)- \frac{m_{2\Lambda B,\omega} N_{\Omega^+}}{|m_{2\Lambda B,\omega} N_{\Omega^+}|}\Big| \,d\omega(x).$$
Using again \rf{equv0} with $u= N_{\Omega^+}(x)$ and $v= \frac{m_{2\Lambda B,\omega} N_{\Omega^+}}{|m_{2\Lambda B,\omega} N_{\Omega^+}|}$, we get
$$\avint_{2\Lambda B}\Big| N_{\Omega^+}(x)- \frac{m_{2\Lambda B,\omega} N_{\Omega^+}}{|m_{2\Lambda B,\omega} N_{\Omega^+}|}\Big| \,d\omega(x)\leq
 2 \avint_{2\Lambda B}|N_{\Omega^+}(x)- m_{2\Lambda B,\omega} N_{\Omega^+} |\,d\omega(x)
\lesssim \|N_{\Omega^+}\|_{*,r(2\Lambda B) ,\omega^+},$$
and so our claim follows.
}

By combining \rf{equv1} and \rf{equv2}, and taking into account that $\|N_{\Omega^+}\|_{*,r(2\Lambda B) ,\omega^+}\to0$ as $r(2\Lambda B)\to0$, the lemma follows.
\end{proof}

\vv

\subsection{The end of the proof of (c) $\Rightarrow$ (a) in the case  $\dfrac{d\omega^+}{d\omega^-}\in B_{3/2}(\omega^-)$} \label{sec4.5}

As explained in \cite{Kenig-Toro-duke}, it is enough to show that there exists some $\beta>0$ such that
for all $\ve>0$, for any ball $B$ centered { in} $\partial\Omega^+$ with radius $r(B)$ small enough,
the following reverse H\"older inequality holds:
$$\left(\avint_B \left(\frac{d\omega^-}{d\omega^+}\right)^{1+\beta}d\omega^+\right)^{1/(1+\beta)}\leq 
(1+\ve)\,\frac{\omega^-(B)}{\omega^+(B)}.$$

For some big $\Lambda>1$ to be chosen below, we consider the ball $B_0=\Lambda B$, we define 
$G_0=G(B_0)$ as in Section \ref{secstop}, and we construct the domains $\Omega_b^\pm$ and $\Omega_s^\pm$ as above (whose construction depends on $B_0$ and thus on $B$).
 Observe that by the maximum principle, for all $E\subset G_0$,
$$\omega^-(E) \leq \omega_b^-(E)\qquad \mbox{and}\qquad \omega^+(E) \geq \omega_s^+(E).$$
As a consequence,
$$\frac{d\omega^-}{d\omega^+}\leq \frac{d\omega_b^-}{d\omega_s^+}\quad\mbox{ in $G_0$}.$$
So we have
$$\int_{B\cap G_0} \left(\frac{d\omega^-}{d\omega^+}\right)^{1+\beta}d\omega^+ 
= \int_{B\cap G_0} \left(\frac{d\omega^-}{d\omega^+}\right)^{\beta}d\omega^-
\leq \int_{B\cap G_0} \left(\frac{d\omega_b^-}{d\omega_s^+}\right)^{\beta}d\omega_b^- = 
\int_{B\cap G_0} \left(\frac{d\omega_b^-}{d\omega_s^+}\right)^{1+\beta}d\omega_s^+.
$$
Then we write
\begin{equation}\label{eqspl99}
\avint_B \left(\frac{d\omega^-}{d\omega^+}\right)^{1+\beta}d\omega^+ \leq
\frac1{\omega^+(B)} \int_{B\cap G_0} \left(\frac{d\omega_b^-}{d\omega_s^+}\right)^{1+\beta}d\omega_s^+
+ \frac1{\omega^+(B)} \int_{B\setminus G_0} \left(\frac{d\omega^-}{d\omega^+}\right)^{1+\beta}d\omega^+.
\end{equation}
To estimate the last term, we just apply H\"older's inequality and we take into account that
$\frac{d\omega^-}{d\omega^+}$ satisfies some reverse H\"older's inequality with exponent $1+2\beta$
(because $\omega^-\in A_\infty(\omega^+)$ and we choose $\beta$ in this way).
Then we get, by Lemma \ref{lem4.1},
\begin{align}\label{eqali741}
\frac1{\omega^+(B)} \int_{B\setminus G_0} \left(\frac{d\omega^-}{d\omega^+}\right)^{1+\beta}d\omega^+
& \leq 
\frac1{\omega^+(B)} \left(\int_{B} \left(\frac{d\omega^-}{d\omega^+}\right)^{1+2\beta}d\omega^+\right)^{\frac{1+\beta}{1+2\beta}} \omega^+(B\setminus G_0)^{\frac{\beta}{1+2\beta}}\\
& \leq C
\left(\frac{\omega^+(B\setminus G_0)}{\omega^+(B)}\right)^{\frac{\beta}{1+2\beta}}\,
\left(\frac{\omega^-(B)}{\omega^+(B)}\right)^{1+\beta}\nonumber \\
& \leq C(\Lambda)\,(\ve')^{\frac{\beta}{1+2\beta}}
\left(\frac{\omega^-(B)}{\omega^+(B)}\right)^{1+\beta}.\nonumber
\end{align}

Next we deal with the first term on the right hand side of \rf{eqspl99}. { The following lemma
is a straightforward consequence of the arguments above and the techniques in the work of Kenig and Toro \cite{Kenig-Toro-duke}.}

\begin{lemma}\label{lemguai1}
For any $\ve_4>0$, if $\tau_0$ is small enough, $\Lambda$ is big enough and $r(\Lambda B)$ is small enough,
we have
$$\avint_{B} \left(\frac{d\omega_b^-}{d\omega_s^+}\right)^{2}d\omega_s^+
\leq (1+\ve_4)\,\left(\frac{\omega_b^-(B)}{\omega_s^+(B)}\right)^{2}.$$
\end{lemma}

\begin{proof}
Denote $\sigma= \HH^n|_{\partial \Omega_s^+} =\HH^n|_{\partial \Omega_b^-}$.
Recall that, by Lemma \ref{lemcreix}, $\Omega^+_s$ is a chord arc domain, and
by Lemma \ref{lemnorm}, the unit normal to $\partial \Omega_s^+$ satisfies
$$\|N_{\Omega_s^+}\|_{*,r(\Lambda B),\sigma} \lesssim C(A,\delta)\|N_{\Omega^+}\|_{*,r(10\Lambda B),\omega^+} + \tau_0^{1/2} + \ve_1,$$
with $\ve_1(\tau_0)$ as in Lemma \ref{lemnorm}. So given any $\ve_2>0$, if $r(\Lambda B)$ is small enough, and $\tau_0$ small enough, we will have
$$\|N_{\Omega_s^+}\|_{*,\Lambda r(B),\sigma}\leq \ve_2.$$
Then by a suitable quantitative version of (c) $\Rightarrow$ (a) in Theorem A from Kenig and Toro (more precisely, by \cite[Theorem 4.2]{Kenig-Toro-annals} and \cite[Corollary 5.2]{Kenig-Toro-duke}),
given any $\ve_3>0$, 
assuming $\Lambda$ big enough, $\tau_0$ small enough, and $\ve_2$ small enough, we have that 
$$\bigg\|\log \frac{d\omega_s^+}{d\sigma} \bigg\|_{*,r(\Lambda^{1/2}B),\sigma}\leq \ve_3
\qquad \mbox{and}\qquad \bigg\|\log \frac{d\omega_b^-}{d\sigma} \bigg\|_{*,r(\Lambda^{1/2}B),\sigma}\leq \ve_3.
$$
From this fact and Korey's work \cite{Korey} it follows that, given any $\ve_4>0$, if $\ve_3$ is small enough and $\Lambda$ big enough, then $\frac{d\sigma}{d\omega_s^+}$
and $\frac{d\omega_b^-}{d\sigma}$ satisfy the following reverse H\"older inequalities:
$$\avint_{B} \left(\frac{d\omega_b^-}{d\sigma}\right)^{4}d\sigma
\leq (1+\ve_4)\,\left(\frac{\omega_b^-(B)}{\sigma(B)}\right)^{4}$$
and 
$$\avint_{B} \left(\frac{d\sigma}{d\omega_s^+}\right)^{3}d\omega_s^+
\leq (1+\ve_4)\,\left(\frac{\sigma(B)}{\omega_s^+(B)}\right)^{3}$$
(here we have chosen the exponents $4$ and $3$ because they are useful for our purposes, although Korey's work shows that one can choose arbitrarily large exponents if $\ve_3$ is small enough).
Therefore,
\begin{align*}
\avint_{B} \left(\frac{d\omega_b^-}{d\omega_s^+}\right)^2\,d\omega_s^+ &= \frac1{\omega_s^+(B)}
\int_{B} \left(\frac{d\omega_b^-}{d\sigma}\right)^2 \,\frac{d\sigma}{d\omega_s^+}\,
d\sigma \\
& \leq \frac1{\omega_s^+(B)}\left(\int_{B} \left(\frac{d\omega_b^-}{d\sigma}\right)^4 \,d\sigma\right)^{1/2}
\left(\int_{B} \left(\frac{d\sigma}{d\omega_s^+}\right)^2\,
d\sigma\right)^{1/2}\\
& = \frac{\sigma(B)^{1/2}}{\omega_s^+(B)^{1/2}}
\left(\avint_{B} \left(\frac{d\omega_b^-}{d\sigma}\right)^4 \,d\sigma\right)^{1/2}
\left(\avint_{B} \left(\frac{d\sigma}{d\omega_s^+}\right)^3\,
d\omega_s^+\right)^{1/2}\\
& \leq (1+\ve_4)^{1/2}\,\frac{\sigma(B)^{1/2}}{\omega_s^+(B)^{1/2}}\,\left(\frac{\omega_b^-(B)}{\sigma(B)}\right)^{2}
(1+\ve_4)^{1/2}\,\left(\frac{\sigma(B)}{\omega_s^+(B)}\right)^{3/2}\\
& = (1+\ve_4) \,\left(\frac{\omega_b^-(B)}{\omega_s^+(B)}\right)^2.
\end{align*}
\end{proof}

\vv
From \rf{eqspl99}, \rf{eqali741}, and Lemma \ref{lemguai1} (assuming $\beta\leq1$), we deduce
\begin{align*}
\avint_B \left(\frac{d\omega^-}{d\omega^+}\right)^{1+\beta}d\omega^+ & \leq
\frac{\omega_s^+(B)}{\omega^+(B)} \avint_{B} \left(\frac{d\omega_b^-}{d\omega_s^+}\right)^{1+\beta}d\omega_s^+
+  C(\Lambda)\,(\ve')^{\frac{\beta}{1+2\beta}}
\left(\frac{\omega^-(B)}{\omega^+(B)}\right)^{1+\beta}\\
& \leq
\frac{\omega_s^+(B)}{\omega^+(B)}\left(\avint_{B} \left(\frac{d\omega_b^-}{d\omega_s^+}\right)^2d\omega_s^+\right)^{\frac{1+\beta}2}
+  C(\Lambda)\,(\ve')^{\frac{\beta}{1+2\beta}}
\left(\frac{\omega^-(B)}{\omega^+(B)}\right)^{1+\beta}\\
& \leq (1+\ve_4)^{\frac{1+\beta}2}\,
\frac{\omega_s^+(B)}{\omega^+(B)}
\left(\frac{\omega_b^-(B)}{\omega_s^+(B)}\right)^{1+\beta}
+   C(\Lambda)\,(\ve')^{\frac{\beta}{1+2\beta}}
\left(\frac{\omega^-(B)}{\omega^+(B)}\right)^{1+\beta}\\
& =
\left((1+\ve_4)^{\frac{1+\beta}2}\,\left(\frac{\omega^+(B)}{\omega_s^+(B)}\right)^\beta\left(\frac{\omega_b^-(B)}{\omega^-(B)}\right)^{1+\beta} +
 C(\Lambda)\,(\ve')^{\frac{\beta}{1+2\beta}}\right)
\left(\frac{\omega^-(B)}{\omega^+(B)}\right)^{1+\beta}.
\end{align*}
In view of this estimate, to conclude the proof of Theorem \ref{teo1}, it suffices to prove the following.

\begin{lemma}\label{lemquasigu}
Given $\ve_5>0$, if $r(\Lambda B)$ is small enough, $\Lambda$ big enough, and $\tau_0$ small enough { (possibly depending on $\Lambda$)}, then
\begin{equation}\label{eqffii1}
\omega^+(B)\leq (1+\ve_5)\,\omega_s^+(B)\quad \mbox{ and }\quad
\omega_b^-(B)\leq (1+\ve_5)\,\omega^-(B).
\end{equation}
\end{lemma}

\begin{proof}
We will show first that 
{$\omega_b^-(B)\leq (1+C\ve_5)\,\omega^-(B)$ for some $C$ depending just on $n$, which clearly is equivalent to the second estimate in \rf{eqffii1}.}

Notice that $\Omega^-\subset \Omega_b^-$ and 
\begin{equation}\label{eqdisth1}
\dist_H(\partial\Omega^-,\partial\Omega^-_b)\leq C\Lambda\tau_0r(B),
\end{equation}
{ by the construction of the approximating domains associated with $\Lambda B$.}
 Then, 
  we can write
\begin{align}\label{eqsp953}
\omega_b^-(B) & = \omega_b^{-,p}(B) = \int_{\partial\Omega^-} \omega^{-,x}_b(B)\,d\omega^{-,p}(x)\\
& = \int_{(1+\ve_5)B\cap \partial\Omega^-} \omega^{-,x}_b(B)\,d\omega^{-,p}(x)\nonumber\\
&\quad + \int_{(2B\setminus (1+\ve_5) B)\cap \partial\Omega^-} \omega^{-,x}_b(B)\,d\omega^{-,p}(x) + \sum_{k\ge1} \int_{(2^{k+1}B\setminus 2^kB)\cap \partial\Omega^-} \omega^{-,x}_b(B)\,d\omega^{-,p}(x).
\nonumber
\end{align}
Next we estimate each of the three terms on the right hand side separately. Regarding the first one, from
Theorem 4.1 in \cite{Kenig-Toro-duke}, we deduce that if $\tau_0$ is small enough and $\Lambda$ big enough,
\begin{align*}
\int_{(1+\ve_5)B\cap \partial\Omega^-} \omega^{-,x}_b(B)\,d\omega^{-,p}(x) & \leq 
\omega^-((1+\ve_5)B) \\
&\leq 
{
(1+2\ve_5)}\,
 \omega^-(B) \left(\frac{r((1+\ve_5)B)}{r(B)}\right)^n
\leq \big(1+C\ve_5\big) \omega^-(B).
\end{align*}

Now we turn our attention to the second term on the right hand side of \rf{eqsp953}:
$$ \int_{(2B\setminus (1+\ve_5) B)\cap \partial\Omega^-} \omega^{-,x}_b(B)\,d\omega^{-,p}(x)\leq
\omega^-(2B)\,\sup_{x\in (2B\setminus (1+\ve_5) B)\cap \partial\Omega^-}
\omega^{-,x}_b(B).$$
Assuming that $\Lambda\tau_0\leq\ve_5^{1+1/\eta}$, from \rf{eqdisth1} and the  H\"older continuity of $\omega_b^{-,\cdot}$
in $(2B\setminus (1+\ve_5) B)\cap \Omega_b^-$ (see Lemma 4.1 \cite{Jerison-Kenig})
we obtain
$$\sup_{x\in (2B\setminus (1+\ve_5) B)\cap \partial\Omega^-}
\omega^{-,x}_b(B) \lesssim \left(\frac{\Lambda \tau_0\,r(B)}{\ve_5\,r(B)}\right)^\eta =
(\Lambda \tau_0\ve_5^{-1})^\eta \leq \ve_5 
,$$
for some $\eta>0$ depending only on the NTA character of $\Omega^\pm$. Thus, using also
that $\omega^-(2B)\approx\omega^-(B)$ we derive
$$ \int_{(2B\setminus (1+\ve_5) B)\cap \partial\Omega^-} \omega^{-,x}_b(B)\,d\omega^{-,p}(x)\lesssim 
\ve_5\omega^-(B).$$

Finally we deal with the last term on the right hand side of \rf{eqsp953}. We argue as above:
$$\sum_{k\ge1} \int_{(2^{k+1}B\setminus 2^kB)\cap \partial\Omega^-} \omega^{-,x}_b(B)\,d\omega^{-,p}(x)
\lesssim \sum_{k\ge1} \omega^-(2^{k+1}B) \sup_{x\in (2^{k+1}B\setminus 2^k B)\cap \partial\Omega^-} \omega^{-,x}_b(B).$$
Using the  H\"older continuity of $\omega_b^{-,\cdot}$ and \cite[Lemma 4.4]{Jerison-Kenig}, we get
$$\sup_{x\in(2^{k+1}B\setminus 2^k B)\cap \partial\Omega^-} \omega^{-,x}_b(B)\lesssim
\left(\frac{\Lambda \tau_0\,r(B)}{r(2^kB)}\right)^\eta \omega_b^{-,p_k}(B),$$
where $p_k$ is a corkscrew point for $2^k B$.
Now we take into account that 
$$\omega_b^{-,p_k}(B) \approx \frac{\omega_b^-(B)}{\omega_b^-(2^kB)},$$
and then we deduce that
$$\sum_{k\ge1} \int_{(2^{k+1}B\setminus 2^kB)\cap \partial\Omega^-} \omega^{-,x}_b(B)\,d\omega^{-,p}(x)\lesssim \sum_{k\ge1} \omega^-(2^{k+1}B)
\left(2^{-k}\Lambda \tau_0\right)^\eta  \frac{\omega_b^-(B)}{\omega_b^-(2^kB)}.$$
Next we use that, by Lemma \ref{lemfacil},
$$\omega^{-}(2^{k+1}B)\lesssim \left(\frac{r(2^{k+1}B)}{r(B)}\right)^{n+\gamma}\!\omega^-(B)
\quad \mbox{ and }\quad
\omega_b^{-}(B)\lesssim \left(\frac{r(B)}{r(2^kB)}\right)^{n-\gamma}\!\omega_b^{-}(2^kB),
$$
and then we get
$$\sum_{k\ge1} \int_{2^{k+1}B\setminus 2^kB\cap \partial\Omega^-} \omega^{-,x}_b(B)\,d\omega^{-,p}(x)\lesssim \sum_{k\ge1} \left(2^{-k}\Lambda \tau_0\right)^\eta\,(2^k)^{2\gamma}\,\omega^-(B).$$
Assuming $\tau_0$ small enough, we have $2\gamma\leq \eta/2$, and then
we deduce that the right hand side is at most $C\left(\Lambda \tau_0\right)^\eta$. Gathering the
estimates obtained for the three terms on the right hand side of 
\rf{eqsp953}, the second estimate in \rf{eqffii1} follows.

The proof of the inequality in \rf{eqffii1} is analogous. We just have to replace $\omega_b^-$ by $\omega^+$ and $\omega^-$ by $\omega^+_s$.
\end{proof}

\vv

\subsection{The proof of (c) $\Rightarrow$ (a) in the case  $\dfrac{d\omega^-}{d\omega^+}\in B_{3/2}(\omega^+)$}\label{seclast}

First we claim that the assumptions $N\in\vmo(\omega^+)$ and $\omega^-\in A_\infty(\omega^+)$ imply that
$N\in\vmo(\omega^-)$. The proof of this fact follows in the same way as the proof of Corollary
\ref{coro2} in Section \ref{seccoro2}, just replacing $\sigma$ by $\omega^-$. We leave the details for the reader.

Now it is  easy to see that in all the arguments in Sections \ref{secstop} - \ref{sec4.5} one can interchange the 
roles of $\omega^+$ and $\omega^-$, and then one deduces that 
$\log\frac{d\omega^+}{d\omega^-}
\in\vmo(\omega^-)$. Let us see that this implies that $\log\frac{d\omega^-}{d\omega^+}
\in\vmo(\omega^+)$. Given a ball $B$ centered { in} $\partial\Omega^+$, denote
$$C_B=m_{B,\omega^-}\Big(\log\frac{d\omega^+}{d\omega^-}\Big).$$
Then we have, for any $p>1$,
\begin{align*}
\avint_{B} \left|\log\frac{d\omega^-}{d\omega^+} + C_B\right|\,d\omega^+ &= \frac{\omega^-(B)}{\omega^+(B)} 
\avint_{B} \left|\log\frac{d\omega^+}{d\omega^-} - C_B\right|\,\frac{d\omega^+}{d\omega^-}\,d\omega^-
\\
& \leq
\frac{\omega^-(B)}{\omega^+(B)} 
\left(\avint_{B} \left|\log\frac{d\omega^+}{d\omega^-} - C_B\right|^p\,d\omega^-\right)^{1/p}
\left(\avint_{B} \left(\frac{d\omega^+}{d\omega^-}\right)^{p'}\,d\omega^-\right)^{1/p'}.
\end{align*}
Since $\omega^-$ and $\omega^+$ satisfy an $A_\infty$ relation, $\frac{d\omega^+}{d\omega^-}\in B_{p'}(\omega^-)$ for some $p'\in (1,\infty)$. Together with the John-Nirenberg inequality, this implies
that
$$\avint_{B} \left|\log\frac{d\omega^-}{d\omega^+} + C_B\right|\,d\omega^+\lesssim
\left\|\log\frac{d\omega^+}{d\omega^-}\right\|_{*,Cr(B),\omega^-},$$
which tends to $0$ (uniformly on $B$) as $r(B)\to 0$, because $\log\frac{d\omega^+}{d\omega^-}
\in\vmo(\omega^-)$. Clearly, this implies that $\log\frac{d\omega^-}{d\omega^+}
\in\vmo(\omega^+)$.


\section{Proof of (b) $\Rightarrow$ (d) $\Rightarrow$ (a)}

\subsection{The implication (b) $\Rightarrow$ (d)} This is a direct consequence of Lemma \ref{keylemma}.
In fact, this ensures 
that, under the assumptions in (b),
 $$\lim_{\rho\to0} \sup_{r(B)\leq \rho}\big|N_B - m_{B,\omega^+}N_{\Omega^+}\big| =0,$$
where the supremum is taken over all balls centered in $\partial\Omega^+$ with radius at most $\rho$.
Together with the fact that $N_{\Omega^+}\in \vmo(\omega^+)$, this implies that
\begin{equation}\label{eq*934}
\lim_{\rho\to 0} \sup_{r(B)\leq \rho} \avint_B |N_{\Omega^+} - N_{B}|\,d\omega^+ = 0. 
\end{equation}
Also, the fact that $\frac{d\omega^-}{d\omega^+}\in B_{3/2}(\omega^+)$ tells us, in particular, that
$\omega^-\in A_\infty(\omega^+)$, which is equivalent to the fact that 
$\Omega^+$ and $\Omega^-$ have joint big pieces of chord-arc domains.
\vv

{
\subsection{The implication (d) $\Rightarrow$ (a)} This is proven like (c) $\Rightarrow$ (a)
 in the case $\frac{d\omega^+}{d\omega^-}\in B_{3/2}(\omega^-)$ in Section \ref{secctoa}. Indeed, the reader can check that the only place where
the assumption 
 $\frac{d\omega^+}{d\omega^-}\in B_{3/2}(\omega^-)$ is
used is in the proof of Lemma \ref{keylemma}, which in turn is necessary for the proof of Lemma \ref{lemnorm}.
The remaining arguments for (c) $\Rightarrow$ (a) only
require the condition $\omega^-\in A_\infty(\omega^+)$. 
So the only changes that one has to do for all these arguments to be valid for the implication (d) $\Rightarrow$ (a) consist of eliminating 
Lemma \ref{keylemma} and the whole Section \ref{sec4.4} which is devoted to its proof, and then replacing
Lemma \ref{lemnorm} by the following variant:

\begin{lemma}\label{lemnorm'}
Suppose that $\Omega^+$ is $(\tau_1,r_0)$-Reifenberg flat for some $0<\tau_1\leq\tau_0$,
and that $\Lambda_0 r(B_0)\leq r_0$ for some $\Lambda_0\geq1$. Suppose also that $\omega^+$ and $\omega^-$ are mutually absolutely continuous and denote $\sigma =\HH^n|_{\partial\Omega^+_b}$. If $\tau_1$ is small enough (depending on $\tau_0$ and $\Lambda_0$), then
the oscillation of the inner unit normal to $\partial\Omega^+_b$ 
in any ball $B$ centered { in} $\partial\Omega^+_b$ with radius $0<r(B)\leq r( B_0)$
satisfies
$$ \avint_{B\cap \partial\Omega^+_b} |N_{\Omega^+_b} -m_{B,\sigma} N_{\Omega^+_b}|\,d\sigma\lesssim C(A,\delta)\|N_{\Omega^+}\|_{*,r(10\Lambda_0 B_0),\omega^+} + \tau_0^{1/2} + \ve_1,$$
where
$$\ve_1 \approx \sup_{\substack{x\in\partial\Omega^+\\0<r\leq 5r(B)}}  \avint_{\bar B(x,r)} |N_{\Omega^+} - N_{\bar B(x,r)}|\,d\omega^+$$
and $N_{\bar B(x,r)}$ stands for the inner unit normal to the best 
approximating hyperplane for $\bar B(x,r)\cap \partial\Omega^+$. 
Analogous estimates hold for $\Omega^-_b$ and $\Omega^\pm_s$.
\end{lemma}

Observe that, by the assumption in (d), $\ve_1\to0$ as $r(B)\to0$.

\begin{proof}
The proof is almost the same as the one of Lemma \ref{lemnorm}. The only difference is that now we do not have to appeal to Lemma \ref{keylemma} to prove
\rf{eqclau*92}. Indeed, this estimate just has to be replaced by
\begin{align*}
|N_S - m_{B_S,\omega^+}N_{\Omega^+}| &\leq \avint_{B_S} |N_{\Omega^+} - N_S|\,d\omega^+ \lesssim \avint_{5\bar B_S} |N_{\Omega^+} - N_S|\,d\omega^+\\
&\leq
\sup_{\substack{x\in\partial\Omega^+\\0<r\leq 5r(B)}}  \avint_{\bar B(x,r)} |N_{\Omega^+} - N_{\bar B(x,r)}|\,d\omega^+.
\end{align*}
\end{proof}
}

 \enlargethispage{0.5cm}

\end{document}